\documentclass[10pt,reqno,a4paper]{article} 

\usepackage{anysize}
\marginsize{3.0cm}{3.0cm}{2.2cm}{2.2cm}

\usepackage[english]{babel}
\usepackage[centertags]{amsmath}
\usepackage{amsfonts,amssymb,amsthm}
\usepackage[svgnames]{xcolor}
\usepackage{comment}
\usepackage{hyperref} 
\usepackage{mathrsfs}
\usepackage[sans]{dsfont}
\usepackage{accents}

\let\OLDthebibliography\thebibliography
\renewcommand\thebibliography[1]{
  \OLDthebibliography{#1}
  \setlength{\parskip}{0pt}
  \setlength{\itemsep}{0pt plus 0.3ex} }

\numberwithin{equation}{section}

\theoremstyle{plain}
\newtheorem{theorem}{Theorem}[section]
\newtheorem{proposition}[theorem]{Proposition}
\newtheorem{lemma}[theorem]{Lemma}

\theoremstyle{definition}
\newtheorem{remarkbase}[theorem]{Remark}
\newenvironment{remark}{\pushQED{\qed} \remarkbase}{\popQED\endremarkbase}

\newcommand{\N}{{\mathbb N}}
\newcommand{\R}{{\mathbb R}}
\newcommand{\C}{{\mathbb C}}

\newcommand{\T}{{\mathbb T}}
\renewcommand{\S}{{\mathbb S}}

\newcommand{\mA}{\mathcal{A}}

\newcommand{\mF}{\mathcal{F}}

\newcommand{\mL}{\mathcal{L}}

\newcommand{\mS}{\mathcal{S}}

\renewcommand{\a}{\alpha}
\renewcommand{\b}{\beta}

\renewcommand{\d}{\delta}

\newcommand{\e}{\varepsilon}
\newcommand{\ph}{\varphi}
\newcommand{\lm}{\lambda}
\newcommand{\Lm}{\Lambda}

\newcommand{\Om}{\Omega}

\newcommand{\s}{\sigma}

\renewcommand{\th}{\vartheta}

\newcommand{\la}{\langle}
\newcommand{\ra}{\rangle}
\newcommand{\pa}{\partial}
\renewcommand{\div}{\mathrm{div}\,}
\newcommand{\grad}{\nabla}

\DeclareMathOperator{\divergence}{div}

\DeclareMathOperator*{\esssup}{ess\,sup}

\title{On the Dirichlet-Neumann operator\\ for nearly spherical domains}
 
\author{\normalsize{Pietro Baldi, Vesa Julin, Domenico Angelo La Manna}}
\date{} 

\begin{document}

\maketitle

\noindent
\textbf{Abstract.} 
We consider the Dirichlet-Neumann operator for a nearly spherical domain in $\R^n$, 
and prove sharp analytic and tame estimates in Sobolev class. 
The novelty of this paper concerns technical improvements,  
the most important of which are the independence of the analyticity radius on the high norms 
and the regularity loss of one in the elevation function. 
These properties are expectable but nontrivial to prove. 
The result is obtained by introducing local charts and a convenient class 
of non-isotropic Sobolev spaces of high, possibly fractional tangential regularity 
and integer, limited regularity in the normal direction. 

\tableofcontents

\section{Introduction and main results} 
\label{sec:intro}

The Dirichlet-Neumann operator is a nonlocal operator 
appearing in many free boundary problems coming from physical models,
especially in fluid dynamics, in shape optimization and inverse problems. 
In this paper we consider the Dirichlet-Neumann operator for nearly spherical domains in $\R^n$, 
and we improve some recent results about its analytical dependence on the shape of the domain 
and tame (i.e., linear in the highest order norms) estimates in Sobolev class. 

Let $\Om \subset \R^n$ be an open, bounded set, and assume that its boundary $\pa \Om$  
is the graph of a function $h$ on the unit sphere, 
\begin{equation} \label{def.pa.Om}
\pa \Om = \{ x (1 + h(x)) : x \in \S^{n-1} \}, \quad \ 
\S^{n-1} = \{ x \in \R^n : |x| = 1 \},
\end{equation}
with $1 + h > 0$. We call $h  : \S^{n-1} \to \R$ the height, or elevation, function.  
We also define 
\begin{equation} \label{def.gamma}
\gamma : \S^{n-1} \to \pa \Om, \quad \ 
\gamma(x) = x (1 + h(x)),
\end{equation}
which is a bijection of the unit sphere onto the surface $\pa \Omega$. 
Given a function $\psi : \S^{n-1} \to \R$, we consider the Laplace problem 
with Dirichlet datum $\psi(x)$ at the point $\gamma(x)$ on the boundary, that is,  
\begin{equation} \label{problem.in.Om}
\Delta \Phi = 0 \ \ \text{in } \Om, \quad \ 
\Phi = \psi \circ \gamma^{-1} \ \ \text{on } \pa \Om, \quad \ 
\Phi \in H^1(\Om).
\end{equation}
Under suitable assumptions on $h$ and $\psi$, the solution $\Phi$ is unique. 
We define $G(h)\psi : \S^{n-1} \to \R$ as the normal derivative of $\Phi$ 
at the point $\gamma(x)$ of the boundary $\pa \Om$, i.e., 
\begin{equation} \label{def.G.geom}
G(h)\psi (x) = \la (\grad \Phi)(\gamma(x)) , \nu_{\Om}(\gamma(x)) \ra, 
\quad x \in \S^{n-1},
\end{equation}
where $\nu_\Om(p)$ is the unit normal vector to the boundary $\pa \Om$ 
at the point $p \in \pa \Om$ pointing outwards, 
and $\la \, , \, \ra$ is the usual scalar product of $\R^n$. 
The map $(h, \psi) \to G(h)\psi$ is called the Dirichlet-Neumann operator. 
It is linear in $\psi$, and it depends nonlinearly on $h$. 

The analytic properties of the Dirichlet-Neumann and related operators 
have been studied in the classical works \cite{Calderon, Coifman.Meyer}, 
and then in the water wave literature, 
especially for nearly flat unbounded domains (models for the ocean), 
often including a normalizing factor in front of the right-hand side in \eqref{def.G.geom}: 
see the relevant works \cite{Alazard.Delort, Alazard.Metivier, 
Craig.Nicholls, Craig.Schanz.Sulem, Lannes.book}, 
the recent papers \cite{Alazard.Burq.Zuily1, Berti.Maspero.Ventura, Feola.Montalto.Terracina}, 
the articles \cite{Groves.Nilsson.Pasquali.Wahlen, Pasquali} for nonzero vorticity, 
and the many other references within. 
For non-flat geometries, the literature is less abundant; 
see the recent works \cite{Huang.Karakhanyan.cone} for the cone,
\cite{Huang.Karakhanyan.jets, Yang} for the cylinder,
\cite{LaScala} for the disc, 
and \cite{BJL, Beyer.Gunther, Shao.On.the.Cauchy, Shao.Toolbox} for the sphere.

The present paper deals with nearly spherical domains. 
The main result is Theorem \ref{thm.G.in.intro}, which is 
set in the usual $L^2$-based Sobolev spaces $H^s(\S^{n-1})$ 
of functions defined on the sphere $\S^{n-1}$. 
We recall that the $H^s(\S^{n-1})$ norm of a function 
can be defined by localization, rectification and extension, 
using a partition of unity of the sphere $\S^{n-1}$,  
or by means of its orthogonal decomposition into spherical harmonics of $\S^{n-1}$. 
These two definitions give equivalent norms 
(\cite{Lions.Magenes.volume.1}, Remark 7.6 in Section 7.3). 
For the general theory of Sobolev spaces on open bounded domains and on smooth manifolds 
we refer, e.g., to Lions-Magenes \cite{Lions.Magenes.volume.1}, 
Taylor \cite{Taylor.volume.1}, Triebel \cite{Triebel.1983}. 
We denote, in short, 
\begin{equation} \label{short.norm}
\| f \|_s := \| f \|_{H^s(\S^{n-1})}.
\end{equation}

\begin{theorem} \label{thm.G.in.intro}
Let $s_0, s \in \R$, with $s \geq 1/2$, $s_0 > (n-1)/2$. 
There exist $\delta, C_0, C_s > 0$ such that, 
for $h$ in the set 
\[
U := \{ h \in H^{s_0+1}(\S^{n-1}) : \| h \|_{s_0+1} < \delta \} 
\cap H^{s+1}(\S^{n-1}),
\] 
for $\psi \in H^{s+1}(\S^{n-1})$, 
the function $G(h)\psi$ satisfies 
\begin{align}
\| G(h)\psi \|_{s}
& \leq C_0 \| \psi \|_{s+1} 
+ (\chi_{s > s_0}) C_s (1 + \| h \|_{s+1}) \| \psi \|_{s_0+1},
\label{est.G.in.intro}
\end{align}
where 
$(\chi_{s>s_0}) = 1$ if $s>s_0$, 
$(\chi_{s>s_0}) = 0$ if $s \leq s_0$. 
The constants $\delta, C_0$ are independent of $s$, 
the constant $C_s$ is increasing in $s$, 
and all $\delta, C_0, C_s$ are independent of $h,\psi$. 
Moreover, the map $h \to G(h)$ is analytic from $U$ 
(with the norm $\| \ \|_{s+1} + \| \ \|_{s_0+1}$) 
to the space of bounded linear operators 
$\mL (H^{s+1}(\S^{n-1}) , H^{s}(\S^{n-1}) )$.
\end{theorem}

In \eqref{est.G.in.intro}, the norms $\| h \|_{s+1}$ and $\| \psi \|_{s+1}$ appear linearly, 
even if $G(h)\psi$ depends nonlinearly on $h$:  
for this reason inequality \eqref{est.G.in.intro} is said to be a \emph{tame} estimate. 
We also prove (Theorem \ref{thm.der.G.in.intro}) similar multilinear estimates 
for the first and second derivative of the Dirichlet-Neumann operator with respect to the elevation function $h$. 
Analogous multilinear estimates can also be proved for higher order derivatives.

\begin{theorem} \label{thm.der.G.in.intro}
In the assumptions of Theorem \ref{thm.G.in.intro}, 
the first and second derivative of $G$ with respect to $h$ 
in direction $\eta$, $\eta_1$, $\eta_2$ satisfy 
\begin{align} 
& \| G'(h)[\eta] \psi \|_s
\leq C_0 \| \eta \|_{s_0+1} \| \psi \|_{s+1}
+ (\chi_{s > s_0}) C_s (\| \eta \|_{s+1}
+ \| h \|_{s+1} \| \eta \|_{s_0+1}) \| \psi \|_{s_0+1}
\label{echin.63.lin.in.intro}
\end{align}
and 
\begin{align} 
\| G''(h)[\eta_1, \eta_2] \psi \|_{s} 
& \leq C_0 \| \eta_1 \|_{s_0+1} \| \eta_2 \|_{s_0+1} 
\| \psi \|_{s+1} 
+ (\chi_{s > s_0}) C_s ( \| \eta_1 \|_{s+1} \| \eta_2 \|_{s_0+1}
\notag \\ 
& \quad \ 
+ \| \eta_1 \|_{s_0+1} \| \eta_2 \|_{s+1}
+ \| h \|_{s+1} \| \eta_1 \|_{s_0+1} \| \eta_2 \|_{s_0+1} )
\| \psi \|_{s_0+1},
\label{echin.63.second.der.in.intro}
\end{align}
where $C_s$ is increasing in $s$  
and $C_0$ is independent of $s$. 
\end{theorem}

%\emph{What's new}.
Theorem \ref{thm.G.in.intro} improves a similar result in \cite{BJL}, where it is proved that 
\begin{equation} \label{weaker.est.G}
\| G(h)\psi \|_{H^{k-\frac12}(\S^2)} 
\leq C ( \| \psi \|_{H^{k+\frac12}(\S^2)} 
+ \| h \|_{H^{k+1}(\S^2)} \| \psi \|_{H^{\frac52}(\S^2)})
\end{equation}
for all $h \in H^{k+1}(\S^2)$ with $\| h \|_{H^3(\S^2)} < \delta$, 
all $\psi \in H^{k+\frac12}(\S^2)$, 
where $k \geq 2$ is an integer, and $\delta, C$ depend on $k$
(see Theorem 5.12 of \cite{BJL}).
Thus, Theorem \ref{thm.G.in.intro} improves the result in \cite{BJL} 
under these technical aspects: 
\begin{itemize}

\item 
the ``loss of regularity'' of the elevation function $h$,  
that is, the difference between the Sobolev regularity of $h$ in the right-hand side 
and that of $G(h)\psi$ in the left-hand side, 
is 1 in \eqref{est.G.in.intro} 
and $3/2$ in \eqref{weaker.est.G};

\item 
the constant $\delta$, which is the radius of the ball where analyticity and tame estimates hold, 
is independent of the high regularity parameter $s$ in Theorem \ref{thm.G.in.intro}, 
while $\delta$ depends on $k$ in \cite{BJL};

\item 
the regularity parameter $s$ in \eqref{est.G.in.intro} 
is a real number, while $k$ in \eqref{weaker.est.G} is an integer;

\item 
the constant $C_0$ in front of $\| \psi \|_{s+1}$ in \eqref{est.G.in.intro} is independent of $s$, 
while the constant $C$ in \eqref{weaker.est.G} depends on $k$;

\item 
on the sphere $\S^2$, the regularity threshold for the elevation function $h$ 
is any real $s_0 + 1 > 2$ in Theorem \ref{thm.G.in.intro} with $n=3$, 
and it is $3$ in \cite{BJL};

\item 
Theorem \ref{thm.G.in.intro} holds in any dimension $n \geq 2$, 
while the result in \cite{BJL} is proved for $n=3$. 
\end{itemize}

\emph{Strategy of the proof}. 
The first part of our proof is rather classical: 
we consider a change of variable $\tilde \gamma : B_1 \to \Omega$ 
that transforms \eqref{problem.in.Om} into 
\begin{equation} \label{problem.in.B.in.intro}
\div (P \grad u) = 0 \ \ \text{in } B_1, \quad \ 
u = \psi \ \ \text{on } \S^{n-1}, \quad \ 
u \in H^1(B_1),
\end{equation}
where $B_1$ is the open unit ball of $\R^n$, 
$u = \Phi \circ \tilde \gamma$ is the new unknown,
and $P$ is a $n \times n$ matrix depending on $\tilde \gamma$. 
Note that \eqref{problem.in.B.in.intro} is a problem in a \emph{fixed} (i.e., independent of $h$) domain. 
As is observed in \cite{BJL}, the Dirichlet-Neumann operator \eqref{def.G.geom} can be re-written 
in terms of $h$ and $u$, see formula \eqref{formula.G}.
Then we prove that the solution $u$ of \eqref{problem.in.B.in.intro} depends analytically on $h$, 
by writing it as a power series $u = \sum u_m$ of $h$ and its derivatives,
see \eqref{prob.u.0}, \eqref{prob.u.n}, and estimating each term $u_m$ of the series. 
Such estimates are simple to obtain in low norm, see Lemma \ref{lemma:olive.low} 
and Proposition \ref{prop:est.u.low.norm}. 

The nontrivial part of the proof comes with the estimates in higher norms. 
These estimates are not too difficult to obtain 
if the goal is to prove estimate \eqref{weaker.est.G}, 
and, in fact, the proof of \eqref{weaker.est.G} in \cite{BJL} 
is not long and it largely relies on classical results. 
Here, however, we aim to prove Theorem \ref{thm.G.in.intro}, 
with all the technical improvements listed above, 
and some new ingredient is required. 

The fact that the loss of regularity for the elevation function $h$ is 1, instead of $3/2$, 
is obtained by using a diffeomorphism $\tilde \gamma$ 
which, close to the boundary of $\Om$, 
is a right inverse of the trace operator. 
This can be done in many ways; we find it convenient 
to use the harmonic extension of $h$ to the unit ball, 
and to define $\tilde \gamma$ through it, see \eqref{def.tilde.gamma}. 
As an alternative, one can use localization, rectification, 
and an explicit construction with the Fourier transform, one for each local chart, 
similarly as it is done (globally, i.e., without local charts) 
in the flat case, for example, in \cite{Lannes.book, Berti.Maspero.Ventura}.   

To prove that the radius $\delta$ in Theorem \ref{thm.G.in.intro} 
is independent of $s$ is more demanding. 
Estimate \eqref{est.G.in.intro} is deduced from high regularity boundary estimates 
for the solution $u$ of problem \eqref{problem.in.B.in.intro}, 
which is the sum of the power series $\sum u_m$. 
To estimate each term $u_m$ of the series, a key ingredient is 
a tame estimate with asymmetric constants for the Sobolev norm 
of the product of two functions, 
that is, an estimate of the form 
\begin{equation} \label{prod.est.in.intro}
\| fg \|_s
\leq C_{s_0} \| f \|_s \| g \|_{s_0} + C_s \| f \|_{s_0} \| g \|_s,
\end{equation}
where $C_{s_0}$ depends on $s_0$ and it is independent of $s$, 
while $C_s$ also depends on $s$. 
This inequality makes it possible to estimate 
the high norm of any power series, say $\| \sum f^m \|_s$, 
only assuming a condition of the form $\| f \|_{s_0} < \delta$ 
on the low norm of $f$, with radius $\delta$ independent of $s$.
Estimate \eqref{prod.est.in.intro} is well-known and commonly used 
in the Sobolev spaces $H^s(\R^n)$ and $H^s(\T^n)$,  
where \eqref{prod.est.in.intro} is not difficult to prove 
by using Fourier transform or Fourier series. 
The same proof for the Sobolev spaces $H^s(\S^{n-1})$ on the sphere, 
instead, are hard to adapt to the analogous of Fourier series, 
namely the series expansion of functions in spherical harmonics, 
and this is one of the reasons for which we use 
Sobolev norms defined by localization and rectification, 
see \eqref{def.norm.u}, instead of spectral norms. 
With the norms \eqref{def.norm.u} of $H^s(\S^{n-1})$, 
estimate \eqref{prod.est.in.intro} becomes simple to prove, 
see \eqref{prod.est.unified.Hr.S2}. 

The validity of estimate \eqref{prod.est.in.intro} 
becomes a more serious issue for functions defined in the open ball $B_1$, 
where Sobolev spaces $H^s(B_1)$, $s \in \R$, 
are defined as the spaces of the restrictions to $B_1$ 
of the functions in $H^s(\R^n)$, or by intrinsic representation for integer $s$ 
and then by interpolation (see Taylor's \cite{Taylor.volume.3} 
and Triebel's \cite{Triebel.1983} books).    
The difficulty in proving \eqref{prod.est.in.intro} in $B_1$ 
is related to the use of extension operators, 
say $\mathtt E$, usually constructed by taking coefficients, 
obtained by inverting a Van der Monde matrix, 
that depend on (the integer part of) the high Sobolev index $s$. 
This gives bounds in low and high Sobolev norms of the type 
\[
\| \mathtt{E} f \|_{H^s(\R^n)} \leq C \| f \|_{H^s(B_1)}, \quad \ 
\| \mathtt{E} f \|_{H^{s_0}(\R^n)} \leq C' \| f \|_{H^{s_0}(B_1)}, 
\]
where the constants $C$ and $C'$ both depend on $s$. 
The fact that also $C'$ (in an estimate for the low $s_0$ norms) 
depends on $s$ makes the validity of the asymmetric product estimate \eqref{prod.est.in.intro}
for Sobolev functions in the ball $B_1$ uncertain, or, at least, highly nontrivial to prove.

We overcome this delicate issue by using localizations and introducing 
\emph{non-isotropic} seminorms where the high, and possibly fractional, Sobolev regularity 
is only taken in the tangential directions, 
while in the normal direction only integer, limited $H^1$ or $H^2$ regularity is considered, 
see \eqref{def.seminorm}. 
This idea is very efficient, simple and more or less classical, 
as it can be found in reference books like Triebel \cite[section 4.2.1]{Triebel.1983}
and Taylor \cite[section 5.11]{Taylor.volume.1}, 
and also, in the water wave literature, in Lannes \cite{Lannes.book},  
Coutand-Shkoller \cite{Coutand.Shkoller}, and other authors.  
The use of non-isotropic seminorms for functions in $B_1$ 
allows one to avoid estimating high order derivatives in the normal direction, 
which is not necessary for analyzing the Dirichlet-Neumann operator 
and is what ultimately causes the dependence of the constants by $s$.

\medskip

\emph{Organization of the paper}. 
In Section \ref{sec:new.01} we give definitions and properties about Sobolev spaces 
on $\R^n$, $\S^{n-1}$, $\R^n_+$, its norms and seminorms, both standard and non-isotropic, 
local charts, transition maps, harmonic extensions, Poisson problems. 
The proofs are collected in Section \ref{sec:appendix}, the Appendix. 
In Section \ref{sec:higher} we introduce non-isotropic seminorms of functions in $B_1$
and we use them to study the higher tangential regularity.
After these preliminaries, in Section \ref{sec:DN} we prove Theorems \ref{thm.G.in.intro}, 
\ref{thm.der.G.in.intro}.

\medskip

\textbf{Acknowledgements}. 
This work is supported by GNAMPA, 
by the Academy of Finland grant 347550,
by PRIN 2022E9CF89 \emph{Geometric Evolution Problems and Shape Optimization},
by PRIN 2020XB3EFL \emph{Hamiltonian and dispersive PDEs}, 
and by University of Naples Federico II through FRA 2024 \emph{Geometric Topics in Fluid Dynamics}.

\section{Function spaces}
\label{sec:new.01}

In this section we collect some properties for both standard and non-isotropic versions 
of Sobolev spaces of functions on $\R^n, \R^n_+, \S^d$. 
To define these spaces, we follow the approach of Triebel's book \cite{Triebel.1983}.
The results of this subsection are all more or less classical, 
but it is nontrivial to find a complete proof for all of them in literature. 
The proofs are collected in the Appendix. 

\medskip

\underline{Sobolev spaces on $\R^n$}.
On $\R^n$, for $s \in \R$, we consider the usual Sobolev spaces 
\begin{equation}  \label{def.H.s.classical}
H^s(\R^n) := \{ u \in \mS'(\R^n) : \| u \|_{H^s(\R^n)} < \infty \}, \quad \ 
\| u \|_{H^s(\R^n)}^2 := \int_{\R^n} | \hat u(\xi) |^2 (1 + |\xi|^2)^s \, d\xi,
\end{equation}
where $\mS'(\R^n)$ is the space of tempered distributions on $\R^n$, 
and $\hat u$ is the Fourier transform of $u$ on $\R^n$. 
We also consider the following non-isotropic version of those spaces, 
where the last real variable $x_n$ plays a distinguished role.
For $x \in \R^n$, denote $x' = (x_1, \ldots, x_{n-1})$, so that $x = (x', x_n)$. 
For $s,r \in \R$, we define 
\begin{align}  
H^{s,r}(\R^n) 
& := \{ u \in \mS'(\R^n) : \| u \|_{H^{s,r}(\R^n)} < \infty \}, 
\notag \\ 
\| u \|_{H^{s,r}(\R^n)}^2 
& := \int_{\R^n} | \hat u(\xi) |^2 (1 + |\xi'|^2)^r (1 + |\xi|^2)^s \, d\xi 
= \| \Lm'_r u \|_{H^s(\R^n)}^2,
\label{def.X.s.m}
\end{align}
where $\Lm'_r$ is the Fourier multiplier of symbol $(1 + |\xi'|^2)^{\frac{r}{2}}$.  
These spaces will be mainly used for $s \in \{ 0, 1, 2 \}$ and $r \geq 0$ real. 
They are used by Triebel, see \cite{Triebel.1983}, Definition 4.2.1, page 218; 
see also Lannes \cite{Lannes.book}, Definition 2.11.

\medskip

\underline{Sobolev spaces on open domains}. 
On open sets $\Om \subseteq \R^n$, following \cite{Triebel.1983}, 
Definition 1 in Section 3.2.2, we define 
\begin{align}  
H^s(\Om) 
& := \{ u \in \mS'(\Om) : \| u \|_{ H^s(\Om) } < \infty \},
\notag \\
\| u \|_{ H^s(\Om) } 
& := \inf \, \{ \| v \|_{H^s(\R^n)} : v \in H^s(\R^n), \ v |_{\Om} = u \},
\label{def.H.s.Om} 
\\
H^{s,r}(\Om) 
& := \{ u \in \mS'(\Om) : \| u \|_{H^{s,r}(\Om)} < \infty \}, 
\notag \\
\| u \|_{ H^{s,r}(\Om) } 
& := \inf \, \{ \| v \|_{H^{s,r}(\R^n)} : v \in H^{s,r}(\R^n), \ v |_{\Om} = u \},
\label{def.H.s.r.Om} 
\end{align}
with the convention that the infimum of the empty set is $\infty$. 
In other words, $H^s(\Om)$ is the set of the restrictions to $\Om$ of the elements of $H^s(\R^n)$, 
and similarly for $H^{s,r}(\Om)$.  
We mainly deal with the case when $\Om$ is the open half space 
\begin{equation}  \label{def.R.n.+} 
\R^n_+ := \{ (x', x_n) \in \R^n : x_n > 0 \}. 
\end{equation}

\underline{Sobolev spaces on smooth compact manifolds}.
We recall the construction and definition 
in \cite{Triebel.1983}, Sections 3.2.1 and 3.2.2. 
Let $\Om \subset \R^n$ be a bounded open set with smooth boundary $\Gamma = \pa \Om$.
Fix $N$ open balls $K_1 , \ldots, K_N$ that cover $\Gamma$, 
namely $\Gamma \subset \cup_{j=1}^N K_j$, 
such that $\Gamma \cap K_j$ is nonempty for all $j = 1, \ldots, N$. 
For every ball $K_j$, fix a function $\mathtt f_j$ such that 
\begin{itemize}

\vspace{-4pt}

\item[$(i)$] 
$\mathtt f_j : \R^n \to \R^n$ is a $C^\infty$ diffeomorphism of $\R^n$ onto itself, 

\vspace{-7pt}

\item[$(ii)$] 
$\mathtt f_j$ maps the ball $K_j$ onto a bounded open subset $A_j = \mathtt f_j(K_j)$ of $\R^n$, 

\vspace{-7pt}

\item[$(iii)$]
$\mathtt f_j$ maps the center $p_j$ of the ball $K_j$ into $0$, 

\vspace{-7pt}

\item[$(iv)$] 
$\mathtt f_j$ maps $\Gamma \cap K_j$ onto the $(n-1)$-dimensional open subset 
$\mathtt f_j(\Gamma \cap K_j) = \{ y \in A_j : y_n = 0 \} = A_j \cap \R^n_0$ 
of the hyperplane $\R^n_0 := \{ y \in \R^n : y_n = 0 \}$, 

\vspace{-7pt}

\item[$(v)$] 
$\mathtt f_j$ maps $\Om \cap K_j$ onto the simply connected open subset 
$\mathtt f_j(\Om \cap K_j) = \{ y \in A_j : y_n > 0 \} = A_j \cap \R^n_+$ of the half space $\R^n_+$,

\vspace{-7pt}

\item[$(vi)$] 
the Jacobian matrix $D \mathtt f_j(x)$ is invertible at every point $x$ 
in the closure $\overline{K}_j$ of $K_j$,

\vspace{-7pt}

\item[$(vii)$] 
the Jacobian matrix $D \mathtt f_j(p_j)$ of $\mathtt f_j$ at the center $p_j$ of $K_j$ 
is the identity matrix. 
\end{itemize}

\vspace{-4pt}

Also fix a smooth open set $K_0$ such that $\overline K_0 \subset \Om$ 
and $\overline \Om \subset \cup_{j=0}^N K_j$, 
and fix a resolution of unity $\{ \psi_j \}_{j=0}^N$, 
where $\psi_j : \R^n \to \R$ is a $C^\infty$ function with $\mathrm{supp}(\psi_j) \subset K_j$,
and $\sum_{j=0}^N \psi_j(x) = 1$ for all $x \in \overline \Om$. 
Thus the restriction of $\psi_1, \ldots, \psi_N$ to $\Gamma$ is a resolution of unity for $\Gamma$, 
and, in addition, $\sum_{j=1}^N \psi_j(x) = 1$ for all $x \in \overline \Om \setminus K_0$. 
On $\Gamma$, for $s \in \R$, we define 
\begin{align}
H^s(\Gamma) 
& := \{ u \in \mS'(\Gamma) : (\psi_j u)(\mathtt g_j( \cdot, 0)) \in H^s(\R^{n-1}) \ \ 
\forall j = 1, \ldots, N \},
\notag \\
\| u \|_{H^s(\Gamma)} 
& := \sum_{j=1}^N \| (\psi_j u)(\mathtt g_j(\cdot,0)) \|_{H^s(\R^{n-1})}, 
\quad \mathtt g_j := \mathtt f_j^{-1},
\label{def.H.s.manifold}
\end{align}
where $(\psi_j u)(\mathtt g_j(\cdot,0))$ is defined on $\R^{n-1}$, 
extended to 0 in the set of all $y' \in \R^{n-1}$ such that $\mathtt g_j(y',0) \notin K_j$,
in which case one has $\psi_j (\mathtt g_j(y',0)) = 0$. 

\medskip

\underline{Local charts on $\S^{n-1}$}.
For $\Gamma = \S^{n-1}$, the construction above can be more explicit 
and well adapted to the spherical geometry. 
Let $B_1, \S^{n-1}$ be the open unit ball and unit sphere of $\R^n$. 
For $\delta \in (0, \frac14)$, let 
\[
K_0 = B_{1-\delta}(0), \quad  
K_j = B_{2\delta}(p_j), \quad j = 1, \ldots, N, 
\]
be open balls in $\R^n$,
where the points $p_j \in \S^{n-1}$ are appropriately positioned to obtain the covering property 
$\overline{B}_1 \subset \cup_{j=0}^N K_j$. 
Fix a resolution of unity $\{ \psi_j \}_{j=0}^N$, 
where $\psi_j : \R^n \to \R$ is a $C^\infty$ function with $\mathrm{supp}(\psi_j) \subset K_j$,
and $\sum_{j=0}^N \psi_j(x) = 1$ for all $x \in \overline B_1$. 
Thus,
\begin{equation} \label{res.unity.in.the.annulus}
\sum_{j=1}^N \psi_j(x) = 1 \quad \forall x \in \R^n, \  1-\delta \leq |x| \leq 1.
\end{equation}
Let $\mathtt R_j$ be a $n \times n$ rotation matrix mapping the center $p_j$ of the ball $K_j$ 
into the ``South Pole'' $p_0 := (0, . . . , 0, -1)$ of $\S^{n-1}$.  
Let $\mathtt f_j(x) := \mathtt f(\mathtt R_j x)$ be the composition of the rotation $\mathtt R_j$
with the function 
\begin{equation} \label{explicit.mathtt.f} 
\mathtt f : \mathbb H_0 \to \mathbb H_1, \quad \ 
\mathtt f(x) := \Big( - \frac{x'}{x_n} , 1 - |x| \Big),
\end{equation}
where $x' = (x_1, \ldots, x_{n-1})$, and $\mathbb H_0$, $\mathbb H_1$ 
are the open half spaces 
$\mathbb H_0 := \{ x \in \R^n : x_n < 0 \}$, 
$\mathbb H_1 := \{ y \in \R^n : y_n < 1 \}$.
The function $\mathtt f$ is a well defined, smooth bijection of $\mathbb H_0$ onto $\mathbb H_1$.
Its inverse is the function 
\begin{equation} \label{explicit.mathtt.f.inv}
\mathtt g : \mathbb H_1 \to \mathbb H_0, \quad \ 
\mathtt g(y) := (1 - y_n) (1 + |y'|^2)^{-\frac12} (y', -1),
\end{equation}
because, if $\mathtt f(x) = y$, then $|x| = 1 - y_n$, $x' = - x_n y'$, 
whence $(1-y_n)^2 = |x|^2 = |x'|^2 + x_n^2 = x_n^2(1+|y'|^2)$; 
for $x_n < 0$, one obtains $x = \mathtt g(y)$. 
The map $\mathtt{g}$ is a well defined, smooth bijection of $\mathbb H_1$ onto $\mathbb H_0$. 
Moreover $\mathtt f( \mathbb H_0 \cap \S^{n-1} ) = \R^n_0$ 
and $\mathtt g(\R^n_0) = \mathbb H_0 \cap \S^{n-1}$, where $\R^n_0$ is the hyperplane 
$\{ y \in \R^n : y_n = 0 \}$. 
Let $\mathtt g_j(y) := \mathtt R_j^{-1} \mathtt g(y)$. 
Then 
\begin{equation} \label{inv.large.domain}
\mathtt f_j(\mathtt g_j(y)) = y \quad \forall y \in \mathbb H_1, 
\qquad \ 
\mathtt g_j(\mathtt f_j(x)) = x \quad \forall x \in \mathtt R_j^{-1} \mathbb H_0,
\end{equation}
where $\mathtt R_j^{-1} \mathbb H_0 := \{ x \in \R^n : \mathtt R_j x \in \mathbb H_0 \}
= \{ \mathtt R_j^{-1} z : z \in \mathbb H_0 \}$,
and 
\begin{equation} \label{ttf.ttg.sets}
\mathtt f_j( (\mathtt R_j^{-1} \mathbb H_0) \cap \S^{n-1} ) = \R^n_0, \quad \ 
\mathtt g_j( \R^n_0 ) = (\mathtt R_j^{-1} \mathbb H_0) \cap \S^{n-1}.
\end{equation}

For $\delta < \frac12$, the ball $B_{2\delta}(p_0)$ is contained in the half space $\mathbb H_0$, 
therefore $\mathtt f_j$ is well defined in the ball $K_j$, 
with bijective image 
\[
\mathtt f_j(K_j) = \mathtt f(B_{2\delta}(p_0)) =: A_0.
\] 
One has $\mathtt f_j(p_j) = \mathtt f(p_0) = 0$ and 
\begin{equation} \label{ttf.j.K.j.cap.S.d}
\mathtt f_j(K_j \cap \S^{n-1}) 
= \mathtt f( B_{2\delta}(p_0) \cap \S^{n-1} )
= A_0 \cap \R^n_0
= D \times \{ 0 \}, 
\end{equation}
where 
\begin{equation} \label{def.disc.D}
D := \{ y' \in \R^{n-1} : (y',0) \in A_0 \}
= \{ y' \in \R^{n-1} : |y'| < r_{\delta} \} \subset \R^{n-1}
\end{equation}
is the ball of $\R^{n-1}$ of radius $r_{\delta} = 2 \delta (1 - \delta^2)^{\frac12} (1 - 2 \delta^2)^{-1}$.
Also, 
\begin{equation} \label{local.prop.ttg.j}
\mathtt g_j(A_0) = K_j, \quad \ 
\mathtt g_j(D \times \{0\}) = K_j \cap \S^{n-1}.
\end{equation}
Moreover, we define 
\begin{equation} 
E_j := \mathrm{supp}(\psi_j) \cap \S^{n-1} \subset K_j \cap \S^{n-1},
\quad 
F_j := \{ y' \in \R^d : \mathtt g_j(y',0) \in \mathrm{supp}(\psi_j) \} \subset D,
\label{def.compact.F.j}
\end{equation}
and we note that 
\begin{equation} \label{local.prop.ttg.j.E.j.F.j}
\mathtt f_j(E_j) = F_j  \times \{ 0 \}, \quad \ 
\mathtt g_j(F_j \times \{0\}) = E_j.
\end{equation}

\medskip

\underline{Sobolev spaces on $\S^{n-1}$}. 
Given a function $u : \S^{n-1} \to \C$, its Sobolev norm \eqref{def.H.s.manifold}
is defined by localization and rectification. 
\emph{Localization}: one has $u = \sum_{j=1}^N u \psi_j$ on $\S^{n-1}$, 
and, for each $j$, the product $u \psi_j$ vanishes in $\S^{n-1} \setminus \mathrm{supp}(\psi_j)$; 
recall that $\mathrm{supp}(\psi_j)$ is a compact subset of the open ball $K_j$. 
\emph{Rectification}: by \eqref{ttf.ttg.sets}, 
$\mathtt g_j(\cdot,0)$ maps $\R^{n-1}$ into $\S^{n-1}$. 
Hence both $u(\mathtt g_j(\cdot,0))$ and $\psi_j(\mathtt g_j(\cdot,0))$ are defined in $\R^{n-1}$, and 
\begin{equation} \label{psi.j.is.zero}
\psi_j(\mathtt g_j(y',0)) = 0 \quad \forall y' \in \R^{n-1} \setminus F_j,
\end{equation}
where $F_j$ is the compact set defined in \eqref{def.compact.F.j}.
Then the Sobolev norm is defined by \eqref{def.H.s.manifold}, that is, 
\begin{equation} \label{def.norm.u}
\| u \|_{H^s(\S^{n-1})} := \sum_{j=1}^N \| (u \psi_j) \circ \mathtt g_j(\cdot, 0) \|_{H^s(\R^{n-1})}. 
\end{equation}
We note that, unlike for a more general manifold, 
in \eqref{def.norm.u} the extension step is not required, 
because $(u \psi_j) \circ \mathtt g_j(\cdot, 0)$ is already defined in $\R^{n-1}$, 
with compact support contained in $F_j$. 
This is due to the geometry of $\S^{n-1}$ (invariance by rotations) 
and to our choice of local charts 
($\mathtt f$ gives the stereographical projection of the half sphere 
$\S^{n-1} \cap \mathbb H_0$ onto the hyperplane $\R^n_0 = \R^{n-1} \times \{0\}$).

\medskip

\underline{Transition maps}. 
If two balls $K_j, K_\ell$ have nonempty intersection, 
then the transition map $\mathtt f_\ell \circ \mathtt g_j$ 
can be explicitly calculated using \eqref{explicit.mathtt.f}, \eqref{explicit.mathtt.f.inv}
and the basic property $|\mathtt R_\ell \mathtt R_j^{-1} x| = |x|$ of rotation matrices. 
One finds that 
\begin{equation}  \label{formula.transition.map}
\mathtt f_\ell \circ \mathtt g_j(y) = (\mathtt T_{\ell j}(y') , y_n), 
\quad \ 
\mathtt T_{\ell j}(y') 
= - \frac{[\mathtt R_\ell \mathtt R_j^{-1} (y', -1)]'}{[\mathtt R_\ell \mathtt R_j^{-1} (y', -1)]_n}
\end{equation}
for all $y \in \mathtt f_j(K_j \cap K_\ell)$.
Formula \eqref{formula.transition.map} shows the remarkable property that 
the transition map $\mathtt f_\ell \circ \mathtt g_j$ does not move the last variable $y_n$, 
and it acts on $y'$ as a function independent of $y_n$. 

\medskip

\underline{Extension of local charts}.
To estimate the $H^{s,r}(\R^n_+)$ norms of products and compositions of functions,
it is convenient to extend $\mathtt f$, $\mathtt g$ and $\mathtt f_k \circ \mathtt g_j$ 
outside their original domains. 
The first order Taylor expansion of $\mathtt f(x)$ around $p_0$ is $x-p_0$; 
therefore we define 
\begin{equation} \label{def.tilde.mathtt.f}
\tilde{\mathtt f}(x) := 
x - p_0 + \ph_\delta(x-p_0) ( \mathtt f(x) - x + p_0 )
\quad \ \forall x \in \R^n,
\end{equation}
where $\ph_\delta(x-p_0) := \ph ( \frac{ |x-p_0| }{ 2\delta } )$  
for some cut-off function $\ph \in C^\infty(\R)$ such that $\ph(t) = 1$ for $t \leq 1$, 
$\ph(t) = 0$ for $t \geq 2$, and $0 \leq \ph \leq 1$.  
Taking the half radius $\delta$ of the balls $K_j$ sufficiently small, 
the function $\tilde{\mathtt f}$ is a diffeomorphism of $\R^n$ 
because $|D \tilde{\mathtt f}(x) - I| \leq 1/2$ for all $x \in \R^n$. 
Also, $\tilde{\mathtt f} = \mathtt f$ in $B_{2\delta}(p_0)$. 
We define $\tilde{\mathtt g}$ as the inverse of $\tilde{\mathtt f}$ in \eqref{def.tilde.mathtt.f}.  
We also define $\tilde{\mathtt f}_j := \tilde{\mathtt f} \circ \mathtt R_j$ 
and $\tilde{\mathtt g}_j$ as its inverse function.

\medskip

\underline{Extension of transition maps}.
Concerning the transition map, of course 
$\tilde{\mathtt f}_\ell \circ \tilde{\mathtt g}_j$ 
is an extension of \eqref{formula.transition.map}; 
however, it is convenient to introduce another, more explicit extension, 
preserving the property of not moving $y_n$ and acting on $y'$ independently on $y_n$. 
The affine function $y' \mapsto \mathtt R_\ell \mathtt R_j^{-1} (y', -1) 
= \mathtt R_\ell \mathtt R_j^{-1} (0, -1) + \mathtt R_\ell \mathtt R_j^{-1} (y', 0)$ 
is defined for all $y' \in \R^{n-1}$, and its value at $y'=0$ is 
\begin{equation} \label{last.column}
\mathtt R_\ell \mathtt R_j^{-1} (0, -1) 
= \mathtt R_\ell \mathtt R_j^{-1} p_0
= \mathtt R_\ell p_j 
= \mathtt R_\ell p_\ell + \mathtt R_\ell (p_j - p_\ell)
= p_0 + \mathtt R_\ell (p_j - p_\ell),
\end{equation}
with $|\mathtt R_\ell (p_j - p_\ell)| = |p_j - p_\ell| < 4\delta$ 
because $p_j, p_\ell$ are the centers of the two intersecting balls $K_j, K_\ell$ of radius $2\delta$. 
We write the rotation matrix $\mathtt R_\ell \mathtt R_j^{-1}$ as 
$\big( \! \begin{smallmatrix} \mathtt A & \mathtt b \\ \mathtt c & \mathtt d \end{smallmatrix} \! \big)$,
where $\mathtt A$ is an $(n-1) \times (n-1)$ matrix, 
$\mathtt b$ is a column with $n-1$ components, 
$\mathtt c$ is a row with $n-1$ components,
and $\mathtt d$ is a scalar.  
Thus 
\begin{alignat}{2}
\mathtt A y' & = [ \mathtt R_\ell \mathtt R_j^{-1} (y', 0) ]', 
\qquad & 
\mathtt b & = [ \mathtt R_\ell \mathtt R_j^{-1} (0,1) ]' 
= - [ \mathtt R_\ell (p_j - p_\ell) ]',
\notag \\
\mathtt c y' & = [ \mathtt R_\ell \mathtt R_j^{-1} (y', 0) ]_n, 
\qquad & 
\mathtt d & = [ \mathtt R_\ell \mathtt R_j^{-1} (0,1) ]_n 
= 1 - [ \mathtt R_\ell (p_j - p_\ell) ]_n
\label{rotat.03}
\end{alignat}
for all $y' \in \R^{n-1}$, where we have used \eqref{last.column} 
and the fact that $p_0'=0, (p_0)_n = -1$. 
Hence $\mathtt T_{\ell j}(y')$ in 
\eqref{formula.transition.map} becomes 
\begin{equation} \label{rotat.06}
\mathtt T_{\ell j}(y') = \frac{- \mathtt b + \mathtt A y'}{\mathtt d - \mathtt c y'}.
\end{equation}

\begin{lemma} \label{lemma:T.ell.j}
Let $\delta < 1/8$. Then the scalar $\mathtt d$ and 
the matrix $\mathtt A \mathtt d - \mathtt b \mathtt c$ are invertible, 
one has 
\begin{equation} \label{rotat.12}
|\mathtt d - 1| < 4 \delta, \quad \ 
|\mathtt c| = |\mathtt b| < 4 \delta, \quad \ 
| (\mathtt A \mathtt d - \mathtt b \mathtt c)^{-1} y'| \leq \lm |y'| 
\quad \forall y' \in \R^{n-1}, \quad \ 
\lm := \mathtt d^{-1},
\end{equation}
and, for all $|y'| \leq 1$, one has the converging Taylor expansion
\begin{equation}  \label{rotat.08}
\frac{- \mathtt b + \mathtt A y'}{\mathtt d - \mathtt c y'} 
= - \lm \mathtt b + \lm^2 (\mathtt A \mathtt d - \mathtt b \mathtt c ) y' (1 + \mathtt S(y')), \quad \ 
\mathtt S(y') := \sum_{m=1}^\infty (\lm \mathtt c y')^m.
\end{equation}
\end{lemma}

\begin{proof}
By \eqref{rotat.03}, the vector $(\mathtt b, \mathtt d - 1)$ 
is the opposite of $\mathtt R_\ell (p_j - p_\ell)$, 
and its norm is $|(\mathtt b, \mathtt d - 1)|$ 
$= |\mathtt R_\ell (p_j - p_\ell)|$ 
$ = |p_j - p_\ell|$ 
$< 4 \delta$.
This gives the bound for $|\mathtt b|, |\mathtt d - 1|$ in \eqref{rotat.12}.
Since $\mathtt R_\ell \mathtt R_j^{-1}$ is a rotation matrix, 
all its columns and rows have unitary norm. 
In particular, taking the $n$th row and $n$th column, one has 
$|\mathtt b|^2 + \mathtt d^2 = 1$ and 
$|\mathtt c|^2 + \mathtt d^2 = 1$, 
whence $|\mathtt c| = |\mathtt b|$. 

To study the matrix $\mathtt A \mathtt d - \mathtt b \mathtt c$,  
we observe that 
\begin{equation} \label{rotat.10}
\begin{pmatrix} \mathtt A & \mathtt b \\ \mathtt c & \mathtt d \end{pmatrix}
\begin{pmatrix} \mathtt d I' & 0 \\ - \mathtt c & 1 \end{pmatrix}
= 
\begin{pmatrix} \mathtt A \mathtt d - \mathtt b \mathtt c & \mathtt b \\ 0 & \mathtt d \end{pmatrix}\!,
\end{equation}
where $I'$ is the $(n-1) \times (n-1)$ identity matrix. 
The first matrix in \eqref{rotat.10} is the rotation $\mathtt R_\ell \mathtt R_j^{-1}$, 
whose determinant is 1, while the determinant of the second matrix is $\mathtt d^{n-1}$.  
Hence the determinant of the third matrix in \eqref{rotat.10} is $\mathtt d^{n-1}$, which is nonzero,  
and this is also the product of $\mathtt d$ with the determinant 
of $\mathtt A \mathtt d - \mathtt b \mathtt c$. 
This proves that the matrix $\mathtt A \mathtt d - \mathtt b \mathtt c$ is invertible. 

Now let $z' \in \R^{n-1}$. 
We consider the product of \eqref{rotat.10} with the vector $(z',0)$,  
and we calculate its norm. 
The matrix in the RHS of \eqref{rotat.10} 
times the vector $(z',0)$ gives $((\mathtt A \mathtt d - \mathtt b \mathtt c)z', 0)$,   
whose norm is $|(\mathtt A \mathtt d - \mathtt b \mathtt c)z'|$.  
The matrix product in the LHS of \eqref{rotat.10} 
times the same vector $(z',0)$ gives 
$\mathtt R_\ell \mathtt R_j^{-1} (\mathtt d z', - \mathtt c z')$, 
whose norm is equal to the norm of $(\mathtt d z', - \mathtt c z')$,
which is $\geq \mathtt d |z'|$. 
This proves that $|(\mathtt A \mathtt d - \mathtt b \mathtt c)z'| \geq \mathtt d |z'|$ 
for all $z' \in \R^{n-1}$. 
Since $\mathtt A \mathtt d - \mathtt b \mathtt c$ is invertible, 
taking $z' = (\mathtt A \mathtt d - \mathtt b \mathtt c)^{-1} y'$ gives 
the third inequality of \eqref{rotat.12}. 

Formula \eqref{rotat.08} is obtained by writing $(\mathtt d - \mathtt c y')^{-1} 
= \lm (1 - \lm \mathtt c y')^{-1}$ as a power series, which converges for $\lm |\mathtt c y'| < 1$. 
One has $\lm \leq (1 - 4 \delta)^{-1} < 2$ and $|\mathtt c| < 4 \delta$, 
whence $\lm |c y'| < 8 \delta < 1$ for all $|y'| \leq 1$.  
\end{proof}

We define 
\begin{equation}  \label{rotat.09}
\tilde{\mathtt T}_{\ell j}(y') 
:= - \lm \mathtt b 
+ \lm^2 (\mathtt A \mathtt d - \mathtt b \mathtt c ) y' \big( 1 + \ph_\d(y') S(y') \big), 
\quad \ 
\ph_\d(y') := \ph \Big( \frac{|y'|}{4 \delta} \Big),
\end{equation}
where $\ph \in C^\infty(\R)$ is a smooth cut-off function such that 
$\ph(t) = 1$ for $t \leq 1$, 
$\ph(t) = 0$ for $t \geq 2$, 
and $0 \leq \ph \leq 1$. 
Hence $\tilde{\mathtt T}_{\ell j} = \mathtt T_{\ell j}$ where $\ph_\delta = 1$, 
i.e., in the ball $|y'| \leq 4 \delta$.  
Since $\delta < 1/4$, the set $A_0 = \mathtt f(B_{2\delta}(p_0))$ 
is contained in the ball $|y| < 4\delta$, and  
\begin{equation} \label{extension.in.A0}
(\tilde{\mathtt T}_{\ell j}(y'), y_n) 
= (\mathtt T_{\ell j}(y'), y_n) 
= \mathtt f_\ell (\mathtt g_j(y)) 
\quad \ \forall y = (y', y_n) \in A_0.
\end{equation}

\begin{lemma}  \label{lemma:rotat.tilde}
There exists a universal constant $\delta_0 \in (0, 1/8)$ such that, 
for $0 < \delta \leq \delta_0$, if $K_j \cap K_\ell \neq \emptyset$, then 
the map $\tilde{\mathtt T}_{\ell j}$ in \eqref{rotat.09}
is a $C^\infty$ diffeomorphism of $\R^{n-1}$ 
that coincides with $\mathtt T_{\ell j}$ in \eqref{formula.transition.map} 
in the ball $|y'| \leq 4 \delta$,
and coincides in $|y'| \geq 8 \delta$ with the affine map 
$y' \mapsto - \lm \mathtt b + \lm^2 (\mathtt A \mathtt d - \mathtt b \mathtt c ) y'$.
\end{lemma}

\begin{proof}
Let us prove that $\tilde{\mathtt T}_{\ell j}$ is a bijection of $\R^{n-1}$. 
Given $x' \in \R^{n-1}$, one has $\tilde{\mathtt T}_{\ell j}(y') = x'$ 
if and only if $y' = \Phi(y')$, where 
$\Phi(y') := \lm^{-2} (\mathtt A \mathtt d - \mathtt b \mathtt c)^{-1} (x' + \lm \mathtt b) 
- y' \ph_\delta(y') \mathtt S(y')$. 
The Jacobian matrix $D \Phi(y')$ vanishes for $|y'| \geq 8 \delta$, 
and, for $\delta < 1/16$, one has
\[
| D \Phi(y') z' | \leq C \delta |z'| \quad \ \forall y', z' \in \R^{n-1}, 
\]
where the constant $C$ depends only on $\| \ph' \|_{L^\infty(\R)}$; 
we can assume that, say, $\| \ph' \|_{L^\infty(\R)} \leq 10$, so that $C$ is a universal constant.
Hence, for $C \delta \leq 1/2$, 
$\Phi$ is a contraction mapping on $\R^{n-1}$, and it has a unique fixed point. 
This proves that $\tilde{\mathtt T}_{\ell j}$ is a bijection of $\R^{n-1}$. 
The $C^\infty$ regularity of $\tilde{\mathtt T}_{\ell j}$ 
and of its inverse map $\tilde{\mathtt T}_{\ell j}^{-1}$ 
follows from a standard implicit function argument. 
\end{proof}

\medskip

\underline{Properties of Sobolev spaces on $\R^n_+$}. 
The next lemma collects some useful properties of the Sobolev spaces, both standard and non-isotropic, 
on $\R^n$ and $\R^n_+$.

\begin{lemma} \label{lemma:est.tools.R.n.+} 
On $\R^n$ and $\R^n_+$ one has the following properties.

$(i)$ \emph{Density}. 
For $s,r \in \R$, the Schwartz class $\mS(\R^n)$ is dense in $H^{s,r}(\R^n)$, 
and the set $\mS(\R^n_+) := \{ \ph |_{\R^n_+} : \ph \in \mS(\R^n) \}$ 
of the restrictions to $\R^n_+$ of the Schwartz functions of $\R^n$ 
is dense in $H^{s,r}(\R^n_+)$. 

$(ii)$ \emph{Embedding}. 
Let $s,r \in \R$, with $s > 1/2$, $s+r > n/2$. 
Then $H^{s,r}(\R^n) \subset L^\infty(\R^n)$ and 
$H^{s,r}(\R^n_+) \subset L^\infty(\R^n_+)$, with 
\begin{equation} \label{embedding.ineq}
\| v \|_{L^\infty(\R^n)} \leq C_{s,r} \| v \|_{H^{s,r}(\R^n)}, 
\quad 
\| u \|_{L^\infty(\R^n_+)} \leq C_{s,r} \| u \|_{H^{s,r}(\R^n_+)},
\end{equation}
for all $v \in H^{s,r}(\R^n)$, 
all $u \in H^{s,r}(\R^n_+)$, 
where the constant $C_{s,r}$ depends on $s,r+s$. 
Moreover, if $v \in H^{s,r}(\R^n)$, then $v$ is uniformly continuous in $\R^n$  
and $v(x) \to 0$ as $|x| \to \infty$.  
If $u \in H^{s,r}(\R^n_+)$, then $u$ is uniformly continuous in $\R^n_+$, 
and $u(x) \to 0$ as $|x| \to \infty$ in $\R^n_+$.  

$(iii)$ \emph{Partial derivatives}. 
The partial derivatives (in the sense of distributions) satisfy
\begin{align} 
\| \pa_{x_k} u \|_{H^{s,r}(\R^n_+)} 
& \leq \| u \|_{H^{s,r+1}(\R^n_+)} 
\leq \| u \|_{H^{s+1,r}(\R^n_+)}, 
\quad k=1, \ldots, n-1,
\label{pa.k.est}
\\
\| \pa_{x_n} u \|_{H^{s,r}(\R^n_+)} 
& \leq \| u \|_{H^{s+1,r}(\R^n_+)},
\label{pa.n.est}
\end{align} 
for all $u \in H^{s+1,r}(\R^n_+)$, for all $s,r \in \R$. 

$(iv)$ \emph{Trace}. 
Let $s,r \in \R$, with $s > 1/2$. 
On $\R^n$, the trace operator 
\begin{equation} \label{def.trace.H.s.r.R.n}
T : H^{s,r}(\R^n) \to H^{s+r-\frac12}(\R^{n-1}), 
\quad 
u \mapsto T u = u(\cdot, 0)
\end{equation}
is a well defined, bounded, linear operator, with 
\begin{equation} \label{trace.est.R.n}
\| u(\cdot , 0) \|_{ H^{s+r-\frac12}(\R^{n-1}) } 
\leq C_s \| u \|_{ H^{s,r}(\R^n) }
\end{equation}
for all $u \in H^{s,r}(\R^n)$ 
(first defined on $\mS(\R^n)$, then extended to $H^{s,r}(\R^n)$ by density),
where the constant $C_s$ depends on $s$ and it is independent of $r$.  

On $\R^n_+$, for $s>\frac12$, $r \geq 0$, 
$u \in H^{s,r}(\R^n_+)$, define $Tu$ as the trace $Tv$ of any $v \in H^{s,r}(\R^n)$ 
such that $v|_{\R^n_+} = u$.  
This definition is well posed (i.e., it is independent of the considered extension $v$),   
it coincides with $u(\cdot, 0)$ when $u \in \mS(\R^n_+)$, and 
\begin{equation} \label{trace.est}
\| u(\cdot , 0) \|_{ H^{s+r-\frac12}(\R^{n-1}) } 
\leq C_s \| u \|_{ H^{s,r}(\R^n_+) }
\end{equation}
for all $u \in H^{s,r}(\R^n_+)$, where $C_s$ is the same constant as in \eqref{trace.est.R.n}.

$(vi)$ \emph{Extension from $\R^n_+$ to $\R^n$}. 
Given $u \in \mS(\R^n_+)$, let $f$ be its extension by reflection 
\begin{equation} \label{def.reflex}
f(x) = f(x', x_n) = u(x',|x_n|), \quad x = (x',x_n) \in \R^n.
\end{equation}
For all real $r \geq 0$, for $s = 0,1$, 
the linear map $u \to f$ extends to a bounded linear operator 
\begin{equation} \label{reflex.est}
R : H^{s,r}(\R^n_+) \to H^{s,r}(\R^n), \quad 
\| Ru \|_{H^{s,r}(\R^n)} \leq C \| u \|_{H^{s,r}(\R^n_+)} 
\quad \forall u \in H^{s,r}(\R^n_+),
\end{equation}
where $C = \sqrt{2}$, 
with $Ru = u$ in $\R^n_+$. 

$(vii)$ For every real $r \geq 0$, one has 
\begin{align} 
\| u \|_{H^{0,r}(\R^n_+)}^2 
= \int_0^\infty \| u(\cdot, x_n) \|_{H^r(\R^{n-1})}^2 \, dx_n
\quad & \forall u \in H^{0,r}(\R^n_+),
\label{equiv.norm.H.0.r}
\\
\frac12 \| u \|_{H^{1,r}(\R^n_+)}^2 
\leq \| u \|_{H^{0,r+1}(\R^n_+)}^2 + a \| \pa_{x_n} u \|_{H^{0,r}(\R^n_+)}^2 
\leq \| u \|_{H^{1,r}(\R^n_+)}^2 
\quad & \forall u \in H^{1,r}(\R^n_+),
\label{equiv.norm.H.1.r}
\end{align}
where $a = (2\pi)^{-2}$. 
In particular, for $r=0$, 
\begin{align} 
\| u \|_{H^{0,0}(\R^n_+)} 
= \| u \|_{L^2(\R^n_+)}, 
\quad
c_0 \| u \|_{H^{1,0}(\R^n_+)}^2 
\leq \sum_{|\a| \leq 1} \| \pa_x^\a u \|_{L^2(\R^n_+)}^2 
\leq c_1 \| u \|_{H^{1,0}(\R^n_+)}^2, 
\label{equiv.norm.H.r.zero}
\end{align}
where $c_0 = \frac12$, $c_1 = (2\pi)^2$. 

$(vii)$ \emph{Product estimate}.
Let $r_0, r \in \R$, with $r_0 + 1 > n/2$ and $r \geq 0$. 
One has 
\begin{equation} \label{prod.est}
\| uv \|_{H^{1,r}(\R^n_+)} 
\leq C_{r_0} \| u \|_{H^{1,r_0}(\R^n_+)} \| v \|_{H^{1,r}(\R^n_+)} 
+ C_r \| u \|_{H^{1,r}(\R^n_+)} \| v \|_{H^{1,r_0}(\R^n_+)}
\end{equation}
for all $u,v \in H^{1,r}(\R^n_+) \cap H^{1, r_0}(\R^n_+)$, 
where $C_{r_0}$ depends on $r_0$ and it is independent of $r$, 
while $C_r$ depends on $r, r_0$, 
and it is an increasing function of $r$. 
For $0 \leq r \leq r_0$, one also has
\begin{equation} \label{prod.est.below.r.0}
\| uv \|_{H^{1,r}(\R^n_+)} 
\leq C_{r_0} \| u \|_{H^{1,r_0}(\R^n_+)} \| v \|_{H^{1,r}(\R^n_+)} 
\end{equation}
for all $u \in H^{1,r_0}(\R^n_+)$, 
$v \in H^{1,r}(\R^n_+)$. 
Estimates \eqref{prod.est} and \eqref{prod.est.below.r.0} can be written together as 
\begin{equation} \label{prod.est.unified}
\| uv \|_{H^{1,r}(\R^n_+)} 
\leq C_{r_0} \| u \|_{H^{1,r_0}(\R^n_+)} \| v \|_{H^{1,r}(\R^n_+)} 
+ (\chi_{r > r_0}) C_r \| u \|_{H^{1,r}(\R^n_+)} \| v \|_{H^{1,r_0}(\R^n_+)},
\end{equation}
where $(\chi_{r > r_0}) = 1$ if $r > r_0$, 
$(\chi_{r > r_0}) = 0$ if $r \leq r_0$,
$C_r$ is increasing in $r$, 
and $C_{r_0}$ does not depend on $r$. 
The same estimates also hold with $\R^n_+$ replaced by $\R^n$.

$(viii)$ \emph{Multiplication by bounded functions with bounded derivatives}. 
One has 
\begin{equation} \label{est.prod.infty.0}
\| fu \|_{H^{0,r}(\R^n_+)} 
\leq 2 \| f \|_{L^\infty(\R^n_+)} \| u \|_{H^{0,r}(\R^n_+)} 
+ C_r \| f \|_{W^{0,b,\infty}(\R^n_+)} \| u \|_{H^{0,0}(\R^n_+)}
\end{equation}
for all $f \in W^{0,b,\infty}(\R^n)$, $u \in H^{0,r}(\R^n_+)$, 
$r \geq 0$ real, where $b$ is the smallest integer such that $b \geq r$, 
\begin{equation}  \label{def.norm.W.0.b.infty}
\| f \|_{W^{0,b,\infty}(\R^n_+)} := 
\esssup_{x_n \in (0,\infty)} \| f(\cdot, x_n) \|_{W^{b,\infty}(\R^{n-1})}, 
\end{equation}
$W^{0,b,\infty}(\R^n_+)$ is the space of functions whose norm \eqref{def.norm.W.0.b.infty} 
is finite, and $C_r$ is increasing in $r$. 
Also,
\begin{equation} \label{est.prod.infty.1}
\| fu \|_{H^{1,r}(\R^n_+)} 
\leq C_0 \| f \|_{W^{1,\infty}(\R^n_+)} \| u \|_{H^{1,r}(\R^n_+)} 
+ C_r \| f \|_{W^{1,b,\infty}(\R^n_+)} \| u \|_{H^{1,0}(\R^n_+)}
\end{equation}
for all $u \in H^{1,r}(\R^n_+)$, all $f \in W^{0,b,\infty}(\R^n_+)$ 
with $\pa_{x_n} f \in W^{0,b,\infty}(\R^n_+)$,  
where 
\begin{equation}  \label{def.norm.W.1.b.infty}
\| f \|_{W^{1,b,\infty}(\R^n_+)} := \| f \|_{W^{0,b+1,\infty}(\R^n_+)}
+ \| \pa_{x_n} f \|_{W^{0,b,\infty}(\R^n_+)}
\end{equation}
and $W^{1,b,\infty}(\R^n_+)$ is the space of functions whose norm \eqref{def.norm.W.1.b.infty}
is finite. 

$(ix)$ \emph{Composition with functions of $x'$}. 
Let $s \in \{ 0, 1 \}$, $r \geq 0$ real, $r_0 > (n-1)/2$.
Let $f : \R^{n-1} \to \R^{n-1}$, 
$f(x') = a + A x' + g(x')$
be a diffeomorphisms of $\R^{n-1}$, 
with $a \in \R^{n-1}$, 
$A \in \mathrm{Mat}_{(n-1) \times (n-1)}(\R)$,
$g \in H^{r+1+r_0+s}(\R^{n-1})$.
One has 
\begin{equation} 
\| x \mapsto u (f(x'),x_n) \|_{H^{s,r}(\R^n_+)} 
\leq C_{r,f} ( \| u \|_{H^{s,r}(\R^n_+)} 
+ \| g \|_{H^{1 + r_0 + r + s}(\R^{n-1})} \| u \|_{H^{s,0}(\R^n_+)} ) 
\label{composition.est.s.r} 
\end{equation}
for all $u \in H^{s,r}(\R^n_+)$,
where $C_{r,f}$ depends on $r, |A|, \| g \|_{H^{1+r_0}(\R^{n-1})}$,
and it is increasing in $r$.
\end{lemma}

\begin{proof} 
See the Appendix. 
\end{proof}

\begin{lemma} \label{lemma:est.Lap.R.n.+}
Let $r \geq 0$ be any real number, let $s=0$ or $s=1$, 
and let $u \in H^{s+1,r}(\R^n_+)$. Then 
\begin{align} 
\| u \|_{H^{s+1,r}(\R^n_+)} 
& \leq C \| \Delta u \|_{H^{s-1,r}(\R^n_+)} 
+ C^{r+1} \| \Delta u \|_{H^{s-1,0}(\R^n_+)} 
+ C \| u(\cdot, 0) \|_{H^{s + \frac12 + r}(\R^{n-1})}
\notag \\ 
& \quad \ 
+ C^{r+1} \| u(\cdot, 0) \|_{H^{s + \frac12}(\R^{n-1})}
+ C^{r+1} \| u \|_{H^{s,0}(\R^n_+)} 
\label{meran.04}
\end{align}
where the constant $C$ is independent of $s, r, u$. 
\end{lemma}

\begin{proof} 
See the Appendix. 
\end{proof}

\begin{remark}
The constant $C$ in \eqref{meran.04} is independent of $s$ 
because the $H^{s+1,r}(\R^n_+)$ norm of $u$ 
is estimated only in the cases $s = 0$ and $s=1$; 
in general, for $s$ arbitrarily varying in $\N_0$, 
the constant $C$ would depend on $s$ in a nontrivial way, 
related to the operator norm of the extension operators involved in the proof. 
Here, however, we are only interested in high (and possibly non integer) Sobolev regularity 
with respect to the variable $x'$, corresponding to tangential directions, 
and in one or two derivatives with respect to the variable $x_n$, 
corresponding to the normal direction. 
\end{remark}

\medskip

\underline{Properties of Sobolev spaces on $\S^{n-1}$}. 
The next lemma deals with Sobolev spaces $H^s(\S^{n-1})$ on the sphere.

\begin{lemma} \label{lemma:est.tools.S.n-1} 
The Sobolev spaces $H^s(\S^{n-1})$ defined in \eqref{def.H.s.manifold} 
with $\Gamma = \S^{n-1}$ and the norms defined in \eqref{def.norm.u} 
satisfy the following properties.

$(i)$ \emph{Interpolation inequality}. Let $s_0, s_1 \in \R$, 
$s = s_0 (1-\th) + s_1 \th$, $0 < \th < 1$. Then 
\begin{equation} \label{interpol.Hr.S2}
\| u \|_{H^s(\S^{n-1})} 
\leq \| u \|_{H^{s_0}(\S^{n-1})}^{1-\th} 
\| u \|_{H^{s_1}(\S^{n-1})}^\th
\end{equation}
for all $u \in H^{s_1}(\S^{n-1})$.  

$(ii)$ \emph{Product estimate}. Let $s, s_0 \in \R$, 
with $s \geq 0$ and $s_0 > d/2$. Then 
\begin{equation} \label{prod.est.unified.Hr.S2}
\| uv \|_{H^{s}(\S^{n-1})} 
\leq C_{s_0} \| u \|_{H^{s_0}(\S^{n-1})} \| v \|_{H^{s}(\S^{n-1})}
+ (\chi_{s > s_0}) C_s \| u \|_{H^{s}(\S^{n-1})} \| v \|_{H^{s_0}(\S^{n-1})},
\end{equation}
where $(\chi_{s > s_0}) = 1$ if $s > s_0$, 
$(\chi_{s > s_0}) = 0$ if $s \leq s_0$,
the constant $C_{s_0}$ depends on $s_0$ and it is independent of $s$, 
while $C_s$ depends on $s, s_0$ and it is increasing in $s$.  

$(iii)$ \emph{Power estimate}. Let $s, s_0$ be as in $(ii)$. Then 
\begin{align}  
\| u^m \|_{H^{s}(\S^{n-1})} 
& \leq ( C_{s_0} \| u \|_{H^{s_0}(\S^{n-1})} )^{m-1} \| u \|_{H^{s}(\S^{n-1})}  
\notag \\ & \quad \ 
+ (\chi_{s > s_0}) (m-1) C_s ( C_{s_0} \| u \|_{H^{s_0}(\S^{n-1})} )^{m-2} 
\| u \|_{H^{s_0}(\S^{n-1})} \| u \|_{H^{s}(\S^{n-1})}
\label{power.est.Hr.S2}
\end{align}
for all integers $m \geq 2$, 
where $C_{s_0}, C_s$ are the constants in $(ii)$.   

$(iv)$ \emph{Tangential gradient}. 
Let $s \in \R$, with $s \geq 0$. Then 
\begin{equation}  \label{est.grad.Hr.S2}
\| \grad_{\S^{n-1}} u \|_{H^s(\S^{n-1})} 
\leq C_0 \| u \|_{H^{s+1}(\S^{n-1})} + C_s \| u \|_{H^1(\S^{n-1})},
\end{equation} 
with $C_0$ independent of $s$ 
and $C_s$ increasing in $s$. 

$(v)$ \emph{Embedding}. 
Let $s_0 \in \R$, with $s_0 > (n-1)/2$. 
If $u \in H^{s_0}(\S^{n-1})$, then $u$ is continuous, with
\begin{equation} \label{embedding.Sd}
\| u \|_{L^\infty(\S^{n-1})} \leq C_{s_0} \| u \|_{H^{s_0}(\S^{n-1})}.
\end{equation} 
\end{lemma}

\begin{proof}
See the Appendix.
\end{proof}

\medskip

\underline{Laplace and Poisson problem in the unit ball $B_1$}. 
We recall some classical results, not stated in a general version, 
but just in the form they are needed below. 

\begin{lemma}[Harmonic extension operator, or Poisson integral map] 
\label{lemma:def.harmonic.ext.op} 
Given $f \in C^\infty(\S^{n-1})$, there exists a unique 
$u \in C^\infty(B_1) \cap C^0(\overline{B_1}) \cap H^1(B_1)$ 
such that 
\[
\Delta u = 0 \ \  \text{in } B_1, \quad \ 
u = f \ \  \text{on } \S^{n-1}.
\]
The linear map $f \to u$ has a unique continuous extension 
\[
\mathtt{PI} : H^{\frac12}(\S^{n-1}) \to H^1(B_1).
\]
\end{lemma}

\begin{proof} 
See, e.g., \cite{Taylor.volume.1}, Proposition 1.7 in Section 5.1. 
\end{proof}

\begin{lemma}[Solution map for the Poisson problem in divergence form] 
\label{lemma:op.S}
For any $g \in L^2(B_1)$ there exists a unique $u$ such that 
\[
u \in H^1_0(B_1), \quad \ \  
\Delta u = \div g \ \  \text{in } B_1
\]
in the sense of distributions. 
The map $g \to u$ is a linear continuous operator 
\[
\mathtt S : L^2(B_1) \to H^1_0(B_1).
\]
\end{lemma}

\begin{proof} 
See, e.g., \cite{Taylor.volume.1}, 
Propositions 1.1 and 1.2, and Theorem 1.3 in Section 5.1. 
\end{proof}

\section{Higher tangential Sobolev regularity}
\label{sec:higher}

We decompose the annulus $1 - (\delta/2) < |x| < 1 - (\delta/4)$ 
into a sequence of annuli $\rho_{k+1} < |x| < \rho_k$, 
where $\rho_0 = 1 - (\delta/4)$, 
the sequence $(\rho_k)$ is decreasing and it converges to $1-(\delta/2)$,
and for each annulus we introduce a smooth, radial cut-off function $\zeta_k$ 
in the following way. 
Let $\ph \in C^\infty(\R, \R)$ with 
$\ph = 0$ on $(-\infty, 0]$, 
$\ph = 1$ on $[1, \infty)$, 
and $0 \leq \ph \leq 1$. 
We define 
\begin{equation}  \label{def.rho.k.zeta.k}
\rho_k := 1 - \frac{\delta}{4} \sum_{\ell=0}^k \frac{1}{2^\ell}, 
\quad 
\zeta_k(x) := \ph \Big( \frac{|x| - \rho_{k+1}}{\rho_k - \rho_{k+1}} \Big)
\quad \forall x \in \R^n, \ \ k \in \N_0.
\end{equation}
Thus $\zeta_k(x) = 0$ for $|x| \leq \rho_{k+1}$,
$\zeta_k(x) = 1$ for $|x| \geq \rho_k$, and 
\begin{equation}  \label{est.der.zeta.k}
|\zeta_k(x)| \leq 1, \quad 
| \pa_x^\alpha \zeta_k(x)| 
\leq C_\alpha 2^{k |\alpha|} 
\quad \forall x \in \R^n, \ \  k \in \N_0, \ \ \alpha \in \N_0^n,
\end{equation}
where the constant $C_\alpha$ depends on the length $|\alpha|$ of the multi-index $\alpha$
(and also on $\ph$ and $\delta$, which we consider as fixed). 
We note that $\zeta_{k+1} = 1$ where $\zeta_k \neq 0$, i.e., 
\begin{equation} \label{zeta.k.is.1.where}
\zeta_{k+1} = 1 \text{ on } \mathrm{supp}(\zeta_k), 
\quad \ \forall k \in \N_0.
\end{equation}
We also define 
\begin{equation} \label{def.zeta.*}
\rho_k^* := 1 - \frac{\delta}{2} \sum_{\ell=0}^k \frac{1}{2^\ell}, 
\quad 
\zeta_k^*(x) := \ph \Big( \frac{|x| - \rho^*_{k+1}}{\rho^*_k - \rho^*_{k+1}} \Big)
\quad \forall x \in \R^n, \ \ k \in \N_0.
\end{equation}
Note that $\rho_k^*, \zeta_k^*$ are defined in the same way as $\rho_k, \zeta_k$, 
except that $\delta/4$ in \eqref{def.rho.k.zeta.k} 
is replaced by $\delta/2$ in \eqref{def.zeta.*}, 
so that the sequence of annuli $\rho^*_{k+1} < |x| < \rho^*_k$ 
covers the annulus $1 - \delta < |x| < 1 - (\delta/2)$. 
Thus, 
\begin{equation} \label{zeta.*.k.k'}
\zeta^*_{k'} = 1 \text{ on } \mathrm{supp}(\zeta_k), 
\quad \ \forall k, k' \in \N_0.
\end{equation}
For any function $u$, we denote 
\begin{align} \label{def.seminorm}
|u|_{X^{s,r}_k} := \sum_{j=1}^N \| (u \zeta_k \psi_j) \circ \mathtt g_j \|_{H^{s,r}(\R^n_+)},
\quad 
|u|_{X^{s,r}_{*,k}} := \sum_{j=1}^N \| (u \zeta^*_k \psi_j) \circ \mathtt g_j \|_{H^{s,r}(\R^n_+)},
\end{align}
where $\psi_j, \mathtt g_j$ are the localization and rectification maps defined in section \ref{sec:new.01}.

\emph{Notation:} We say that a property involving the seminorms $| \ |_{X^{s,r}_k}$ 
also hold for the $*$ seminorms if it also holds with each seminorm $| \ |_{X^{s,r}_k}$ 
replaced by the corresponding $*$ seminorm $| \ |_{X^{s,r}_{*,k}}$. 

Before studying high regularity, 
we prove the following basic inequality for $r = 0$, $s = 0,1$. 

\begin{lemma} \label{lemma:low.norm.u.n.X}
For all $k \in \N_0$, all $u \in H^1(B_1)$, one has 
\begin{align} \label{est.u.X.low}
|u|_{X^{0,0}_k} \leq C \| u \|_{L^2(B_1)}, \quad \ 
|u|_{X^{1,0}_k} \leq C_k \| u \|_{H^1(B_1)}, 
\end{align}
where $C$ does not depend on $k$ 
and $C_k$ is increasing in $k$. 
The same holds for the $*$ seminorms. 
\end{lemma}

\begin{proof}
By \eqref{est.der.zeta.k}, for any multi-index $\a \in \N_0^n$, one has  
\begin{equation}  \label{est.zeta.k.psi.j}
\| (\zeta_k \psi_j) \circ \mathtt g_j \|_{L^\infty(\R^n_+)} \leq 1, \quad 
\| \pa_y^\a \{ (\zeta_k \psi_j) \circ \mathtt g_j \} \|_{L^\infty(\R^n_+)} \leq C_\a 2^{k|\a|}. 
\end{equation}
Since $\psi_j \circ \mathtt g_j$ is compactly supported in $A_0$, 
using \eqref{equiv.norm.H.r.zero}, one has 
\begin{align*}
\| (u \zeta_k \psi_j) \circ \mathtt g_j \|_{H^{1,0}(\R^n_+)} 
& \leq 2 \sum_{|\a| \leq 1} \| \pa_y^\a \{ (u \zeta_k \psi_j) \circ \mathtt g_j \} \|_{L^2(\R^n_+)}
\\ & 
= 2 \sum_{|\a| \leq 1} \| \pa_y^\a \{ (u \zeta_k \psi_j) \circ \mathtt g_j \} \|_{L^2(A_0 \cap \R^n_+)}
\\ & 
\leq C \sum_{|\a| \leq 1} \| \pa_x^\a (u \zeta_k \psi_j) \|_{L^2(K_j \cap B_1)} 
\leq C \| u \|_{H^1(B_1)} + C 2^k \| u \|_{L^2(B_1)}
\end{align*}
for some constant $C$ independent of $k$;
we have used the fact that $\mathtt g_j$ is a diffeomorphism of $A_0 \cap \R^n_+$ onto $K_j \cap B_1$.
Taking the sum over $j = 1, \ldots, N$ we obtain the second inequality in \eqref{est.u.X.low}. 
The other inequality is proved similarly. 
\end{proof}

\begin{lemma} \label{lemma:pinoli.S}
The linear map $\mathtt S$ defined in Lemma \ref{lemma:op.S} satisfies 
\begin{align}
|\mathtt S g|_{X^{2,0}_k} 
& \leq C_0 |g|_{X^{1,0}_k} + C_k \| g \|_{L^2(B_1)} 
\label{pinoli.15}
\end{align}
and, for all integers $b \geq 1$, all real $r$ in the interval $b-1 < r \leq b$, 
\begin{align} 
|\mathtt S g|_{X^{2,r}_k} \leq 
C_0 |g|_{X^{1,r}_k}
+ C_{r,k} \Big( \| g \|_{L^2(B_1)}
+ |g|_{X^{1,0}_k}
+ |g|_{X^{1,0}_{k+b}}
+ \sum_{\ell=1}^{b-1} |g|_{X^{1,r-\ell}_{k+\ell}} \Big),
\label{pinoli.16}
\end{align}
where $C_0$ is independent of $b,r,k$,  
and $C_{r,k}$ is increasing in $r$ and in $k$. 
For $b=1$ the last sum is empty.
The same also holds for the $*$ seminorms. 
\end{lemma}

\begin{proof}
Let $u = \mathtt S g$,   
so that $u \in H^1_0(B_1)$ and $\Delta u = \div g$ in $B_1$. 
Let $j \in \{1, \ldots, N\}$, and consider the diffeomorphism 
$\mathtt f_j : B_1 \cap K_j \to A_0 \cap \R^n_+$ 
and its inverse $\mathtt g_j$. 
Since $\Delta u = \div g$, 
taking test functions in $C^{\infty}_c(B_1 \cap K_j)$, 
integrating in $B_1$, using the divergence theorem, 
making the change of variable $x = \mathtt g_j(y)$ in the integral, 
and using the transformation rule for the gradient
\begin{equation} \label{trans.rule.grad}
(\grad u) \circ \mathtt g_j (y) = [D \mathtt g_j(y)]^{-T} \grad( u \circ \mathtt g_j)(y), 
\quad \ y \in A_0 \cap \R^n_+,
\end{equation}
one finds that 
\begin{equation} \label{eq.for.u.circ.mathtt.g.j}
\divergence (\mathtt p \mathtt Q \grad \tilde u_j) 
= \divergence (\mathtt p (D \mathtt g_j)^{-1} \tilde g_j) 
\quad \text{in } A_0 \cap \R^n_+,
\end{equation}
where 
\begin{alignat}{2} 
\tilde u_j & := u \circ \mathtt g_j, \qquad & 
\mathtt p(y) & := \det (D \mathtt g_j(y)) = \det (D \mathtt g(y)), 
\label{def.mathtt.p}
\\
\tilde g_j & := g \circ \mathtt g_j, \qquad & 
\mathtt Q(y) & := [D \mathtt g_j(y)]^{-1} [D \mathtt g_j(y)]^{-T} 
= [D \mathtt g(y)]^{-1} [D \mathtt g(y)]^{-T}
\label{def.mathtt.Q}
\end{alignat}
in $A_0 \cap \R^n_+$. 
The last identities in \eqref{def.mathtt.p} and \eqref{def.mathtt.Q}
hold because $\mathtt g_j = \mathtt R_j^{-1} \circ \mathtt g$ 
and $\mathtt R_j$ is a rotation. 
By \eqref{explicit.mathtt.f.inv}, in $A_0$ one has 
\begin{align}
\label{una.buona.volta}
\mathtt Q(y) & = \begin{bmatrix} 
\mathtt q(y) (\mathrm{I'} + y' \otimes y')  & 0 \\
0 & 1 \end{bmatrix}\!\!,
\quad
\mathtt q(y) := \frac{1+|y'|^2}{(1-y_n)^2},
\quad  
\mathtt p(y) = \frac{ (1-y_n)^{n-1} }{ (1 + |y'|^2)^{\frac{n}{2}} },
\end{align}
where $I'$ is the $(n-1) \times (n-1)$ identity matrix. 
Since the bottom-right entry of the matrix $\mathtt Q$ is 1, 
and since derivatives with respect to $y_n$ play a separate role, 
it is convenient to work with $\mathtt Q$ instead of $\mathtt P$. 
Dividing by $\mathtt p$, \eqref{eq.for.u.circ.mathtt.g.j} becomes 
\begin{align} 
\divergence (\mathtt Q \grad \tilde u_j) 
& = - \mathtt p^{-1} \la \grad \mathtt p , \mathtt Q \grad \tilde u_j \ra
+ \mathtt p^{-1} \divergence (\mathtt p (D \mathtt g_j)^{-1} \tilde g_j) 
\quad \text{in } A_0 \cap \R^3_+.
\label{eq.for.u.circ.mathtt.g.j.with.mathtt.Q}
\end{align}
Now, given any two functions $v,a$, one has 
\begin{equation} \label{div.P.grad.v.m}
\divergence (\mathtt Q \grad (va)) 
= a \divergence (\mathtt Q \grad v) 
+ 2 \la \mathtt Q \grad a , \grad v \ra 
+ v \divergence (\mathtt Q \grad a).
\end{equation}
To take advantage of the fact that $\mathtt Q(y)$ is close to $\mathtt Q(0) = I$ for $y$ close to $0$, 
we write $\mathtt Q$ in the LHS of \eqref{div.P.grad.v.m} 
as the sum $I + (\mathtt Q - I)$, and \eqref{div.P.grad.v.m} becomes 
\begin{equation} \label{Lap.v.m}
\Delta (va)
= \divergence \{ (I - \mathtt Q) \grad (va) \}
+ a \divergence (\mathtt Q \grad v) 
+ 2 \la \mathtt Q \grad a , \grad v \ra 
+ v \divergence ( \mathtt Q \grad a).
\end{equation}
Let 
\begin{alignat}{2}
\tilde \psi_j & := \psi_j \circ \mathtt g_j, \qquad & 
a_{kj} & := \tilde \zeta_{kj} \tilde \psi_j = (\zeta_k \psi_j) \circ \mathtt g_j, 
\label{def.a.kj}
\\ 
\tilde \zeta_{kj} & := \zeta_k \circ \mathtt g_j, \qquad &  
w_{kj} & := \tilde u_j a_{kj} = \tilde u_j \tilde \zeta_{kj} \tilde \psi_j
= (u \zeta_k \psi_j) \circ \mathtt g_j.
\label{def.w.kj}
\end{alignat} 
By \eqref{Lap.v.m} with $v = \tilde u_j$, $a = a_{kj}$, 
using \eqref{eq.for.u.circ.mathtt.g.j.with.mathtt.Q} 
to substitute the second term in the RHS of \eqref{Lap.v.m}, 
we get 
\begin{align} 
\Delta w_{kj} 
& = \divergence ((I - \mathtt Q) \grad w_{kj} ) 
- a_{kj} \mathtt p^{-1} \la \grad \mathtt p , \mathtt Q \grad \tilde u_j \ra
+ a_{kj} \mathtt p^{-1} \divergence (\mathtt p (D \mathtt g_j)^{-1} \tilde g_j) 
\notag \\ & \quad  
+ 2 \la \mathtt Q \grad a_{kj} , \grad \tilde u_j \ra
+ \tilde u_j \divergence ( \mathtt Q \grad a_{kj} )
\label{eq.Lap.w}
\end{align}
in $A_0 \cap \R^n_+$. 
Note that $\tilde \psi_j$ is compactly supported in $A_0$ 
because $\psi_j$ is compactly supported in $K_j$. 
Hence $w_{kj}$ is also compactly supported in $A_0$, and we 
extend it to $\R^n_+$ by setting $w_j := 0$ in $\R^n_+ \setminus A_0$. 
We also extend $\mathtt Q, \mathtt p$ to $\R^n$ 
replacing $\mathtt g$ with its extension 
$\tilde{\mathtt g} = \tilde{\mathtt f}^{-1}$, see \eqref{def.tilde.mathtt.f};
with a little abuse of notation, we denote the extensions without changing letters.
For future convenience, we note here that, 
by \eqref{def.tilde.mathtt.f}, one has $\tilde{\mathtt f}(x) = x - p_0$ 
for all $x$ outside the ball $B_{4\delta}(p_0)$, 
and therefore $\mathtt p = 1$, $\mathtt Q = I$ 
outside the set $\mathtt f(B_{4\delta}(p_0))$, 
which is a neighborhood of $A_0 = \mathtt f(B_{2\delta}(p_0))$. 

Since $u \in H^1_0(B_1)$ and $\psi_j, \zeta_k, \mathtt g_j$ are smooth, 
one has $w_{kj} \in H^1_0(\R^n_+)$. 
Then, by Lemma \ref{lemma:est.Lap.R.n.+}, for every real $r \geq 0$ one has 
\begin{align} 
\| w_{kj} \|_{H^{2,r}(\R^n_+)} 
& \leq C \| \Delta w_{kj} \|_{H^{0,r}(\R^n_+)} 
+ C^{r+1} \| \Delta w_{kj} \|_{L^2(\R^n_+)} 
+ C^{r+1} \| w_{kj} \|_{H^1(\R^n_+)}, 
\label{olive.01}
\end{align}
where $C > 0$ is the constant in \eqref{meran.04},
and $\Delta w_{kj}$ is given by \eqref{eq.Lap.w}. 
Now we estimate each term in the right-hand side of \eqref{eq.Lap.w}. 

\emph{First term}. 
By \eqref{una.buona.volta}, the last column and the last row 
of the matrix $(I - \mathtt Q(y))$ vanish in $A_0$ $\cap \, \R^n_+$,
and $\divergence ((I - \mathtt Q) \grad w_j)$ 
is the sum of terms of the form $\pa_{y_i} ( c(y) \pa_{y_k} w_j)$ 
with $i, k \in \{ 1, \ldots, n-1 \}$, without derivatives of $w_j$ with respect to $y_n$. 
Hence, by \eqref{pa.k.est} and \eqref{est.prod.infty.0}, one has 
\begin{align*}
\| \divergence ((I - \mathtt Q) \grad w_{kj}) \|_{H^{0,r}(\R^n_+)} 
& \leq C_0 \| I - \mathtt Q \|_{L^\infty(\R^n_+)} \| w_{kj} \|_{H^{0,r+2}(\R^n_+)}
\notag \\ & \quad \ 
+ C_r \| I - \mathtt Q \|_{W^{0,b,\infty}(\R^n_+)} \| w_{kj} \|_{H^{0,1}(\R^n_+)},
\end{align*}
where $C_0, C_r, b$ are given by the application of \eqref{est.prod.infty.0}. 
Given any $\e > 0$, the radius $2\delta$ of the balls $K_1, \ldots, K_N$ 
can be chosen sufficiently small to have 
$C_0 \| I - \mathtt Q \|_{L^\infty(\R^n_+)} \leq \e$; 
note that such a $\delta$ is independent on $r$. 
Hence 
\begin{align}  \label{olive.02}
\| \divergence ((I - \mathtt Q) \grad w_{kj}) \|_{H^{0,r}(\R^n_+)} 
& \leq \e \| w_{kj} \|_{H^{0,r+2}(\R^n_+)} 
+ C_r \| w_{kj} \|_{H^{0,1}(\R^n_+)},
\end{align} 
where $C_r$ is a constant depending on $r$, increasing in $r$, 
incorporating the factor $\| I - \mathtt Q \|_{W^{0,b,\infty}(\R^n_+)}$. 

\emph{Second term}. 
By \eqref{def.mathtt.p}, 
\eqref{zeta.k.is.1.where} and \eqref{res.unity.in.the.annulus}, 
the second term in the RHS of \eqref{eq.Lap.w} is 
\begin{align}
& a_{kj} \mathtt p^{-1} 
\la \grad \mathtt p , \mathtt Q \grad \tilde u_{j} \ra
= \sum_{\ell \in \mA_j} a_{kj} \mathtt p^{-1} \la \grad \mathtt p , \mathtt Q \grad 
\{ (u \zeta_{k+1} \psi_\ell) \circ \mathtt g_j \} \ra,
\label{echin.03}
\end{align}
where $\mA_j := \{ \ell \in \{ 1, \ldots, N \} : K_\ell \cap K_j \neq \emptyset \}$. 
Identity \eqref{echin.03} holds because, by \eqref{zeta.k.is.1.where}, 
$\zeta_{k+1} \circ \mathtt g_j = 1$ in the support of $\zeta_k \circ \mathtt g_j$, 
by \eqref{res.unity.in.the.annulus},  
and because, if the balls $K_\ell, K_j$ are disjoint, 
then $\psi_\ell \circ \mathtt g_j = 0$ in the support of $\psi_j \circ \mathtt g_j$. 
Moreover, for every $\ell \in \mA_j$, one has 
\begin{equation} \label{transition.w}
(u \zeta_{k+1} \psi_\ell) \circ \mathtt g_j 
= (u \zeta_{k+1} \psi_\ell) \circ \mathtt g_\ell \circ \mathtt f_\ell \circ \mathtt g_j
= w_{k+1,\ell} \circ (\mathtt f_\ell \circ \mathtt g_j),
\end{equation}
where the transition map $\mathtt f_\ell \circ \mathtt g_j$ 
in $A_0$ is given by formula \eqref{formula.transition.map}.
We consider the extension of the transition map to $\R^n_+$ 
given in \eqref{rotat.09}, \eqref{extension.in.A0}. 
By Lemma \ref{lemma:rotat.tilde}, the function $\tilde{\mathtt T}_{\ell j}$ 
is a diffeomorphism of $\R^{n-1}$, and it is the sum 
of an affine map and a smooth function with compact support. 
Hence, by the composition estimate \eqref{composition.est.s.r}, we get
\begin{equation}  \label{echin.04}
\| w_{k+1,\ell} \circ (\mathtt f_\ell \circ \mathtt g_j) \|_{H^{1,r}(\R^n_+)}
\leq C_r \|  w_{k+1,\ell} \|_{H^{1,r}(\R^n_+)},
\end{equation}
where the constant $C_r$ is increasing in $r$. 
Concerning $a_{kj}$, one has estimate \eqref{est.zeta.k.psi.j}. 
Also, by \eqref{def.rho.k.zeta.k}, 
$\zeta_k(\mathtt g_j(y))$ depends only on $|\mathtt g_j(y)|$; 
moreover $\mathtt g_j(y) = \mathtt R_j^{-1} \mathtt g(y)$, 
and $|\mathtt g_j(y)| = |\mathtt g(y)| = 1 - y_n$ by \eqref{explicit.mathtt.f.inv}.  
Hence $\zeta_k \circ \mathtt g_j$ depends only on $y_n$, and not on $y'$, 
so that $\pa_y^\a \{ (\zeta_k \psi_j) \circ \mathtt g_j \} 
= (\zeta_k \circ \mathtt g_j) \pa_y^\a (\psi_j \circ \mathtt g_j)$ 
for all multi-indices $\a = (\a', 0)$. 
As a consequence, by \eqref{def.a.kj} and \eqref{def.norm.W.0.b.infty}, one has
\begin{equation}  \label{est.a.W.0.b.infty}
\| a_{kj} \|_{W^{0,b,\infty}(\R^n_+)}
= \| (\zeta_k \psi_j) \circ \mathtt g_j \|_{W^{0,b,\infty}(\R^n_+)}
\leq \| \psi_j \circ \mathtt g_j \|_{W^{0,b,\infty}(\R^n_+)} \leq C_r
\end{equation}
for some constant $C_r$ 
dependent on $r$, increasing in $r$ (in fact, $C_r$ depends on $b$).  
By \eqref{est.prod.infty.0}, \eqref{transition.w}, \eqref{echin.04}, 
\eqref{est.zeta.k.psi.j}, \eqref{est.a.W.0.b.infty}, 
each term of the sum in the RHS of \eqref{echin.03} satisfies 
\begin{equation}  \label{echin.05}
\| a_{kj} \mathtt p^{-1} \la \grad \mathtt p , \mathtt Q \grad 
\{ (u \zeta_{k+1} \psi_\ell) \circ \mathtt g_j \} \ra \|_{H^{0,r}(\R^n_+)}
\leq C_r \| w_{k+1,\ell} \|_{H^{1,r}(\R^n_+)}.
\end{equation}
The sum of \eqref{echin.05} over $\ell \in \mA_j$ gives an estimate 
of the second term in the RHS of \eqref{eq.Lap.w}.

\emph{Third term.} We write the third term in the RHS of \eqref{eq.Lap.w} as 
\begin{align}
a_{kj} \mathtt p^{-1} \divergence (\mathtt p (D \mathtt g_j)^{-1} \tilde g_j) 
= \mathtt p^{-1} \divergence (\mathtt p (D \mathtt g_j)^{-1} \gamma_{kj}) 
- \la (D \mathtt g_j)^{-1} \tilde g_j , \grad a_{kj} \ra,
\label{third.term.id}
\end{align}
where 
\begin{equation} \label{def.gamma.nkj}
\gamma_{kj} := \tilde g_j a_{kj} 
= (g \zeta_k \psi_j) \circ \mathtt g_j.
\end{equation}
By \eqref{est.prod.infty.0}, \eqref{pa.k.est}, \eqref{pa.n.est}, 
\eqref{est.prod.infty.1}, the first term in the RHS of \eqref{third.term.id} satisfies 
\begin{equation}  \label{olive.21}
\| \mathtt p^{-1} \divergence (\mathtt p (D \mathtt g_j)^{-1} \gamma_{kj}) \|_{H^{0,r}(\R^n_+)} 
\leq C_0 \| \gamma_{kj} \|_{H^{1,r}(\R^n_+)} + C_r \| \gamma_{kj} \|_{H^{1,0}(\R^n_+)}.
\end{equation}
Using \eqref{def.mathtt.Q}, \eqref{zeta.k.is.1.where} and \eqref{res.unity.in.the.annulus} 
like in \eqref{echin.03}, and using the identity 
$\mathtt g_j = \mathtt g_\ell \circ \mathtt f_\ell \circ \mathtt g_j$ as in \eqref{transition.w},
we write the last term of \eqref{third.term.id} as 
\begin{align}
\la (D \mathtt g_j)^{-1} \tilde g_j , \grad a_{kj} \ra
& = \sum_{\ell \in \mA_j} 
\la (D \mathtt g_j)^{-1} ((g \zeta_{k+1} \psi_\ell) \circ \mathtt g_j) , \grad a_{kj} \ra
\notag \\
& = \sum_{\ell \in \mA_j} 
\la (D \mathtt g_j)^{-1} (\gamma_{k+1, \ell} \circ (\mathtt f_\ell \circ \mathtt g_j)) , \grad a_{kj} \ra.
\label{echin.06}
\end{align}
Hence, using \eqref{est.prod.infty.0}, 
\eqref{est.zeta.k.psi.j}, 
\eqref{est.a.W.0.b.infty}, 
and Lemma \ref{lemma:rotat.tilde} and \eqref{composition.est.s.r} 
for the composition with the transition map $\mathtt f_\ell \circ \mathtt g_j$,
one proves that each term of the sum in \eqref{echin.06} satisfies 
\begin{equation} \label{olive.20}
\| \la (D \mathtt g_j)^{-1} (\gamma_{k+1, \ell} \circ (\mathtt f_\ell \circ \mathtt g_j)) , 
\grad a_{kj} \ra \|_{H^{0,r}(\R^n_+)} 
\leq C_r 2^k \| \gamma_{k+1, \ell} \|_{H^{0,r}(\R^n_+)}.
\end{equation}
The constant $C_r$ in \eqref{olive.20} depends on $r$ 
because we have estimated a composition; 
the factor $2^k$ is due to the presence 
of the first derivative $\pa_{y_n} (\zeta_{k+1} \circ \mathtt g_j)$.  
From \eqref{olive.21} and \eqref{olive.20} we obtain an estimate for \eqref{third.term.id}.

\emph{Fourth and fifth term}. 
Proceeding as in \eqref{echin.03}-\eqref{echin.05}, one proves that the 
fourth and fifth term in \eqref{eq.Lap.w} are the sum over $\ell \in \mA_j$ of terms 
satisfying the same bound \eqref{echin.05}, but with $C_r 2^{2k}$ instead of $C_r$. 
The factor $2^{2k}$ is due to the presence of the second derivative of $\zeta_{k+1}$ 
with respect to $y_n$.

\emph{Estimate of $u$}. 
By \eqref{eq.Lap.w}, \eqref{olive.02}, \eqref{echin.05}, \eqref{olive.21}, \eqref{olive.20}, 
and \eqref{echin.05} with factor $2^{2k}$, one has 
\begin{multline}
\| \Delta w_{kj} \|_{H^{0,r}(\R^n_+)}
\leq 
\e \| w_{kj} \|_{H^{0, r+2}(\R^n_+)}
+ C_r \| w_{kj} \|_{H^{1,0}(\R^n_+)}
+ C_r 2^{2k} \sum_{\ell \in \mA_j} \| w_{k+1, \ell} \|_{H^{1,r}(\R^n_+)}
\\
+ C_0 \| \gamma_{kj} \|_{H^{1,r}(\R^n_+)} 
+ C_r \| \gamma_{kj} \|_{H^{1,0}(\R^n_+)} 
+ C_r 2^k \sum_{\ell \in \mA_j} \| \gamma_{k+1,\ell} \|_{H^{0,r}(\R^n_+)}.
\label{echin.10}
\end{multline}
By \eqref{olive.01}, \eqref{echin.10} and \eqref{echin.10}$|_{r=0}$, we get
\begin{align}
\| w_{kj} \|_{H^{2,r}(\R^n_+)}
& \leq 
\e C_0 \| w_{kj} \|_{H^{0, r+2}(\R^n_+)} 
+ \e C_r \| w_{kj} \|_{H^{0, 2}(\R^n_+)}
+ C_r \| w_{kj} \|_{H^{1,0}(\R^n_+)}
\notag \\ & \quad \ 
+ C_r 2^{2k} \sum_{\ell \in \mA_j} \| w_{k+1, \ell} \|_{H^{1,r}(\R^n_+)}
+ C_0 \| \gamma_{kj} \|_{H^{1,r}(\R^n_+)} 
+ C_r \| \gamma_{kj} \|_{H^{1,0}(\R^n_+)} 
\notag \\ & \quad \ 
+ C_r 2^k \sum_{\ell \in \mA_j} \| \gamma_{k+1,\ell} \|_{H^{0,r}(\R^n_+)}
\label{olive.05}
\end{align}
for some $C_0, C_r$ (possibly different than above).
Recalling definitions \eqref{def.w.kj}, \eqref{def.gamma.nkj} of $w_{kj}, \gamma_{kj}$,  
taking the sum of \eqref{olive.05} over $j = 1, \ldots, N$, 
and using notation \eqref{def.seminorm}, 
we obtain 
\begin{align}
|u|_{X^{2,r}_k}
& \leq \e a_0 |u|_{X^{0,r+2}_k}
+ \e b_r |u|_{X^{0,2}_k}
+ C_r |u|_{X^{1,0}_k}
+ C_r 2^{2k} |u|_{X^{1,r}_{k+1}}
\notag \\ & \quad \ 
+ C_0 |g|_{X^{1,r}_k}
+ C_r |g|_{X^{1,0}_k}
+ C_r 2^k |g|_{X^{0,r}_{k+1}},
\label{olive.09}
\end{align}
where we have denoted by $a_0$ and $b_r$ the constants appearing with the factor $\e$;
also, $b_r, C_r$ are increasing functions of $r$. 
For $r=0$, \eqref{olive.09} becomes
\begin{align}
|u|_{X^{2,0}_k}
& \leq \e (a_0 + b_0) |u|_{X^{0,2}_k}
+ C_0 |u|_{X^{1,0}_k}
+ C_0 2^{2k} |u|_{X^{1,0}_{k+1}}
+ 2 C_0 |g|_{X^{1,0}_k}
+ C_0 2^k |g|_{X^{0,0}_{k+1}}.
\label{olive.10}
\end{align}
Since $|u|_{X^{0,2}_k} \leq |u|_{X^{2,0}_k}$, 
we fix $\e$ (and therefore the half-radius $\delta$ of the balls $K_j$ 
and their cardinality $N$) 
such that $\e (a_0 + b_0) \leq 1/2$,
so that the first term in the RHS of \eqref{olive.10} can be absorbed 
by the LHS of \eqref{olive.10}. 
The second, the third and the fifth term in the RHS of \eqref{olive.10} are estimated 
by \eqref{est.u.X.low} and Lemma \ref{lemma:op.S}, and we get \eqref{pinoli.15}. 
Note that $\e, \delta, N$ do not depend on $r,j,k$. 

Now that $\e$ has been fixed, consider \eqref{olive.09} again. 
Since $a_0 \e \leq 1/2$ and $|u|_{X^{0,r+2}_k} \leq |u|_{X^{2,r}_k}$, 
the first term in the RHS of \eqref{olive.09} can be absorbed by the LHS of \eqref{olive.09}.  
Also, we use \eqref{pinoli.15} to estimate the second and third term in the RHS of \eqref{olive.09}. 
Thus, for all $r,k$, we get
\begin{align}
|u|_{X^{2,r}_k}
& \leq 
C_0 |g|_{X^{1,r}_k} + C_r |g|_{X^{1,0}_k} 
+ C_{r,k} (\| g \|_{L^2(B_1)} + |g|_{X^{0,r}_{k+1}} + |u|_{X^{1,r}_{k+1}} ).
\label{pinoli.12}
\end{align}
For $0 < r \leq 1$, we consider \eqref{pinoli.12}, 
we use the fact that 
\[
|g|_{X^{0,r}_{k+1}} \leq 
|g|_{X^{0,1}_{k+1}} \leq 
|g|_{X^{1,0}_{k+1}}, 
\quad 
|u|_{X^{1,r}_{k+1}} 
\leq |u|_{X^{1,1}_{k+1}}
\leq |u|_{X^{2,0}_{k+1}},
\]
and we use \eqref{pinoli.15} to bound $|u|_{X^{2,0}_{k+1}}$; 
in this way we obtain 
\begin{align}
|u|_{X^{2,r}_k} 
& \leq 
C_0 |g|_{X^{1,r}_k}
+ C_{r,k} (\| g \|_{L^2(B_1)} + |g|_{X^{1,0}_k} + |g|_{X^{1,0}_{k+1}}), 
\quad 0 < r \leq 1,
\label{pinoli.13}
\end{align}
with some possibly larger $C_{r,k}$, and the same $C_0$ of \eqref{pinoli.12}. 

Now we prove \eqref{pinoli.16} by induction on $b$. 
For $b=1$, \eqref{pinoli.16} is \eqref{pinoli.13}. 
Assume that \eqref{pinoli.16} is proved for some integer $b \geq 1$, 
and let $b < r \leq b+1$. 
Apply \eqref{pinoli.12}, 
use the inequality $|u|_{X^{1,r}_{k+1}} \leq |u|_{X^{2,r-1}_{k+1}}$
and the induction assumption \eqref{pinoli.16} (with $r-1, k+1$ instead of $r,k$)
to estimate $|u|_{X^{2,r-1}_{k+1}}$. This gives \eqref{pinoli.16} with $b+1$ instead of $b$ 
(and with constants $C_{r,k}$ possibly different than those of the induction assumption. 
This is not a problem, because now $r$ is in the next interval $(b, b+1]$, 
and each of such intervals has its own set of constants). 
Note that the constant $C_0$ in \eqref{pinoli.16} is not changed by the induction step. 
Hence \eqref{pinoli.16} holds for all integers $b \geq 1$. 
\end{proof}

The next lemma deals with the Poisson integral operator $\mathtt{PI}$. 

\begin{lemma} \label{lemma:pinoli.u.0}
The linear map $\mathtt{PI}$ defined in Lemma \ref{lemma:def.harmonic.ext.op} 
satisfies, for all $k \in \N_0$, all real $r \geq 0$, 
\begin{align} 
| \mathtt{PI} f |_{X^{2,r}_k} 
& \leq C_0 \| f \|_{H^{r+\frac32}(\S^{n-1})}
+ C_{r,k} \| f \|_{H^{\frac32}(\S^{n-1})},
\label{pinoli.interpol.psi}
\end{align}
where $C_0$ is independent of $r,k$ 
and $C_{r,k}$ is increasing in $r$ and $k$.
The same also holds for the $*$ seminorms. 
\end{lemma}

\begin{proof}
We closely follow the proof of Lemma \ref{lemma:pinoli.S}. 
Let $u = \mathtt{PI} f$, so that $u \in H^1(B_1)$, $u|_{\S^{n-1}} = f$, and $\Delta u = 0$ in $B_1$. 
The function $w_{kj}$ defined in \eqref{def.w.kj} 
satisfies \eqref{eq.Lap.w} with $g = 0$, 
that is, without the third term in the RHS, 
and 
\begin{equation} \label{boundary.cond.w.0}
w_{kj} \in H^1(\R^n_+), \quad 
w_{kj} = (f \psi_j) \circ \mathtt g_j \quad \text{on } \R^{n-1} \times \{ 0 \}.
\end{equation}
Note that $\zeta_k \circ \mathtt g_j (\cdot, 0) = 1$ for all $k \in \N_0$. 
By Lemma \ref{lemma:est.Lap.R.n.+}, 
\begin{align} 
\| w_{kj} \|_{H^{2,r}(\R^n_+)} 
& \leq C \| \Delta w_{kj} \|_{H^{0,r}(\R^n_+)} 
+ C^{r+1} \| \Delta w_{kj} \|_{L^2(\R^n_+)} 
+ C \| w_{kj}(\cdot, 0) \|_{H^{r + \frac32}(\R^{n-1})}
\notag \\ & \quad \ 
+ C^{r+1} \| w_{kj}(\cdot, 0) \|_{H^{\frac32}(\R^{n-1})}
+ C^{r+1} \| w_{kj} \|_{H^1(\R^n_+)}, 
\label{pinoli.01.BIS}
\end{align}
where $C$ is the constant in \eqref{meran.04},
$\Delta w_{kj}$ is given by \eqref{eq.Lap.w} (with $g = 0$) 
and the trace $w_{kj}(\cdot, 0)$ is given by \eqref{boundary.cond.w.0}.
Now $\Delta w_{kj}$ satisfies \eqref{echin.10} without the terms with 
$\gamma_{kj}, \gamma_{k+1,\ell}$, which are zero. 
Thus, by \eqref{pinoli.01.BIS}, \eqref{echin.10} and \eqref{echin.10}$|_{r=0}$, we get
\begin{align}
\| w_{kj} \|_{H^{2,r}(\R^n_+)}
& \leq \e C_0 \| w_{kj} \|_{H^{0, r+2}(\R^n_+)} 
+ \e C_r \| w_{kj} \|_{H^{0, 2}(\R^n_+)}
+ C_r \| w_{kj} \|_{H^{1,0}(\R^n_+)}
\notag \\ 
& \quad \ 
+ C_r 2^{2k} \sum_{\ell \in \mA_j} \| w_{k+1, \ell} \|_{H^{1,r}(\R^n_+)}
+ C \| (f \psi_j) \circ \mathtt g_j(\cdot, 0) \|_{H^{r+\frac32}(\R^{n-1})}
\notag \\ 
& \quad \ 
+ C^{r+1} \| (f \psi_j) \circ \mathtt g_j(\cdot, 0) \|_{H^{\frac32}(\R^{n-1})}
\label{pinoli.05.BIS}
\end{align}
instead of \eqref{olive.05}. 
Hence $u$ satisfies \eqref{olive.09} with $g = 0$ and the additional terms 
$C_0 \| f \|_{H^{r+\frac32}(\S^{n-1})}$ and  
$C_r \| f \|_{H^{\frac32}(\S^{n-1})}$ 
in the RHS, and, using also Lemmas \ref{lemma:low.norm.u.n.X} and \ref{lemma:def.harmonic.ext.op}, 
we obtain 
\begin{align}
|u|_{X^{2,0}_k}
& \leq 
C_0 \| f \|_{H^{\frac32}(\S^{n-1})}
+ C_k \| f \|_{H^{\frac12}(\S^{n-1})},
\notag \\
|u|_{X^{2,r}_k}
& \leq 
C_0 \| f \|_{H^{r + \frac32}(\S^{n-1})}
+ C_r \| f \|_{H^{\frac32}(\S^{n-1})}
+ C_{r,k} \| f \|_{H^{\frac12}(\S^{n-1})}
+ C_{r,k} |u|_{X^{1,r}_{k+1}}.
\label{pinoli.12.BIS}
\end{align}
Then, by induction, similarly as in the proof of \eqref{pinoli.15}, \eqref{pinoli.16}, 
one proves that 
\begin{equation} \label{pinoli.16.BIS}
|u|_{X^{2,r}_k} 
\leq C_0 \| f \|_{H^{r + \frac32}(\S^{n-1})} 
+ C_{r,k} \| f \|_{H^{r + \frac12}(\S^{n-1})}.
\end{equation}
Inequality \eqref{pinoli.interpol.psi} is deduced from \eqref{pinoli.16.BIS} 
by the interpolation inequality \eqref{interpol.Hr.S2}, 
using the fact that 
\begin{equation} \label{interpol.trick}
|f|_r 
\leq |f|_{r_0}^{1-\th} |f|_{r_1}^\th 
= \frac{1}{\lm} (\lm^p |f|_{r_0})^{1-\th} |f|_{r_1}^\th
\leq \frac{1}{\lm} ( \lm^p |f|_{r_0} + |f|_{r_1} ) 
= \lm^{p-1} |f|_{r_0} + \frac{|f|_{r_1}}{\lm}  
\end{equation}
for any $\lm>0$, 
with $r = (1-\th) r_0 + \th r_1$, $0<\th<1$, $p = 1/(1-\th)$, 
$|f|_r = \| f \|_{H^r(\S^{n-1})}$, 
and choosing $\lm$ sufficiently large, depending on $r,k$. 
\end{proof}

Concerning the gradient operator, one has the following inequality. 

\begin{lemma} \label{lemma:pinoli.grad}
For all integer $k \geq 0$, all real $r \geq 0$, 
\begin{equation} \label{pinoli.grad}
|\grad u|_{X^{1,r}_k} \leq C_0 |u|_{X^{2,r}_k} +  C_r |u|_{X^{2,0}_k} + C_{r,k} |u|_{X^{1,r}_{k+1}},
\end{equation}
where $C_0$ is independent of $r$ and $k$, 
$C_r$ is increasing in $r$ and independent of $k$, 
and $C_{r,k}$ is increasing in $r$ and in $k$.
The same holds for the $*$ seminorms. 
\end{lemma}

\begin{proof}
Given any $u$, with the notation in \eqref{def.a.kj}, \eqref{def.w.kj}, \eqref{def.mathtt.p}, 
by \eqref{trans.rule.grad} one has 
\[
((\grad u) \zeta_k \psi_j) \circ \mathtt g_j 
= (D \mathtt g_j)^{-T} (\grad \tilde u_j) a_{kj}
= (D \mathtt g_j)^{-T} \grad w_{kj} 
- (D \mathtt g_j)^{-T} \tilde u_j \grad a_{kj} 
\]
in $A_0 \cap \R^n_+$. 
By \eqref{est.prod.infty.1}, 
\eqref{pa.k.est}, \eqref{pa.n.est},
\[
\| (D \mathtt g_j)^{-T} \grad w_{kj} \|_{H^{1,r}(\R^n_+)} 
\leq C_0 \| w_{kj} \|_{H^{2,r}(\R^n_+)} 
+ C_r \| w_{kj} \|_{H^{2,0}(\R^n_+)}.
\]
Using \eqref{res.unity.in.the.annulus} and \eqref{zeta.k.is.1.where} 
and proceeding like in \eqref{echin.03}-\eqref{echin.04}, one has 
\begin{gather*}
\tilde u_j \grad a_{kj} 
= \sum_{\ell \in \mA_j} w_{k+1, \ell} \circ (\mathtt f_\ell \circ \mathtt g_j) \grad a_{kj}, 
\\
\| (D \mathtt g_j)^{-T} \tilde u_j \grad a_{kj} \|_{H^{1,r}(\R^n_+)}
\leq C_{r,k} \sum_{\ell \in \mA_j} \| w_{k+1, \ell} \|_{H^{1,r}(\R^n_+)}.
\end{gather*}
The sum over $j$ gives the thesis.
\end{proof}

The next lemma deals with the Poisson integral operator $\mathtt{PI}$ 
in the $W^{1,\infty}(B_1)$ norm.  

\begin{lemma} \label{lemma:PI.infty}
For any real $r_0$ with $1+r_0 > n/2$, 
the linear map $\mathtt{PI}$ defined in Lemma \ref{lemma:def.harmonic.ext.op} satisfies 
\begin{align} 
\| \mathtt{PI} f \|_{W^{1,\infty}(B_1)} 
= \| \mathtt{PI} f \|_{L^\infty(B_1)} 
+ \| \grad(\mathtt{PI} f) \|_{L^\infty(B_1)} 
& \leq C_{r_0} \| f \|_{H^{r_0+\frac32}(\S^{n-1})},
\label{pinoli.L.infty.bound.PI}
\end{align}
where $C_{r_0}$ depends on $r_0$. 
If $f \in H^{r_0+\frac32}(\S^{n-1})$, 
then $\mathtt{PI} f$ is the restriction to $B_1$ of a function in $C^1(\R^n)$. 
\end{lemma}

\begin{proof}
Let $u = \mathtt{PI} f$. 
Since $u$ is harmonic in $B_1$, from the maximum principle one has 
$\| u \|_{L^\infty(B_1)} \leq \| f \|_{L^\infty(\S^{n-1})}$, 
and, by \eqref{embedding.Sd}, this is 
$\leq C_{s_0} \| f \|_{H^{s_0}(\S^{n-1})}$ with $s_0 > (n-1)/2$.
Each partial derivative $\pa_{x_i} u$ is also harmonic in $B_1$, 
and therefore, by the maximum principle,  
\[
\| \pa_{x_i} u \|_{L^\infty(B_1)} = \| \pa_{x_i} u \|_{L^\infty(\mA)},
\]
where $\mA$ is the annulus $1-(\delta/4) < |x| < 1$.   
In $\mA$ one has $\zeta_0 = 1$ and $\sum_{j=1}^N \psi_j = 1$, whence 
$\pa_{x_i} u = \sum_{j=1}^N (\pa_{x_i} u) \zeta_0 \psi_j$ in $\mA$ and 
\[
\| \pa_{x_i} u \|_{L^\infty(\mA)} 
\leq \sum_{j=1}^N \| (\pa_{x_i} u) \zeta_0 \psi_j \|_{L^\infty(\mA)} 
\leq \sum_{j=1}^N \| (\pa_{x_i} u) \zeta_0 \psi_j \|_{L^\infty(B_1)} 
= \sum_{j=1}^N \| (\pa_{x_i} u) \zeta_0 \psi_j \|_{L^\infty(B_1 \cap K_j)} 
\]
because $\psi_j = 0$ outside $K_j$. 
Using the change of variable $x = \mathtt g_j(y)$ and \eqref{embedding.ineq}, 
\begin{align*}
\| (\pa_{x_i} u) \zeta_0 \psi_j \|_{L^\infty(B_1 \cap K_j)} 
& = \| ((\pa_{x_i} u) \zeta_0 \psi_j) \circ \mathtt g_j \|_{L^\infty(A_0 \cap \R^n_+)} 
\\ 
& \leq \| ((\pa_{x_i} u) \zeta_0 \psi_j) \circ \mathtt g_j \|_{L^\infty(\R^n_+)}
\\ 
& \leq C_{1,r_0} \| ((\pa_{x_i} u) \zeta_0 \psi_j) \circ \mathtt g_j \|_{H^{1,r_0}(\R^n_+)}
\end{align*}
with $1+r_0 > n/2$. 
Hence $\| \pa_{x_i} u \|_{L^\infty(B_1)} \leq C_{1,r_0} |\pa_{x_i} u|_{X^{1, r_0}_0}$.   
Then apply \eqref{pinoli.grad} and \eqref{pinoli.interpol.psi}.
\end{proof}

In the next lemma we estimate the multiplication of two functions. 

\begin{lemma} \label{lemma:pinoli.prod}
For all real $r,r_0$, with $r \geq 0$ and $r_0 + 1 > n/2$,
all integer $k \geq 0$, one has 
\begin{equation} \label{pinoli.prod}
|uv|_{X^{1,r}_k} \leq C_{r_0} |u|_{X^{1,r}_k} |v|_{X^{1,r_0}_{k+1}} 
+ (\chi_{r > r_0}) C_r |u|_{X^{1,r_0}_k} |v|_{X^{1,r}_{k+1}},
\end{equation}
where $C_{r_0}$ is independent of $r$, $C_r$ is increasing in $r$, 
and both are independent of $k$. 
The same holds for the $*$ seminorms. 
Moreover, for all integers $k, k' \geq 0$, 
\begin{equation} \label{pinoli.prod.mixed.*}
|uv|_{X^{1,r}_k} \leq C_{r_0} |u|_{X^{1,r}_k} |v|_{X^{1,r_0}_{*,k'}} 
+ (\chi_{r > r_0}) C_r |u|_{X^{1,r_0}_k} |v|_{X^{1,r}_{*,k'}}.
\end{equation}
\end{lemma}

\begin{proof}
We use \eqref{res.unity.in.the.annulus} and \eqref{zeta.k.is.1.where}, 
and proceed like in \eqref{echin.03}-\eqref{echin.04}.  
Thus 
\[
(u v \zeta_k \psi_j) \circ \mathtt g_j 
= \sum_{\ell \in \mA_j} (u v \zeta_k \zeta_{k+1} \psi_\ell \psi_j) \circ \mathtt g_j 
= w_{kj} \sum_{\ell \in \mA_j} z_{k+1,\ell} \circ \mathtt f_\ell \circ \mathtt g_j, 
\]
where $w_{kj} := (u \zeta_k \psi_j) \circ \mathtt g_j$ % is defined in \eqref{def.w.kj}, 
and $z_{k+1,\ell} := (v \zeta_{k+1} \psi_\ell) \circ \mathtt g_\ell$. 
By \eqref{prod.est}, 
\begin{align*}
\| w_{kj} (z_{k+1,\ell} \circ \mathtt f_\ell \circ \mathtt g_j) \|_{H^{1,r}(\R^n_+)} 
& \leq 
C_{r_0} \| w_{kj} \|_{H^{1,r}(\R^n_+)} 
\| z_{k+1,\ell} \circ \mathtt f_\ell \circ \mathtt g_j \|_{H^{1,r_0}(\R^n_+)} 
\\ & \quad \ 
+ (\chi_{r > r_0}) C_{r} \| w_{kj} \|_{H^{1,r_0}(\R^n_+)} 
\| z_{k+1,\ell} \circ \mathtt f_\ell \circ \mathtt g_j \|_{H^{1,r}(\R^n_+)}.
\end{align*}
By \eqref{composition.est.s.r}, 
\begin{align*}
\| z_{k+1,\ell} \circ \mathtt f_\ell \circ \mathtt g_j \|_{H^{1,\rho}(\R^n_+)} 
& \leq C_\rho \| z_{k+1,\ell} \|_{H^{1,\rho}(\R^n_+)},
\end{align*}
with $\rho = r$ and $\rho = r_0$. The sum over $j$ gives \eqref{pinoli.prod}. 
To prove \eqref{pinoli.prod.mixed.*}, recall \eqref{zeta.*.k.k'}, 
and use $\zeta_{k'}^*$ instead of $\zeta_{k+1}$, 
and $z_{k',\ell}^* := (v \zeta^*_{k'} \psi_\ell) \circ \mathtt g_\ell$
instead of $z_{k+1,\ell}$.
\end{proof}

Applying repeatedly the previous lemma, we estimate integer powers of a function.

\begin{lemma} \label{lemma:pinoli.power}
For any integer $m \geq 2$, 
any real $r, r_0$ with $r \geq 0$, $r_0 + 1 > n/2$, one has 
\begin{align} \label{pinoli.power}
|u^m|_{X^{1,r}_k} 
& \leq ( C_{r_0} |u|_{X^{1,r_0}_{k+1}} )^{m-1} |u|_{X^{1,r}_k} 
+ (m-1) (\chi_{r>r_0}) C_r (C_{r_0} |u|_{X^{1,r_0}_{k+1}})^{m-2} |u|_{X^{1,r_0}_k} |u|_{X^{1,r}_{k+1}},
\end{align}
where $C_{r_0}, C_r$ are the constants in Lemma \ref{lemma:pinoli.prod}.
The same holds for the $*$ seminorms. 
\end{lemma}

\begin{proof} 
By induction on $m$, using \eqref{pinoli.prod},
always putting $u^{m-1}$ with $k$, and $u$ with $k+1$.
\end{proof}

For the multiplication by Carthesian coordinates $x_i$, we have the following inequality. 

\begin{lemma} 
For all $i,i'=1, \ldots, n$, the multiplication by $x_i$ or by $x_i x_{i'}$ satisfies 
\begin{equation} \label{pinoli.prod.x.i}
|x_i u|_{X^{1,r}_k} + |x_i x_{i'} u|_{X^{1,r}_k} 
\leq C_{r_0} |u|_{X^{1,r}_k} 
+ (\chi_{r > r_0}) C_r |u|_{X^{1,r_0}_k},
\end{equation}
with $C_{r_0}$ independent of $r$, 
$C_r$ increasing in $r$, 
and both independent of $k$. 
The same holds for the $*$ seminorms. 
\end{lemma}

\begin{proof}
Follow the proof of Lemma \ref{lemma:pinoli.prod}, 
but for $v(x)=x_i$ or $x_i x_{i'}$ use $\zeta_*$ instead of $\zeta_{k+1}$, 
so that the first derivative of $\zeta_*$ is bounded by a constant 
independent of $k$.  
\end{proof}

\section{Analysis of the Dirichlet-Neumann operator}
\label{sec:DN}

In this section we prove Theorems \ref{thm.G.in.intro} and \ref{thm.der.G.in.intro}.
First, we want to transform the Laplace problem \eqref{problem.in.Om} into a problem in the unit ball $B_1$.
To this aim, we first use the Poisson integral operator to obtain a bijective map of $B_1$ onto $\Om$. 

\begin{lemma} \label{lemma:tilde.gamma}
For $1+r_0>n/2$, there exists $\delta_0$ with the following property.  
Let $h \in H^{r_0+\frac32}(\S^{n-1})$ with $\| h \|_{H^{r_0+\frac32}(\S^{n-1})} \leq \delta_0$, 
and let $\tilde h := \mathtt{PI} h$, where $\mathtt{PI}$ is the harmonic extension operator of 
Lemma \ref{lemma:def.harmonic.ext.op}. 
Then the map 
\begin{equation} \label{def.tilde.gamma}
\tilde \gamma(x) := x(1 + \tilde h(x)), \quad x \in B_1,
\end{equation}
is $C^\infty(B_1)$, it is a bijection of $B_1$ onto $\Om$, and 
\begin{equation} \label{est.tilde.gamma.id}
\| \tilde \gamma - \mathrm{id} \|_{W^{1,\infty}(B_1)} 
\leq C_{r_0} \| h \|_{H^{r_0+\frac32}(\S^{n-1})}.
\end{equation}
\end{lemma}

\begin{proof}
The function $\tilde h$ is $C^\infty(B_1)$ because it is harmonic in $B_1$. 
Therefore $\tilde \gamma : B_1 \to \R^n$ is also $C^\infty(B_1)$. 
Estimate \eqref{est.tilde.gamma.id} for $(\tilde \gamma - \mathrm{id})(x) = x \tilde h(x)$ 
follows from \eqref{pinoli.L.infty.bound.PI}. 
For $\delta_0$ sufficiently small, the RHS of \eqref{est.tilde.gamma.id} 
is $\leq C_{r_0} \delta_0 \leq \frac12$. 
For every $z \in \S^{n-1}$, 
$\tilde\gamma$ is a bijection of the segment $\{ tz : t \in [0,1] \} \subset \overline{B}_1$ 
onto the segment $\{ sz : s \in [0,1+h(z)] \} \subset \overline{\Om}$, 
because $\tilde \gamma(tz) = \ph(t) z$, where $\ph(t) :=  t (1 + \tilde h(tz))$, 
and $\ph(0) = 0$, $\ph(1) = 1+h(z)$, $\ph'(t) = 1 + \tilde h(tz) + t \la (\grad \tilde h)(tz) , z \ra$, 
$|\ph'(t) - 1| \leq \| \tilde h \|_{W^{1,\infty}(B_1)} \leq \frac12$ for all $t \in (0,1)$, 
so that $\ph$ is strictly increasing in $[0,1]$. 
\end{proof}

Using $\tilde \gamma$ as a change of variable, problem \eqref{problem.in.Om} becomes 
\begin{equation} \label{problem.in.B}
\div (P \grad u) = 0 \ \ \text{in } B_1, \quad \ 
u = \psi \ \ \text{on } \S^{n-1}, \quad \ 
u \in H^1(B_1),
\end{equation}
where $u = \Phi \circ \tilde \gamma$, 
$P$ is the $n \times n$ matrix
\begin{equation} \label{def.matrix.P}
P(x) = | \det D \tilde \gamma(x) | [ D \tilde \gamma (x) ]^{-1} [ D \tilde \gamma (x) ]^{-T},
\end{equation}
$D \tilde\gamma$ is the Jacobian matrix of $\tilde \gamma$, 
and $(D \tilde \gamma)^{-T}$ is the transpose of its inverse. 
By \eqref{def.tilde.gamma}, 
\[
D \tilde \gamma(x) = (1 + \tilde h(x)) I + x \otimes \grad \tilde h(x).
\]
Since $(a \otimes b)(a \otimes b) = \la a,b \ra a \otimes b$, one has 
\begin{align}
(D \tilde \gamma)^{-1} 
& = \frac{I}{1+\tilde h} 
- \frac{ x \otimes \grad \tilde h }{ (1 + \tilde h)(1 + \tilde h + \la x, \grad \tilde h \ra) }, 
\notag \\
(D \tilde \gamma)^{-T} 
& = \frac{I}{1+\tilde h} 
- \frac{ (\grad \tilde h) \otimes x }{ (1 + \tilde h)(1 + \tilde h + \la x, \grad \tilde h \ra) }.
\label{D.tilde.gamma.-T}
\end{align}
By the matrix determinant lemma, i.e., the formula $\det(I + a \otimes b) = 1 + \la a, b \ra$, 
we calculate 
\[
\det(D \tilde \gamma) = (1 + \tilde h)^{n-1} (1 + \tilde h + \la x, \grad \tilde h \ra).  
\]
Hence the matrix $P$ in \eqref{def.matrix.P} is 
\begin{align} \label{formula.P}
P & = (1+\tilde h)^{n-3} \bigg[ \big( 1 + \tilde h + \la x , \grad \tilde h \ra \big) I 
- (\grad \tilde h) \otimes x 
- x \otimes \grad \tilde h 
+ \frac{|\grad \tilde h|^2 x \otimes x}{ 1 + \tilde h + \la x , \grad \tilde h \ra } \bigg].
\end{align}

For $x \in \S^{n-1}$, recalling that $\la x , \grad_{\S^{n-1}} h \ra = 0$ 
and writing $\grad \tilde h$ as the sum
$\grad_{\S^{n-1}} h + \la x , \grad \tilde h \ra x$, 
one proves that, at the boundary point $\gamma(x) \in \pa \Om$,  
the unit normal vector to the boundary $\pa \Om$ is 
\begin{equation} \label{nu.Om}
\nu_\Om(\gamma(x)) = \frac{(1+h) x - \grad_{\S^{n-1}} h}{J}, \quad 
J := \sqrt{ (1+h)^2 + |\grad_{\S^{n-1}}h|^2 },
\end{equation}
and the gradient of $\Phi$ is 
\begin{equation} \label{grad.Phi}
(\grad \Phi)(\gamma(x)) = [D \tilde \gamma(x)]^{-T} \grad u(x).
\end{equation}
Since 
\[
\la \grad u , \grad_{\S^{n-1}} h \ra 
= \la \grad_{\S^{n-1}} u , \grad_{\S^{n-1}} h \ra 
= \la \grad_{\S^{n-1}} \psi , \grad_{\S^{n-1}} h \ra
\]
on $\S^{n-1}$, 
by \eqref{D.tilde.gamma.-T}, \eqref{nu.Om}, \eqref{grad.Phi}, 
we calculate that the Dirichlet-Neumann operator defined in \eqref{def.G.geom} is 
\begin{equation} \label{formula.G} 
G(h)\psi = \frac{J \la \grad u , x \ra }{(1+h)(1+h+\la x, \grad \tilde h \ra)} 
- \frac{ \la \grad_{\S^{n-1}} \psi, \grad_{\S^{n-1}} h \ra }{J (1+h)} 
\end{equation} 
on $\S^{n-1}$. To estimate $G(h)\psi$, we use formula \eqref{formula.G}, 
and we start with analyzing the solution $u$ of problem \eqref{problem.in.B}.

\begin{lemma} \label{lemma:P}
Let $r_0, \delta_0, h, \tilde h, \tilde \gamma$ be like in Lemma \ref{lemma:tilde.gamma}. 
Then $P \in C^0(B_1, \mathrm{Mat}_{n \times n}(\R))$.
The map $h \mapsto P = P(h)$ from the ball 
$\{ h \in H^{r_0+\frac32}(\S^{n-1}) : \| h \|_{H^{r_0+\frac32}(\S^{n-1})} < \delta_0 \}$ 
into $C^0(B_1, \mathrm{Mat}_{n \times n}(\R))$ is analytic.  
More precisely, $P$ is the sum of a convergent series 
\[
P = \sum_{m=0}^\infty P_m,
\]
where $P_0 = I$ and $P_m$ depends on $h$ in a $m$-homogeneous way, 
that is, $P_m(\lm h) = \lm^m P_m(h)$ for all $\lm \in \R$, 
and it coincides with a $m$-linear, continuous, symmetric map  
$\accentset{\circ}{P}_m : ( H^{r_0+\frac32}(\S^{n-1}) )^m \to C^0(\R^n, \mathrm{Mat}_{n \times n}(\R))$, 
$(h_1, \ldots, h_m) \mapsto \accentset{\circ}{P}_m [h_1, \ldots, h_m]$,
evaluated at $(h, \ldots, h)$, i.e., $P_m(h)$ $=$ $\accentset{\circ}{P}_m [h$, $\ldots, h ]$. 
The formula of $P_m$ depends on the dimension $n$: 
in dimension $n=2$ one has 
\[
P_m = \sum_{\begin{subarray}{c} \sigma, k \geq 0 \\ \sigma + k = m \end{subarray}} 
(-1)^\sigma \tilde h^\sigma Q_k,
\]
while, in dimension $n \geq 3$, 
\[
P_m = \sum_{\begin{subarray}{c} 0 \leq \sigma \leq n-3 \\ k \geq 0, \,  \sigma + k = m \end{subarray}} 
\binom{n-3}{\sigma} \tilde h^\sigma Q_k,
\]
where, in any dimension $n \geq 2$, $Q_0 = I$, 
\begin{align*}
Q_1 & = (\tilde h + \la x, \grad \tilde h \ra) I - (\grad \tilde h) \otimes x - x \otimes \grad \tilde h,
\\
Q_k & = (-1)^k |\grad \tilde h|^2 (\tilde h + \la x, \grad \tilde h \ra)^{k-2} x \otimes x, 
\quad k \geq 2.
\end{align*}
In any dimension $n \geq 2$, one has 
\begin{equation} \label{est.P.m.C.0} 
\| P_m \|_{C^0(\R^n)} \leq (C \| h \|_{H^{r_0+\frac32}(\S^{n-1})} )^m
\end{equation}
for all $m \in \N$, for some constant $C$ depending on $r_0$, independent of $m$. 
\end{lemma}

\begin{proof} 
The lemma easily follows from formula \eqref{formula.P} and Lemma \ref{lemma:PI.infty}.
Note that $(1 + \tilde h)^{n-3}$ is a Neumann series of $\tilde h$ if $n=2$, 
it is 1 if $n=3$, and it is a polynomial of $\tilde h$ if $n \geq 4$; 
similarly, $(1 + \tilde h + \la x, \grad \tilde h \ra)^{-1}$ is a Neumann series 
of $\tilde h + \la x, \grad \tilde h \ra$. 
\end{proof}

We want to expand the solution $u$ of the elliptic problem \eqref{problem.in.B} in powers of $h$. 
Thus, we consider the series $u = \sum_{m=0}^\infty u_m$ whose terms are defined 
in the following way: $u_0$ is the unique solution of the Laplace problem
\begin{equation} \label{prob.u.0}
u_0 \in H^1(B_1), \quad 
\Delta u_0 = 0 \ \  \text{in } B_1, \quad \ 
u_0 = \psi \ \ \text{on } \S^{n-1},
\end{equation} 
namely $u_0$ is the harmonic extension of $\psi$ to the open unit ball,  
and, for $m \geq 1$, $u_m$ is recursively defined as the unique solution of 
\begin{equation} \label{prob.u.n}
- \Delta u_m = \div g_m  \ \  \text{in } B_1, \quad \ 
u_m \in H^1_0(B_1)
\end{equation} 
with datum 
\begin{equation}  \label{def.g.n}
g_m := \sum_{k=0}^{m-1} P_{m-k} \grad u_k, 
\quad m \geq 1.
\end{equation}
In other words, in the notation of Lemmas 
\ref{lemma:def.harmonic.ext.op} and \ref{lemma:op.S}, 
\begin{equation} \label{pinoli.formula.u.m}
u_0 = \mathtt{PI} \psi, \quad \ 
u_m = \mathtt{S} g_m, \quad m \geq 1.
\end{equation}
By the standard theory of harmonic functions and elliptic PDEs 
(see Lemmas \ref{lemma:def.harmonic.ext.op} and \ref{lemma:op.S}), one has 
\begin{equation}  \label{est.basic.u.0.u.n}
\| u_0 \|_{H^1(B_1)} \leq C_0 \| \psi \|_{H^{\frac12}(\S^{n-1})}, 
\quad \ 
\| u_m \|_{H^1(B_1)} \leq C_1 \| g_m \|_{L^2(B_1)}, \quad m \geq 1,
\end{equation}
for some constants $C_0, C_1$.  
We deduce the following estimates. 

\begin{lemma} \label{lemma:olive.low}
The functions $u_m$ defined in \eqref{prob.u.0}, \eqref{prob.u.n}, \eqref{def.g.n} satisfy
\begin{equation} \label{est.u.n.induction}
\| u_m \|_{H^1(B_1)} 
\leq C_0 \| \psi \|_{H^{\frac12}(\S^{n-1})} (C' \| h \|_{H^{r_0+\frac32}(\S^{n-1})})^m 
\end{equation}
for all $m \geq 0$, where 
$C' = C(1+C_1)$  
and $C, C_0, C_1$ are the constants in \eqref{est.P.m.C.0}, \eqref{est.basic.u.0.u.n}. 
\end{lemma}

\begin{proof}
For $m=0$ the thesis holds by the first inequality in \eqref{est.basic.u.0.u.n}. 
Assume that, for a given $m \geq 1$, \eqref{est.u.n.induction} holds for $u_0, \ldots, u_{m-1}$. 
By the second inequality in \eqref{est.basic.u.0.u.n} and \eqref{def.g.n}, 
\begin{align} 
\| u_m \|_{H^1(B_1)} 
\leq C_1 \| g_m \|_{L^2(B_1)} 
& \leq C_1 \sum_{k=0}^{m-1} \| P_{m-k} \|_{L^\infty(B_1)} \| u_k \|_{H^1(B_1)}
\label{est.g.n}
\end{align}
and, by \eqref{est.P.m.C.0} and the induction assumption, 
this is bounded by the RHS of \eqref{est.u.n.induction}, 
because $C_1 \sum_{k=0}^{m-1} (1+C_1)^k < (1+C_1)^m$.
\end{proof}

From \eqref{est.u.n.induction} it follows that 
the series $\sum u_m$ converges totally in $H^1(B_1)$ norm, 
uniformly in $h$ in a ball of $H^{r_0+\frac32}(\S^{n-1})$ of radius sufficiently small.
Moreover, $u_m$ (as well as $g_m$ and, as already observed, $P_m$) 
depends on $h$ in a $m$-homogeneous way, 
namely $u_m(\lm h) = \lm^m u_m(h)$ for all $\lm \in \R$, 
and it coincides with a $m$-linear, continuous, symmetric map 
\begin{equation} \label{u.n.circ.multilinear.map}
\accentset{\circ}{u}_m : ( H^{r_0+\frac32}(\S^{n-1}) )^m \to H^1_0(B_1), \quad \ 
(h_1, \ldots, h_m) \mapsto \accentset{\circ}{u}_m [h_1, \ldots, h_m]
\end{equation}
evaluated at $(h, \ldots, h)$, i.e., $u_m(h) = \accentset{\circ}{u}_m [h, \ldots, h ]$. 
Explicitly, $g_m(h) = \accentset{\circ}{g}_m[h, \ldots, h]$, where  
\[
\accentset{\circ}{g}_m[ h_1, \ldots, h_m] 
= \frac{1}{m!} \sum_{\sigma \in \mathfrak{S}_m} 
\sum_{k=0}^{m-1} \accentset{\circ}{P}_{m-k} [h_{\sigma(k+1)}, \ldots, h_{\sigma(m)}] 
\grad \{ \accentset{\circ}{u}_k [h_{\sigma(1)}, \ldots, h_{\sigma(k)}] \},
\]
where $\mathfrak{S}_m$ is the set of the permutations of $\{ 1, \ldots, m \}$,  
and $\accentset{\circ}{u}_m [h_1, \ldots, h_m]$ is the solution of 
problem \eqref{prob.u.n} with $\accentset{\circ}{g}_m [h_1, \ldots, h_m]$ instead of $g_m$. 
Adapting the proof of Lemma \ref{lemma:olive.low}, one easily obtains  
\[
\| \accentset{\circ}{u}_m [h_1, \ldots, h_m] \|_{H^1(B_1)} 
\leq C_0 \| \psi \|_{H^{\frac12}(\S^2)} \prod_{j=1}^m (C' \| h_j \|_{H^{r_0+\frac32}(\S^{n-1})}).
\]
Moreover, since $u_m(h) = \accentset{\circ}{u}_m [h, \ldots, h]$, 
the derivative of $u_m$ with respect to $h$ in direction $\eta$ 
is $u_m'(h)[\eta] = m \accentset{\circ}{u}_m[h, \ldots, h, \eta]$, 
its second derivative is $u_m''(h)[\eta_1, \eta_2] 
= m (m-1) \accentset{\circ}{u}_m[h, \ldots, h, \eta_1, \eta_2]$, 
and so on. 
Thus, we have proved the following analyticity result.

\begin{proposition} 
\label{prop:est.u.low.norm}
Given $1+r_0 > n/2$, 
there exist positive constants $\delta_0, C$ with the following property. 
The map $h \mapsto u := \sum_{m=0}^\infty u_m$ from the ball
\begin{equation} \label{ball.h.W.1.infty} 
\{ h \in H^{r_0+\frac32}(\S^{n-1}) : \| h \|_{H^{r_0+\frac32}(\S^{n-1})} < \d_0 \}
\end{equation}
into $\mL (H^{\frac12}(\S^{n-1}) , H^1(B_1))$ is well-defined and analytic, 
and $u = u(h) = u(h,\psi)$ satisfies
\begin{align*}
\| u \|_{H^1(B_1)} 
& \leq C \| \psi \|_{H^{\frac12}(\S^{n-1})}, 
\\ 
\| u'(h)[\eta] \|_{H^1(B_1)} 
& \leq C \| \psi \|_{H^{\frac12}(\S^{n-1})} \| \eta \|_{H^{r_0+\frac32}(\S^{n-1})}, 
\\
\| u''(h)[\eta_1, \eta_2] \|_{H^1(B_1)} 
& \leq C \| \psi \|_{H^{\frac12}(\S^{n-1})} 
\| \eta_1 \|_{H^{r_0+\frac32}(\S^{n-1})} \| \eta_2 \|_{H^{r_0+\frac32}(\S^{n-1})}.
\end{align*}
\end{proposition}

Now we study the higher regularity of $u$, 
using the seminorms $| \ |_{X^{s,r}_k}$, $| \ |_{X^{s,r}_{*,k}}$ defined in \eqref{def.seminorm}. 
The matrices $P_m$, where $\tilde h$ is the harmonic extension of $h$, 
satisfy the following estimate. 

\begin{lemma}
Let $\tilde h = \mathtt{PI} h$. 
For all integers $m \geq 1$, $k \geq 0$, all real $r \geq r_0 > (n-2)/2$, 
one has 
\begin{align} 
|P_m|_{X^{1,r_0}_k} 
& \leq (C_{r_0,k} \| h \|_{H^{r_0+\frac32}(\S^{n-1})} )^m,
\label{pinoli.est.P.m.low}
\\
|P_m|_{X^{1,r}_k} 
& \leq C_{r,k} \, m ( C_{r_0,k} \| h \|_{H^{r_0+\frac32}(\S^{n-1})} )^{m-1} 
\| h \|_{H^{r+\frac32}(\S^{n-1})},
\label{pinoli.est.P.m.high}
\\
\| P_m \|_{L^\infty(B_1)} 
& \leq (C_{r_0} \| h \|_{H^{r_0+\frac32}(\S^{n-1})} )^m, 
\label{pinoli.est.P.m.infty}
\end{align}
where 
$C_{r_0,k}$ is independent of $r$, increasing in $k$;
$C_{r,k}$ is increasing in $r$ and $k$; 
$C_{r_0}$ is independent of $r,k$; 
and all three are independent of $m$. 
The same holds for the $*$ seminorms. 
\end{lemma}

\begin{proof} 
By \eqref{pinoli.grad}, 
\eqref{pinoli.prod.x.i}, 
\eqref{pinoli.interpol.psi},
one has 
\begin{equation} \label{benef.2}
|\grad \tilde h|_{X^{1,r_0}_k} + |\la x , \grad \tilde h \ra|_{X^{1,r_0}_k} 
\leq C_{r_0,k} (|\tilde h|_{X^{2,r_0}_k} + |\tilde h|_{X^{1,r_0}_{k+1}})
\leq C_{r_0,k}' \| h \|_{H^{r_0+\frac32}(\S^{n-1})}
\end{equation}
and 
\begin{align}
|\grad \tilde h|_{X^{1,r}_k} + |\la x , \grad \tilde h \ra|_{X^{1,r}_k} 
& \leq C_{r_0} |\tilde h|_{X^{2,r}_k} + C_{r,k} |\tilde h|_{X^{1,r}_{k+1}}
+ C_{r,k} (|\tilde h|_{X^{2,r_0}_k} + |\tilde h|_{X^{1,r_0}_{k+1}})
\notag \\ 
& \leq C_{r_0}' \| h \|_{H^{r+\frac32}(\S^{n-1})} 
+ C_{r,k}' \| h \|_{H^{r_0+\frac32}(\S^{n-1})},
\label{benef.3}
\end{align}
where we have used the inequality $|\tilde h|_{X^{1,r}_{k+1}} \leq |\tilde h|_{X^{2,r-1}_{k+1}}$
and the interpolation inequality \eqref{interpol.Hr.S2} like in \eqref{interpol.trick}.
Hence, by \eqref{pinoli.prod}, \eqref{pinoli.power} 
and the formula of $P_m$, we obtain 
\eqref{pinoli.est.P.m.low} and \eqref{pinoli.est.P.m.high}. 
The proof of \eqref{pinoli.est.P.m.infty} is similar, using \eqref{pinoli.L.infty.bound.PI}.
\end{proof}

\begin{lemma}
For all integers $m \geq 1$, $k \geq 0$, all real $r,r_0$ with $r \geq 0$, $1+r_0 > n/2$, 
all functions $u$, one has 
\begin{align} 
& |P_m \grad u|_{X^{1,r}_k} 
\leq  (C_{r_0} \| h \|_{H^{r_0+\frac32}(\S^{n-1})} )^m
(C_{r_0} |u|_{X^{2,r}_k} + C_r |u|_{X^{2,0}_k} + C_{r,k} |u|_{X^{1,r}_{k+1}}) 
\notag \\ 
& \qquad \ 
+ (\chi_{r>r_0}) C_{r,k} \, m (C_{r_0} \| h \|_{H^{r_0+\frac32}(\S^{n-1})} )^{m-1} 
\| h \|_{H^{r+\frac32}(\S^{n-1})} (|u|_{X^{2,r_0}_k} + |u|_{X^{1,r_0}_{k+1}}),
\label{est.P.m.grad.u}
\end{align}
where $C_{r_0}$ is independent of $r,k,m$, 
and $C_{r,k}$ is increasing in $r,k$, independent of $m$. 
\end{lemma}

\begin{proof}
Apply \eqref{pinoli.prod.mixed.*} with $k' = 0$, 
use \eqref{pinoli.grad} to estimate $|\grad u|_{X^{1,\rho}_k}$ with $\rho =r, r_0$, 
and \eqref{pinoli.est.P.m.low}, \eqref{pinoli.est.P.m.high} 
(for $*$ seminorms) to estimate $|P_m|_{X^{1,\rho}_{*,0}}$, $\rho = r, r_0$. 
The constant $C_{r_0,k'} = C_{r_0,0}$ 
coming from the application of \eqref{pinoli.est.P.m.low}, \eqref{pinoli.est.P.m.high} 
does not depend on $k$. 
\end{proof}

\begin{lemma} \label{lemma:est.u.n.r.k}
The functions $u_m$ in \eqref{prob.u.n} satisfy 
\begin{align} 
|u_m|_{X^{2,r}_k} 
& \leq ( B c_0 z_0 )^m
\big( \| \psi \|_{H^{r+\frac32}(\S^{n-1})} + A_{r,k} \| \psi \|_{H^{\frac32}(\S^{n-1})} \big)  
\notag \\ & \quad \ 
+ (\chi_{r > r_0}) A_{r,k} B^m ( c_0 z_0 )^{m-1} 
\| h \|_{H^{r+\frac32}(\S^{n-1})} \| \psi \|_{H^{r_0+\frac32}(\S^{n-1})},
\label{echin.45}
\end{align}
for all integers $m \geq 1$, $k \geq 0$, 
all real $r, r_0$ with $r \geq 0$, $r_0 + 1 > n/2$, 
where
\[
z_0 := \| h \|_{H^{r_0+\frac32}(\S^{n-1})}, 
\]
$c_0 := \max \{ \text{$C_{r_0}$ in \eqref{est.P.m.grad.u}, 
$C_{r_0}$ in \eqref{pinoli.est.P.m.infty}, 
$C'$ in \eqref{est.u.n.induction}} \}$,
the constants $B$ and $c_0$ are independent of $r,k$, 
while the constant $A_{r,k}$ is increasing in $r,k$. 
All $B,c_0,A_{r,k}$ are independent of $m$.
\end{lemma}

\begin{proof} 
\emph{First step}. 
We consider the case $r=0$. 
For $m=1$, by 
\eqref{pinoli.formula.u.m}, 
\eqref{def.g.n}, one has 
$u_1 = \mathtt{S} g_1$, $g_1 = P_1 \grad u_0$,  
and $u_0 = \mathtt{PI} \psi$. 
By 
\eqref{est.P.m.grad.u},
\eqref{pinoli.interpol.psi},  
\eqref{pinoli.est.P.m.infty},
\eqref{est.basic.u.0.u.n}, 
one has 
\begin{align} 
|P_1 \grad u_0|_{X^{1,0}_k} 
& \leq C_{r_0,k} c_0 z_0 \| \psi \|_{H^{\frac32}(\S^{n-1})}, 
\label{pinoli.50}
\\
\| P_1 \grad u_0 \|_{L^2(B_1)} 
& \leq C_0 c_0 z_0 \| \psi \|_{H^{\frac12}(\S^{n-1})},  
\label{pinoli.51}
\end{align}
and, by \eqref{pinoli.15},
\[
|u_1|_{X^{2,0}_k} \leq C_{r_0,k} c_0 z_0 \| \psi \|_{H^{\frac32}(\S^{n-1})}
\]
for some $C_{r_0,k}$ depending on $r_0,k$, 
where $c_0$ and $z_0$ are the constants defined in the statement of the lemma. 
Thus \eqref{echin.45} holds for $r=0, m=1$ if $B \geq 1$ and $A_{0,k} \geq C_{r_0,k}$.
Now assume that \eqref{echin.45} with $r=0$ holds for $u_1, \ldots, u_{m-1}$, for some $m \geq 2$. 
By \eqref{pinoli.formula.u.m}, \eqref{def.g.n}, one has 
$u_m = \mathtt{S} g_m$ and 
$g_m = \sum_{\sigma=0}^{m-1} P_{m-\sigma} \grad u_\sigma$. 
From \eqref{pinoli.15},
\eqref{est.P.m.grad.u},
\eqref{pinoli.est.P.m.infty} we deduce that 
\[
|u_m|_{X^{2,0}_k} 
\leq \sum_{\sigma = 0}^{m-1} (c_0 z_0)^{m-\sigma} 
(C_{r_0} |u_\sigma|_{X^{2,0}_k} + C_{r_0,k} |u_\sigma|_{X^{1,0}_{k+1}} 
+ C_{r_0,k} \| u_\sigma \|_{H^1(B_1)})
\]
for some $C_{r_0}$ depending on $r_0$, independent of $k,m$, 
and some $C_{r_0,k}$ depending on $r_0, k$, independent of $m$. 
Using the induction assumption \eqref{echin.45} for $|u_\sigma|_{X^{2,0}_k}$, $\sigma = 1, \ldots, m-1$, 
estimate \eqref{pinoli.interpol.psi} for $|u_0|_{X^{2,0}_k} = |\mathtt{PI} \psi|_{X^{2,0}_k}$, 
\eqref{est.u.X.low} for $|u_\sigma|_{X^{1,0}_{k+1}}$,  
and \eqref{est.u.n.induction} for $\| u_\sigma \|_{H^1(B_1)}$, we obtain 
\[
|u_m|_{X^{2,0}_k}  
\leq (c_0 z_0)^m \| \psi \|_{H^{\frac32}(\S^{n-1})} 
\big( C_{r_0} (1+A_{0,k}) B^{m-1} + m C_{r_0,k} \big)
\]
for $B \geq 2$, 
for some $C_{r_0}, C_{r_0,k}$ independent of $m$ (possibly different from above), 
because $\sum_{\sigma=0}^{m-1} B^\sigma \leq 2 B^{m-1}$ for all $B \geq 2$. 
Hence \eqref{echin.45} with $r=0$ holds for $u_m$ if 
\begin{equation}  \label{provided.01}
C_{r_0} (1+A_{0,k}) B^{m-1} \leq \frac12 (1+A_{0,k}) B^m, \quad \ 
m C_{r_0,k} \leq \frac12 (1+A_{0,k}) B^m.
\end{equation}
The first inequality in \eqref{provided.01} holds for $B \geq 2 C_{r_0}$, 
and, since $m \leq 2^m \leq B^m$, 
the second inequality in \eqref{provided.01} holds for $A_{0,k} \geq 2 C_{r_0,k}$. 
This proves that, for $A_{0,k}$ larger than some constant depending on $k$, 
and $B$ larger than some constant independent of $k$, 
\eqref{echin.45} with $r=0$ holds for all $m \in \N$. 

\medskip

\emph{Second step}. 
We consider the case $0 < r \leq r_0$. 
Let $r \in (0,1] \cap (0, r_0]$. 
For $m=1$, by \eqref{pinoli.16} with $b=1$, one has 
\begin{equation} \label{pinoli.52}
|u_1|_{X^{2,r}_k} 
\leq C_0 |P_1 \grad u_0|_{X^{1,r}_k}
+ C_{r,k} ( \| P_1 \grad u_0 \|_{L^2(B_1)} 
+ |P_1 \grad u_0|_{X^{1,0}_k} + |P_1 \grad u_0|_{X^{1,0}_{k+1}} ). 
\end{equation}
To estimate $|P_1 \grad u_0|_{X^{1,r}_k}$, we use  
\eqref{est.P.m.grad.u}, \eqref{pinoli.interpol.psi}
and the basic inequality 
$|u_0|_{X^{1,r}_{k+1}} 
\leq |u_0|_{X^{1,1}_{k+1}}
\leq |u_0|_{X^{2,0}_{k+1}}$, 
while the other terms in the RHS of \eqref{pinoli.52} 
are estimated in \eqref{pinoli.50}, \eqref{pinoli.51}. 
Thus, we get 
\begin{align*}
|u_1|_{X^{2,r}_k} 
& \leq c_0 z_0 (C_{r_0} \| \psi \|_{H^{r+\frac32}(\S^{n-1})} 
+ C_{r,k} \| \psi \|_{H^{\frac32}(\S^{n-1})} )
\end{align*}
for some $C_{r_0}, C_{r,k}$, 
and \eqref{echin.45} with $r \in (0,1] \cap (0,r_0]$ holds for $u_1$ 
if $B \geq C_{r_0}$ and $A_{r,k} \geq C_{r,k}$. 
Now assume that \eqref{echin.45} with $r \in (0,1] \cap (0,r_0]$ holds for $u_1, \ldots, u_{m-1}$, 
for some $m \geq 2$. 
By \eqref{pinoli.16} with $b=1$, one has
\begin{align*}
|u_m|_{X^{2,r}_k} 
& \leq C_0 \sum_{\sigma = 0}^{m-1} |P_{m-\sigma} \grad u_\sigma|_{X^{1,r}_k} 
\\ & \quad \ 
+ C_{r,k} \sum_{\sigma = 0}^{m-1} ( \| P_{m-\sigma} \grad u_\sigma \|_{L^2(B_1)}
+ |P_{m-\sigma} \grad u_\sigma|_{X^{1,0}_{k}}
+ |P_{m-\sigma} \grad u_\sigma|_{X^{1,0}_{k+1}} ),
\end{align*}
whence, by \eqref{est.P.m.grad.u}, \eqref{pinoli.est.P.m.infty}, \eqref{est.u.X.low}, 
\begin{align*}
|u_m|_{X^{2,r}_k} 
& \leq \sum_{\sigma = 0}^{m-1} (c_0 z_0)^{m-\sigma} \{ C_{r_0} |u_\sigma|_{X^{2,r}_k} 
+ C_{r,k} ( |u_\sigma|_{X^{2,0}_{k}} + |u_\sigma|_{X^{2,0}_{k+1}} + \| u_\sigma \|_{H^1(B_1)} ) \},
\end{align*}
where we have used the basic inequality 
$|u_\sigma|_{X^{1,r}_{k+1}} \leq |u_\sigma|_{X^{1,1}_{k+1}} 
\leq |u_\sigma|_{X^{2,0}_{k+1}}$. 
We use the induction assumption \eqref{echin.45} 
to estimate $|u_\sigma|_{X^{2,r}_k}$, $\sigma = 1, \ldots, m-1$; 
the already proved inequality \eqref{echin.45}$|_{r=0}$ 
to estimate $|u_\sigma|_{X^{2,0}_k}$ and $|u_\sigma|_{X^{2,0}_{k+1}}$, 
$\sigma = 1, \ldots, m-1$;
\eqref{pinoli.interpol.psi} to estimate 
$|u_0|_{X^{2,r}_k}$, 
$|u_0|_{X^{2,0}_k}$, 
$|u_0|_{X^{2,0}_{k+1}}$; 
and \eqref{est.u.n.induction} to estimate $\| u_\sigma \|_{H^1(B_1)}$.
We obtain 
\begin{align*}
|u_m|_{X^{2,r}_k} 
& \leq C_{r_0} B^{m-1} (c_0 z_0)^{m} \| \psi \|_{H^{r+\frac32}(\S^{n-1})} 
\\ & \quad \ 
+ \{ C_{r_0} A_{r,k} + C_{r,k} (1 + A_{0,k} + A_{0,k+1}) \} 
B^{m-1} (c_0 z_0)^{m} \| \psi \|_{H^{\frac32}(\S^{n-1})},  
\end{align*}
and this is bounded by the RHS of \eqref{echin.45} if 
$C_{r_0} \leq B$, 
$C_{r_0} A_{r,k} \leq \frac12 B A_{r,k}$, 
$C_{r,k} (1 + A_{0,k} + A_{0,k+1}) \leq \frac12 B A_{r,k}$, 
that is, if $B$ is larger than some constant independent of $r,k$, 
and $A_{r,k}$ is larger than some constant depending on $r,k$
(note that $A_{0,k}, A_{0,k+1}$ have been fixed at the previous step). 
Thus, \eqref{echin.45} with $r \in (0,1] \cap (0,r_0]$ holds for all $m \in \N$. 

Now we prove, by induction on $b$, 
that \eqref{echin.45} holds for all real $r \in (b-1, b] \cap (0,  r_0]$, 
all $m \in \N$, all $b \in \N$ such that $(b-1, b] \cap (0,  r_0]$ is nonempty.
For $b=1$ we have just proved it. 
Let $b \in \N$, $b \geq 2$, 
and assume that, for all $\b = 1, \ldots, b - 1$, 
the interval $(\b-1, \b] \cap (0, r_0]$ is nonempty 
and \eqref{echin.45} holds for all $r \in (\b-1, \b] \cap (0, r_0]$, all $m \in \N$. 
Suppose that $(b - 1, b] \cap (0, r_0]$ is nonempty, 
and let $r \in (b - 1, b] \cap (0, r_0]$. 
For $m=1$, by \eqref{pinoli.16}, 
\begin{align} 
|u_1|_{X^{2,r}_k} 
& \leq C_0 |P_1 \grad u_0|_{X^{1,r}_k}
+ C_{r,k} \Big( \| P_1 \grad u_0 \|_{L^2(B_1)}
+ |P_1 \grad u_0|_{X^{1,0}_k}
\notag \\ & \quad \ 
+ |P_1 \grad u_0|_{X^{1,0}_{k+b}}
+ \sum_{\ell=1}^{b-1} |P_1 \grad u_0|_{X^{1,r-\ell}_{k+\ell}} \Big).
\label{echin.49}
\end{align}
By \eqref{est.P.m.grad.u}, \eqref{pinoli.interpol.psi}, one has 
\begin{align*}
|P_1 \grad u_0|_{X^{1,r}_k} 
& \leq c_0 z_0 ( C_0 \| \psi \|_{H^{r + \frac32}(\S^{n-1})} 
+ C_{r,k} \| \psi \|_{H^{\frac32}(\S^{n-1})} ),
\\
|P_1 \grad u_0|_{X^{1, r - \ell}_{k+\ell}} 
& \leq c_0 z_0 ( C_0 \| \psi \|_{H^{r - \ell + \frac32}(\S^{n-1})} 
+ C_{r,k,\ell} \| \psi \|_{H^{\frac32}(\S^{n-1})} ),
\end{align*}
where we have used the inequalities  
\[
|u_0|_{X^{1,r}_{k+1}} 
\leq |u_0|_{X^{2,r-1}_{k+1}}, 
\quad 
C_{r,k} \| \psi \|_{H^{r - 1 + \frac32}(\S^{n-1})} 
\leq \| \psi \|_{H^{r + \frac32}(\S^{n-1})} + C_{r,k}' \| \psi \|_{H^{\frac32}(\S^{n-1})}
\]
(see \eqref{interpol.trick}), and similarly for the terms with $r - \ell$.
The terms $|P_1 \grad u_0|_{X^{1,0}_{k}}$, $|P_1 \grad u_0|_{X^{1,0}_{k+b}}$ 
and $\| P_1 \grad u_0 \|_{L^2(B_1)}$ are estimated in \eqref{pinoli.50}, \eqref{pinoli.51}.  
Hence \eqref{echin.49} gives
\begin{equation} \label{pinoli.60}
|u_1|_{X^{2,r}_k} 
\leq c_0 z_0 ( C_0 \| \psi \|_{H^{r + \frac32}(\S^{n-1})} 
+ C_{r,k} \| \psi \|_{H^{\frac32}(\S^{n-1})} ),
\end{equation}
and \eqref{echin.45} with $r \in (b-1,b] \cap (0, r_0]$ holds for $u_1$ 
if $B \geq C_0$ and $A_{r,k} \geq C_{r,k}$.  
Now assume that \eqref{echin.45} with $r \in (b - 1, b] \cap (0, r_0]$ 
holds for $u_1, \ldots, u_{m-1}$, for some $m \geq 2$. 
By \eqref{pinoli.16}, 
\begin{align*}
|u_m|_{X^{2,r}_k} 
& \leq C_0 \sum_{\sigma = 0}^{m-1} |P_{m-\sigma} \grad u_\sigma|_{X^{1,r}_k} 
+ C_{r,k} \sum_{\sigma = 0}^{m-1} \Big( \| P_{m-\sigma} \grad u_\sigma \|_{L^2(B_1)}
+ |P_{m-\sigma} \grad u_\sigma|_{X^{1,0}_{k}}
\\ & \quad \ 
+ |P_{m-\sigma} \grad u_\sigma|_{X^{1,0}_{k+b}} 
+ \sum_{\ell=1}^{b-1} |P_{m-\sigma} \grad u_\sigma|_{X^{1,r-\ell}_{k+\ell}} \Big),
\end{align*}
whence, by \eqref{est.P.m.grad.u}, \eqref{pinoli.est.P.m.infty}, \eqref{est.u.X.low}, 
\begin{align*}
|u_m|_{X^{2,r}_k} 
& \leq \sum_{\sigma = 0}^{m-1} (c_0 z_0)^{m-\sigma} 
\Big\{ C_{r_0} |u_\sigma|_{X^{2,r}_k} 
+ C_{r,k} \Big( \| u_\sigma \|_{H^1(B_1)} 
+ |u_\sigma|_{X^{2,0}_{k}} 
\\ & \quad \ 
+ |u_\sigma|_{X^{2,r-1}_{k+b}}
+ \sum_{\ell=1}^{b-1} |u_\sigma|_{X^{2,r-\ell}_{k+\ell}} \Big) \Big\}.
\end{align*}
Using \eqref{echin.45} for $u_1, \ldots, u_{m-1}$, 
\eqref{pinoli.interpol.psi} for $u_0$,
\eqref{est.u.n.induction} for $u_0, \ldots, u_{m-1}$, 
the interpolation inequality \eqref{interpol.trick}, 
and the inequalities $\sum_{\sigma=1}^{m-1} B^\sigma \leq 2 B^{m-1}$, 
$m \leq B^{m-1}$ for $B \geq 2$, we obtain 
\begin{equation} \label{pinoli.61}
|u_m|_{X^{2,r}_k} 
\leq B^{m-1} (c_0 z_0)^m 
\{ C_{r_0} \| \psi \|_{H^{r + \frac32}(\S^{n-1})} 
+ (C_{r_0} A_{r,k} + C_{r,k,A}) \| \psi \|_{H^{\frac32}(\S^{n-1})} \},
\end{equation}
where $C_{r_0}$ is independent of $r,b,m,k$, 
while $C_{r,k,A}$ is independent of $m$, depending on $r,k$ and on the constants 
$A_{0,k}$, $A_{r-1,k+b}$, 
$A_{r-\ell, k + \ell}$, 
$\ell = 1, \ldots, b-1$, 
all of which have already been fixed. 
Hence \eqref{echin.45} with $r \in (b-1,b] \cap (0,r_0]$ holds for $u_m$ if 
$C_{r_0} B^{m-1} \leq B^m$, 
$C_{r_0} A_{r,k} B^{m-1} \leq \frac12 A_{r,k} B^m$,
and $C_{r,k,A} B^{m-1} \leq \frac12 A_{r,k} B^m$. 
These conditions are satisfied for $B \geq 2 C_{r_0}$ 
and $A_{r,k} \geq 2 C_{r,k,A}$. 
This completes the proof of the induction step. 
Hence \eqref{echin.45} holds for all real $r \in (0,r_0]$, all $m \in \N$. 

\medskip

\emph{Third step}. It remains to prove that \eqref{echin.45} holds for $r > r_0$. 
We prove, by induction on $b$, that \eqref{echin.45} holds for all real $r \in (b-1,b] \cap (r_0,\infty)$,
all $m \in \N$, all $b \in \N$ such that $(b-1,b] \cap (r_0,\infty)$ is nonempty. 
Let $b_0$ be the minimum of such integers, that is, let $b_0$ be the integer part of $1+r_0$. 
Let $r \in (b_0-1, b_0] \cap (r_0,\infty)$. 
Proceeding like in the second step, now with $(\chi_{r>r_0}) = 1$, we get 
\begin{align*}
|u_1|_{X^{2,r}_k} & \leq \eqref{pinoli.60} 
+ C_{r,k} \| h \|_{H^{r+\frac32}(\S^{n-1})} \| \psi \|_{H^{r_0+\frac32}(\S^{n-1})},
\\
|u_m|_{X^{2,r}_k} & \leq \eqref{pinoli.61} 
+ C_{r,k} (1 + A_{r_0,k} + A_{r_0, k+1}) (Bc_0z_0)^{m-1}
\| h \|_{H^{r+\frac32}(\S^{n-1})} \| \psi \|_{H^{r_0+\frac32}(\S^{n-1})}, 
\end{align*}
where we have also used the estimates already proved for $r=r_0$
and the inequality $\sum_{\sigma=1}^{m-1} (m-\sigma) B^\sigma \leq 4 B^{m-1}$, 
which holds for all real $B \geq 2$, all integers $m \geq 2$. 
We choose the constant $A_{r,k}$ similarly as above, 
and we get \eqref{echin.45} for $r \in (b_0-1,b_0] \cap (r_0,\infty)$. 
The induction step is similar. 
\end{proof}

In the next lemma we observe that what we have proved for the function $u_m$ 
can be easily adapted to the $m$-linear map $\accentset{\circ}{u}_m$
in \eqref{u.n.circ.multilinear.map}.

\begin{lemma} \label{lemma:est.u.n.r.k.multlinear}
The $m$-linear map $\accentset{\circ}{u}_m$ in \eqref{u.n.circ.multilinear.map} satisfies  
\begin{align} 
& | \accentset{\circ}{u}_m [h_1, \ldots, h_m] |_{X^{2,r}_k} 
\leq B^m c_0^m \Big( \prod_{j=1}^m \| h_j \|_{H^{r_0+\frac32}(\S^{n-1})} \Big)   
\big( \| \psi \|_{H^{r+\frac32}(\S^{n-1})} 
+ A_{r,k} \| \psi \|_{H^{\frac32}(\S^{n-1})} \big)  
\notag \\ & \qquad 
+ (\chi_{r > r_0}) A_{r,k} B^m c_0^{m-1} 
\frac{1}{m} \Big( \sum_{j=1}^m \| h_j \|_{H^{r+\frac32}(\S^{n-1})} 
\prod_{i \neq j} \| h_i \|_{H^{r_0+\frac32}(\S^{n-1})} \Big)   
\| \psi \|_{H^{r_0 + \frac32}(\S^{n-1})}
\label{echin.45.multilinear}
\end{align}
for all $h_1, \ldots, h_m$, 
all $m \in \N$, $k \in \N_0$, all real $r, r_0$ with $r \geq 0$ and $1+r_0>n/2$. 
The constants $B, c_0, A_{r,k}, c_0$ are those of Lemma \ref{lemma:est.u.n.r.k}.
\end{lemma}

\begin{proof} The proof is a simple adaptation of the previous lemmas. 
\end{proof}

For functions defined on the sphere $\S^{n-1}$, 
we use the short notation \eqref{short.norm}.

\begin{proposition} \label{prop:u}
Recall notation \eqref{short.norm}. 
Let $1+r_0 > n/2$. 
There exists $\delta_0 > 0$ such that, 
for $\| h \|_{r_0+\frac32} \leq \delta_0$,
the series $u = \sum_{m=0}^\infty u_m$, 
with $u_m$ in \eqref{prob.u.0}, \eqref{prob.u.n}, 
satisfies 
\begin{align} 
|u|_{X^{2,r}_k} 
& \leq C_0 \| \psi \|_{r+\frac32}
+ C_{r,k} \| \psi \|_{\frac32}
+ (\chi_{r > r_0}) C_{r,k} \| h \|_{r+\frac32} \| \psi \|_{r_0 + \frac32}
\label{echin.60}
\end{align}
for all $k \in \N_0$, all real $r \geq 0$, 
where $C_{r,k}$ is increasing in $r,k$, 
while $C_0, \delta_0$ depend on $r_0$ and they are independent of $r,k$.
The map $h \mapsto u$ is analytic from the set 
\[
\{ h \in H^{r+\frac32}(\S^{n-1}) \cap H^{r_0+\frac32}(\S^{n-1}) 
: \| h \|_{r_0+\frac32} < \delta_0 \}
\] 
to the Banach space of the bounded linear operators  
of $H^{r+\frac32}(\S^{n-1})$ 
into the space $\{ u \in H^1(B_1) : |u|_{X^{2,r}_k} < \infty \}$ 
endowed with the norm $\| u \|_{H^1(B_1)} + | u |_{X^{2,r}_k}$.  
The derivative of $u$ with respect to $h$ in direction $\eta$ satisfies 
\begin{align} 
|u'(h)[\eta]|_{X^{2,r}_k} 
& \leq C_0 \| \eta \|_{r_0+\frac32}
( \| \psi \|_{r+\frac32} + C_{r,k} \| \psi \|_{\frac32} )
\notag \\ & \quad 
+ (\chi_{r > r_0}) C_{r,k} ( \| \eta \|_{r+\frac32} 
+ \| \eta \|_{r_0+\frac32} \| h \|_{r+\frac32} ) \| \psi \|_{r_0 + \frac32},
\label{echin.60.lin}
\end{align}
and the second derivative satisfies 
\begin{align}
| u''(h)[\eta_1, \eta_2] |_{X^{2,r}_k} 
& \leq C_0 \| \eta_1 \|_{r_0+\frac32} \| \eta_2 \|_{r_0+\frac32}
( \| \psi \|_{r+\frac32} + C_{r,k} \| \psi \|_{\frac32} )
\notag \\ & \quad \  
+ (\chi_{r > r_0}) C_{r,k} ( \| \eta_1 \|_{r+\frac32} \| \eta_2 \|_{r_0+\frac32} 
+ \| \eta_1 \|_{r_0+\frac32} \| \eta_2 \|_{r+\frac32}
\notag\\ & \quad \ 
+ \| \eta_1 \|_{r_0+\frac32} \| \eta_2 \|_{r_0+\frac32} \| h \|_{r+\frac32} )
\| \psi \|_{r_0+\frac32}.
\label{echin.60.second.der}
\end{align}
\end{proposition}

\begin{proof} 
To prove \eqref{echin.60}, use 
\eqref{pinoli.interpol.psi}, 
\eqref{echin.45} 
with $B c_0 z_0 \leq 1/2$,
and fix $\delta_0 := 1 / (2 B c_0)$. 
To prove \eqref{echin.60.lin} and \eqref{echin.60.second.der}, 
recall that 
$u_m'(h)[\eta] = m \accentset{\circ}{u}_m [h, \ldots, h, \eta]$ and  
$u_m''(h)[\eta_1, \eta_2] = m(m-1) \accentset{\circ}{u}_m [h, \ldots, h, \eta_1, \eta_2]$ 
and use \eqref{echin.45.multilinear}.
\end{proof}

\begin{lemma} \label{lemma:est.grad.u.x}
Recall notation \eqref{short.norm}. 
Let $r_0, \delta_0$ be like in Proposition \ref{prop:u}.
For $\| h \|_{r_0+\frac32} \leq \delta_0$, 
the function $u$ in Proposition \ref{prop:u} satisfies 
\begin{align} 
\| (\grad u)|_{\S^{n-1}} \|_{r+\frac12}
& \leq C_0 \| \psi \|_{r+\frac32} 
+ C_r \| \psi \|_{\frac32} 
+ (\chi_{r > r_0}) C_r \| h \|_{r+\frac32} 
\| \psi \|_{r_0+\frac32}
\notag \\ 
& \leq C_0' \| \psi \|_{r+\frac32} 
+ (\chi_{r > r_0}) C_r (1 + \| h \|_{r+\frac32}) \| \psi \|_{r_0+\frac32}
\label{echin.61}
\end{align}
for all real $r \geq 0$, 
where $C_r$ is increasing in $r$, 
while $C_0, C_0', \delta_0$ are independent of $r$.
Also, 
\begin{equation} \label{echin.62}
\| \la \grad u , x \ra|_{\S^{n-1}} \|_{r+\frac12} 
\leq C_0 \| \psi \|_{r+\frac32} 
+ (\chi_{r > r_0}) C_r (1 + \| h \|_{r+\frac32}) \| \psi \|_{r_0+\frac32}.
\end{equation}
Moreover, 
\begin{align} 
& \| (\grad \{ u'(h)[\eta] \})|_{\S^{n-1}} \|_{r+\frac12}
+ \| \la \grad \{ u'(h)[\eta] \} , x \ra|_{\S^{n-1}} \|_{r+\frac12}
\notag \\ 
& \qquad  
\leq C_0 \| \eta \|_{r_0+\frac32} \| \psi \|_{r+\frac32}
+ (\chi_{r > r_0}) C_r (\| \eta \|_{r+\frac32} 
+ \| h \|_{r+\frac32} \| \eta \|_{r_0+\frac32}) 
\| \psi \|_{r_0+\frac32}
\label{echin.61.lin}
\end{align}
and 
\begin{align}
& \| (\grad \{ u''(h)[\eta_1, \eta_2] \})|_{\S^{n-1}} \|_{r+\frac12}
+ \| \la \grad \{ u''(h)[\eta_1, \eta_2] \} , x \ra|_{\S^{n-1}} \|_{r+\frac12}
\notag \\ 
& \qquad 
\leq C_0 \| \eta_1 \|_{r_0+\frac32} \| \eta_2 \|_{r_0+\frac32} \| \psi \|_{r+\frac32}
+ (\chi_{r > r_0}) C_r 
(\| \eta_1 \|_{r+\frac32} \| \eta_2 \|_{r_0+\frac32}
\notag \\ & \qquad \quad \  
+ \| \eta_1 \|_{r_0+\frac32} \| \eta_2 \|_{r+\frac32}
+ \| h \|_{r+\frac32} \| \eta_1 \|_{r_0+\frac32} \| \eta_2 \|_{r_0+\frac32}  )
\| \psi \|_{r_0+\frac32}.
\label{echin.61.second.der}
\end{align}
\end{lemma}

\begin{proof} 
By \eqref{def.rho.k.zeta.k}, $\zeta_0=1$ on $\S^{n-1}$.  
By \eqref{def.norm.u}, 
\eqref{trace.est} with $s=1$, 
and \eqref{def.seminorm}, one has  
\begin{align}
\| (\grad u)|_{\S^{n-1}} \|_{H^{r+\frac12}(\S^{n-1})} 
& = \| (\zeta_0 \grad u)|_{\S^{n-1}} \|_{H^{r+\frac12}(\S^{n-1})}  
\notag \\ & = \sum_{j=1}^N \| (\psi_j \zeta_0 \grad u) \circ \mathtt g_j( \cdot , 0) \|_{H^{r+\frac12}(\R^{n-1})}
\notag \\ & \leq C \sum_{j=1}^N \| (\psi_j \zeta_0 \grad u) \circ \mathtt g_j \|_{H^{1,r}(\R^n_+)}
= C |\grad u|_{X^{1,r}_0}.
\label{benef.1}
\end{align}
By \eqref{pinoli.grad}, 
\eqref{echin.60} with $k=0$, 
the inequality $|u|_{X^{1,r}_k} \leq |u|_{X^{2,r-1}_k}$ for $k=1$, 
and using the interpolation inequality \eqref{interpol.Hr.S2} to estimate $\| \psi \|_{r-1+\frac32}$ 
like in \eqref{interpol.trick} if $r-1 > 0$,
we obtain the first inequality in \eqref{echin.61}. 
The second inequality in \eqref{echin.61} is obtained from the first one because 
\begin{equation} \label{echin.66}
C_r \| \psi \|_{\frac32}
\leq C_{r_0} \| \psi \|_{r + \frac32}
+ (\chi_{r > r_0}) C_r \| \psi \|_{r_0 + \frac32}
\end{equation}
for all $r \geq 0$. To prove \eqref{echin.66}, 
consider the two cases $r \in [0, r_0]$ and $r \in (r_0,\infty)$ separately, 
and use the fact that $C_r \leq C_{r_0}$ for $r \in [0, r_0]$. 
To prove \eqref{echin.62}, use \eqref{echin.61} 
and the product estimate \eqref{prod.est.unified.Hr.S2}. 
The proof of \eqref{echin.61.lin}, \eqref{echin.61.second.der} is similar. 
\end{proof}

Now that $\la \grad u , x \ra|_{\S^{n-1}}$ has been estimated,  
we consider the other terms appearing in \eqref{formula.G}.

\begin{lemma} \label{lemma:other.terms}
Recall notation \eqref{short.norm}. 
For any $s_0 > (n-1)/2$ there exists $\delta_0 > 0$ such that, 
for all $\| h \|_{s_0+1} \leq \delta_0$, all $s \geq s_0$, one has 
\begin{align*}
\| (1+h)^{-1} \|_s 
& \leq C_s (1 + \| h \|_{s}), 
\\
\| \{ (1+h)^2 + | \grad_{\S^{n-1}} h |^2 \}^{\pm \frac12} \|_{s} 
& \leq C_s (1 + \| h \|_{s+1}), 
\\
\| (1 + h + \la \grad \tilde h, x \ra)^{-1} \|_s
& \leq C_s (1 + \| h \|_{s+1}), 
\\
\| \la \grad_{\S^{n-1}} h, \grad_{\S^{n-1}} \psi \ra \|_{s} 
& \leq C_{s_0} \| h \|_{s_0+1} \| \psi \|_{s+1}
+ (\chi_{s > s_0}) C_s \| h \|_{s+1} \| \psi \|_{s_0+1},
\end{align*}
where $\tilde h = \mathtt{PI} h$. 
The constant $C_s$ is increasing in $s$,
while $C_{s_0}, \delta_0$ are independent of $s$. 
\end{lemma}

\begin{proof}
To estimate $(1+h)^{-1}$, we write it as a power series around $h=0$ 
and apply estimate \eqref{power.est.Hr.S2} to each term of the series. 
We estimate $(1+g)^p$, with $p = \pm \frac12$, 
$g = 2h + h^2 + |\grad_{\S^{n-1}} h|^2$, 
in the same way, then we estimate $g$ by \eqref{prod.est.unified.Hr.S2}, \eqref{est.grad.Hr.S2}.
Also $(1+g)^{-1}$, with $g = h + \la \grad \tilde h, x \ra$, 
is estimated in the same way, and,  
applying \eqref{benef.1} to $\la \grad \tilde h, x \ra$ in place of $\grad u$, 
then using \eqref{benef.2}, \eqref{benef.3}, 
one has 
\[
\| \la \grad \tilde h , x \ra|_{\S^{n-1}} \|_{s} 
\leq C_{s_0} \| h \|_{s+1} + (\chi_{s>s_0}) C_s \| h \|_{s_0+1}.
\]
The estimate for $\la \grad_{\S^2} h, \grad_{\S^2} \psi \ra$ follows from 
\eqref{prod.est.unified.Hr.S2}, \eqref{est.grad.Hr.S2}.
\end{proof}

\begin{proof}[\textbf{Proof of Theorems \ref{thm.G.in.intro}, \ref{thm.der.G.in.intro}}]
To prove \eqref{est.G.in.intro}, 
we use identity \eqref{formula.G},
estimate \eqref{echin.62} 
with $r + \frac12 = s$, $r_0 + \frac12 = s_0$ 
for the term $\la \grad u , x \ra$ on $\S^{n-1}$, 
Lemma \ref{lemma:other.terms} for the other terms appearing in \eqref{formula.G}, 
and the product estimate \eqref{prod.est.unified.Hr.S2}. 
The proof of \eqref{echin.63.lin.in.intro} and \eqref{echin.63.second.der.in.intro} 
is similar, using estimates \eqref{echin.61.lin} and \eqref{echin.61.second.der}.  
\end{proof}

\section{Appendix} 
\label{sec:appendix}

This Appendix contains the proofs of Lemmas \ref{lemma:est.tools.R.n.+}, 
\ref{lemma:est.Lap.R.n.+},
\ref{lemma:est.tools.S.n-1}. 

\begin{proof}[\textbf{Proof of Lemma \ref{lemma:est.tools.R.n.+}}] 
$(i)$ \emph{Proof of density}. Let $s,r \in \R$.  
Since $1 \leq 1 + |\xi'|^2 \leq 1 + |\xi|^2$, one has 
\[
(1 + |\xi|^2)^{r_0} \leq (1 + |\xi'|^2)^r \leq (1 + |\xi|^2)^{r_1} \quad \forall \xi \in \R^n, 
\]
where $r_0 := \min \{ 0, r \}$ and $r_1 := \max \{ 0, r \}$.
Hence 
\begin{equation} \label{inscatolati}
\| u \|_{H^{s + r_0}(\R^n)} \leq \| u \|_{H^{s,r}(\R^n)} \leq \| u \|_{H^{s + r_1}(\R^n)}
\end{equation}
for all $u \in H^{s + r_1}(\R^n)$. 
By Fourier truncation, it is immediate to prove that 
$H^{s + r_1}(\R^n)$ is dense in $H^{s,r}(\R^n)$: 
given $u \in H^{s,r}(\R^n)$, 
let $\hat u_N = \hat u$ in the ball $|\xi| < N$, 
and $\hat u_N = 0$ outside that ball. 
Let $u_N$ be the Fourier anti-transform of $\hat u_N$. 
Then $u_N \in H^\s(\R^n)$ for all $\s \in \R$,  
and $\| u - u_N \|_{H^{s,r}(\R^n)}$ tends to zero as $N \to \infty$. 
Given $\e > 0$, fix $N$ such that $\| u - u_N \|_{H^{s,r}(\R^n)} < \e/2$. 
As is known (see, e.g., \cite{Triebel.1983}, Theorem 2.3.3, page 48), 
$\mS(\R^n)$ is dense in $H^\s(\R^n)$ for all $\s \in \R$.  
Hence there exists $\ph \in \mS(\R^n)$ such that $\| u_N - \ph \|_{H^{s+r_1}(\R^n)} < \e/2$. 
By the second inequality in \eqref{inscatolati} and triangular inequality, 
one has $\| u - \ph \|_{H^{s,r}(\R^n)} < \e$. 
This proves that $\mS(\R^n)$ is dense in $H^{s,r}(\R^n)$. 

On $\R^n_+$, let $u \in H^{s,r}(\R^n_+)$, and let $v \in H^{s,r}(\R^n)$ 
with $v |_{\R^n_+} = u$. Given $\e > 0$, take $\ph \in \mS(\R^n)$ such that 
$\| v - \ph \|_{H^{s,r}(\R^n)} < \e$, and let $f := \ph |_{\R^n_+}$. 
Since $v - \ph \in H^{s,r}(\R^n)$ and $(v-\ph)|_{\R^n_+} = u-f$,  
one has
\[
\| u - f \|_{H^{s,r}(\R^n_+)} 
= \inf \{ \| w \|_{H^{s,r}(\R^n)} : w \in H^{s,r}(\R^n), \ w |_{\R^n_+} = u-f \} 
\leq \| v - \ph \|_{H^{s,r}(\R^n)} < \e.
\]
This proves that the set $\{ f : f = \ph |_{\R^n_+} \ \exists \ph \in \mS(\R^n) \}$ 
is dense in $H^{s,r}(\R^n_+)$.

\medskip

$(ii)$ \emph{Proof of \eqref{embedding.ineq}}. 
Let $v$ be a function in the Schwartz class $\mS(\R^n)$, 
and let $\hat v$ be its Fourier transform on $\R^n$.  
By Fourier inversion formula,
$\| v \|_{L^\infty(\R^n)} \leq \| \hat v \|_{L^1(\R^n)}$.  
Multiplying and dividing $\hat v(\xi)$ by 
$(1 + |\xi|^2)^{\frac{s}{2}} (1 + |\xi'|^2)^{\frac{r}{2}}$, 
and using Cauchy-Schwarz inequality, 
one has  
\[
\| \hat v \|_{L^1(\R^n)} 
\leq C_{s,r} \| v \|_{H^{s,r}(\R^n)}, 
\quad  
C_{s,r}^2 = \int_{\R^n} \frac{1}{ (1 + |\xi|^2)^s (1 + |\xi'|^2)^r } \, d\xi.
\]
For any fixed $\xi' \in \R^{n-1}$, 
the change of variable $\xi_n = (1 + |\xi'|^2)^{\frac12} t$ gives
$1 + |\xi|^2 = (1 + |\xi'|^2)(1+t^2)$ and
\begin{equation} \label{int.on.R.like.arctan}
\int_\R \frac{1}{ (1 + |\xi|^2)^s } \, d \xi_n 
= \frac{ K_s }{ (1 + |\xi'|^2)^{s - \frac12} }, \quad 
K_s := \int_\R \frac{1}{(1 + t^2)^s} \, dt.
\end{equation}
The integral $K_s$ converges for $2 s > 1$. 
Hence the integral $C_{s,r}^2$ is finite if $2(r+s-\frac12) > n-1$, i.e., $r+s > n/2$.
The assumption $v \in \mS(\R^n)$ is removed by density.
This proves \eqref{embedding.ineq} on $\R^n$. 
The uniform continuity in $\R^n$ of $v \in H^{s,r}(\R^n)$ 
and its vanishing limit as $|x| \to \infty$ 
hold because $\hat v \in L^1(\R^n)$, using the dominated convergence theorem, 
Riemann-Lebesgue lemma, and the formula for the Fourier transform of $\hat v$. 

On $\R^n_+$, let $u \in H^{s,r}(\R^n_+)$, and let $v \in H^{s,r}(\R^n)$ with $v |_{\R^n_+} = u$. 
Then $u$ is uniformly continuous in $\R^n_+$, 
$u \to 0$ as $|x| \to \infty$ in $\R^n_+$, and 
\[
\| u \|_{L^\infty(\R^n_+)} 
\leq \| v \|_{L^\infty(\R^n)} 
\leq C_{s,r} \| v \|_{H^{s,r}(\R^n)}.
\]
Taking the infimum over all such $v$ 
gives the second embedding inequality \eqref{embedding.ineq}. 

\medskip

$(iii)$ \emph{Proof of \eqref{pa.k.est}, \eqref{pa.n.est}}. 
Inequality \eqref{pa.n.est} is trivial on $\R^n$, i.e., with $\R^n_+$ replaced by $\R^n$. 
On $\R^n_+$, let $u \in H^{s+1,r}(\R^n_+)$. 
For any $v \in H^{s+1,r}(\R^n)$ with $v|_{\R^n_+} = u$, 
the distribution $\pa_{x_n} v$ belongs to $H^{s,r}(\R^n)$
and it coincides with the distribution $\pa_{x_n} u$ in $\R^n_+$. 
Hence the set $\{ \pa_{x_n} v : v \in H^{s+1,r}(\R^n)$, $v|_{\R^n_+} = u \}$  
is a subset of $\{ w \in H^{s,r}(\R^n) : w|_{\R^n_+} = \pa_{x_n} u \}$,
and, by definition \eqref{def.H.s.r.Om},
\begin{align*}
\| \pa_{x_n} u \|_{H^{s,r}(\R^n_+)} 
& = \inf \{ \| w \|_{H^{s,r}(\R^n)} : w \in H^{s,r}(\R^n), \,  w|_{\R^n_+} = \pa_{x_n} u \}
\\ 
& \leq \inf \{ \| \pa_{x_n} v \|_{H^{s,r}(\R^n)} : v \in H^{s+1,r}(\R^n), \, v|_{\R^n_+} = u \}  
\\ 
& \leq \inf \{ \| v \|_{H^{s+1,r}(\R^n)} : v \in H^{s+1,r}(\R^n), \, v|_{\R^n_+} = u \}
= \| u \|_{H^{s+1,r}(\R^n_+)}.  
\end{align*} 
The proof of \eqref{pa.k.est} is similar.

\medskip

$(iv)$ \emph{Proof of \eqref{trace.est.R.n},\eqref{trace.est}}. 
\emph{Proof of \eqref{trace.est.R.n}}.
Let $v \in \mS(\R^n)$, 
and let $\hat v \in \mS(\R^n)$ be its Fourier transform. 
For every $x_n \in \R$, the function $v(\cdot, x_n)$ is in $\mS(\R^{n-1})$, 
and its Fourier transform is 
\begin{equation} \label{hat.g}
\mF_{\R^{n-1}} \{ v(\cdot, x_n) \}(\xi') 
= \widehat{v(\cdot, x_n)} (\xi') 
= \int_\R \hat v(\xi) e^{i 2\pi \xi_n x_n} \, d\xi_n 
\quad \forall \xi' \in \R^{n-1}, \ x_n \in \R,
\end{equation}
because, by Fourier inversion formula and Fubini/Tonelli, 
\[
v(x',x_n)
= v(x) 
= \int_{\R^n} \hat v(\xi) e^{i 2\pi \la \xi , x \ra} \, d\xi 
= \int_{\R^{n-1}} \Big( \int_\R \hat v(\xi) e^{i 2\pi \xi_n x_n} \, d\xi_n \Big) 
e^{i 2\pi \la \xi' , x' \ra} \, d \xi'.
\]
By \eqref{hat.g}, 
multiplying and dividing $\hat v(\xi)$ by $(1 + |\xi|^2)^{\frac{s}{2}}$, 
with $s > \frac12$, 
using Cauchy-Schwarz and \eqref{int.on.R.like.arctan},  
for every $\xi' \in \R^{n-1}$ one has 
\begin{align}
|\widehat{v(\cdot,x_n)}(\xi')|
\leq \int_{\R} |\hat v(\xi)| \, d\xi_n 
\leq \Big( \frac{ K_s }{ (1 + |\xi'|^2)^{s-\frac12}} \Big)^{\frac12}  
\Big( \int_\R |\hat v(\xi)|^2 (1 + |\xi|^2)^s \, d\xi_n \Big)^{\frac12}.
\label{trace.Fou.01}
\end{align}
Hence
\begin{equation} \label{trace.Fou.02}
|\widehat{v(\cdot,x_n)}(\xi')|^2 (1 + |\xi'|^2)^{r+s-\frac12} 
\leq K_s \int_\R |\hat v(\xi)|^2 (1 + |\xi|^2)^s (1 + |\xi'|^2)^r \, d\xi_n
\end{equation}
for any $r \in \R$, and, integrating in $d \xi'$ over $\R^{n-1}$, 
\[
\| v(\cdot,x_n) \|_{H^{r+s-\frac12}(\R^{n-1})} 
\leq C_s \| v \|_{H^{s,r}(\R^n)}, \quad \ 
C_s = K_s^{\frac12}.
\]
This holds for all $v \in \mS(\R^n)$, all $x_n \in \R$.  
Since $\mS(\R^n)$ is dense in $H^{s,r}(\R^n)$, 
for every $x_n \in \R$ 
there exists a unique bounded linear operator 
\begin{equation} \label{T.x.n}
T_{x_n} : H^{s,r}(\R^n) \to H^{s+r-\frac12}(\R^{n-1})
\end{equation}
that extends the map $v \to v(\cdot, x_n)$ to $H^{s,r}(\R^n)$
and satisfies the same bound.
For $x_n=0$ we get \eqref{def.trace.H.s.r.R.n}, \eqref{trace.est.R.n}. 
Moreover, one has the following property of uniform continuity in $x_n$.
Let $v \in H^{s,r}(\R^n)$, let $x_n, x_n^* \in \R$, and $\e > 0$. 
Take $w \in \mS(\R^n)$ such that 
$\| v-w \|_{H^{s,r}(\R^n)} < \e / (4 C_s)$. 
By triangular inequality, 
\begin{equation} \label{cont.trace.1}
\| T_{x_n} v - T_{x_n^*} v \| 
\leq \| T_{x_n} (v-w) \| + \| T_{x_n} w - T_{x_n^*} w \| + \| T_{x_n^*} (w-v) \|, 
\end{equation}
where the norms are those of $H^{s+r-\frac12}(\R^{n-1})$. 
By the inequality satisfied by the operators in \eqref{T.x.n}, 
both the first and the third term on the RHS of \eqref{cont.trace.1} are $< \e/4$. 
Since $w$ is a Schwartz function, using \eqref{hat.g}, one proves that 
the second term in the RHS of \eqref{cont.trace.1} 
is $< \e/2$ for $|x_n-x_n^*|$ sufficiently small. 
Thus, for any $v \in H^{s,r}(\R^n)$, any $x_n^* \in \R$, 
$T_{x_n} v$ converges to $T_{x_n^*} v$ 
in $H^{s+r-\frac12}(\R^{n-1})$ as $x_n \to x_n^*$.

\emph{Proof of \eqref{trace.est}}. 
To deal with $\R^n_+$, we first study the validity of identity \eqref{hat.g} outside the Schwartz class. 
Let $v \in H^{s,r}(\R^n)$, and let $(v_k)$ be a sequence in $\mS(\R^n)$ converging to $v$ in $H^{s,r}(\R^n)$. 
Assume that $s, s+r \geq 0$, so that $H^{s,r}(\R^n) \subseteq L^2(\R^n)$ with continuous inclusion. 
On one hand, since $v \in L^2(\R^n)$, by Fubini/Tonelli, 
for a.e.\ $x_n \in \R$, the function $v(\cdot, x_n)$ is defined a.e.\ in $\R^{n-1}$, 
it belongs to $L^2(\R^{n-1})$, and $f_k(x_n) := \| (v-v_k)(\cdot,x_n) \|_{L^2(\R^{n-1})}$ 
is finite for a.e.\ $x_n \in \R$, $f_k$ belongs to $L^2(\R)$, and 
\[
\| f_k \|_{L^2(\R)} = \| v-v_k \|_{L^2(\R^n)} 
\leq \| v-v_k \|_{H^{s,r}(\R^n)} \to 0 \quad (k \to \infty).
\]
Since $f_k \to 0$ in $L^2(\R)$, there exists a subsequence $(f_{k_j})$ 
converging to zero pointwise a.e.\ in $\R$. 
This means that $v_{k_j}(\cdot, x_n)$ converges to $v(\cdot,x_n)$ in $L^2(\R^{n-1})$, for a.e.\ $x_n \in \R$.
Then, for a.e.\ $x_n \in \R$, applying the Fourier transform in $\R^{n-1}$ and Plancherel, 
$\mF_{\R^{n-1}} \{ v_{k_j}(\cdot, x_n) \}$ converges to 
$\mF_{\R^{n-1}} \{ v(\cdot, x_n) \}$ in $L^2(\R^{n-1})$. 
This limit concerns the LHS of \eqref{hat.g}. 
On the other hand, $\hat v \in L^2(\R^n)$ because $v \in L^2(\R^n)$, 
and therefore, by Fubini/Tonelli, for a.e.\ $\xi' \in \R^{n-1}$, 
the function $\hat v(\xi', \cdot)$ is defined a.e.\ in $\R$, 
and it belongs to $L^2(\R)$. 
In addition, for $s > \frac12$, the second inequality in \eqref{trace.Fou.01} holds, 
whence, taking the square, multiplying by $(1+|\xi'|^2)^{r+s-\frac12}$, with $r \geq 0$, 
and integrating in $d \xi'$ over $\R^{n-1}$, we obtain that 
\begin{equation} \label{trace.Fou.03}
\int_{\R^{n-1}} \| \hat v(\xi', \cdot) \|_{L^1(\R)}^2 \, d\xi' 
\leq \int_{\R^{n-1}} (1+|\xi'|^2)^{r+s-\frac12} \| \hat v(\xi', \cdot) \|_{L^1(\R)}^2 \, d\xi' 
\leq K_s \| v \|_{H^{s,r}(\R^n)}^2 
< \infty.
\end{equation}
Hence $\| \hat v(\xi', \cdot) \|_{L^1(\R)}$ is finite for a.e.\ $\xi' \in \R^{n-1}$.  
As a consequence, for a.e.\ $\xi' \in \R^{n-1}$, 
the integral $\ph(\xi', x_n) := \int_{\R} \hat v(\xi', \xi_n) e^{i 2\pi \xi_n x_n} \, d\xi_n$
in the RHS of \eqref{hat.g} is finite for all $x_n \in \R$, it satisfies 
\begin{equation} \label{trace.Fou.05}
|\ph(\xi',x_n)| \leq \| \hat v(\xi', \cdot) \|_{L^1(\R)},
\end{equation} 
and, by dominated convergence, it is continuous in $x_n \in \R$.
Moreover, by \eqref{trace.Fou.03} and \eqref{trace.Fou.05}, 
$\ph(\cdot, x_n)$ belongs to $L^2(\R^{n-1})$ for all $x_n \in \R$, 
and, by dominated convergence, the function $\R \to L^2(\R^{n-1})$, 
$x_n \mapsto \ph(\cdot, x_n)$ is continuous. 
Let $\ph_k$ be defined like $\ph$, but with $\hat v_k$ in place of $\hat v$. 
Repeating this argument with the difference $\hat v_k - \hat v$ in place of $\hat v$, 
we obtain that  
\begin{align} 
\label{trace.Fou.06}
|(\ph_k - \ph)(\xi', x_n)| 
& \leq \|(\hat v_k - \hat v)(\xi', \cdot) \|_{L^1(\R)}, 
\\
\| (\ph_k - \ph)(\cdot, x_n) \|_{L^2(\R^{n-1})}^2
& \leq \int_{\R^{n-1}} \| (\hat v_k - \hat v)(\xi', \cdot) \|_{L^1(\R)}^2 \, d\xi' 
\leq K_s \| v_k - v \|_{H^{s,r}(\R^n)}^2 
\to 0.
\label{trace.Fou.07}
\end{align}
Thus, $\ph_k(\cdot, x_n) \to \ph(\cdot, x_n)$ in $L^2(\R^{n-1})$ for all $x_n$, uniformly in $x_n \in \R$. 
This concerns the RHS in \eqref{hat.g}. 
Now \eqref{hat.g} holds for all $v_k$, and hence, in particular, for all $v_{k_j}$. 
Then, for a.e.\ $x_n \in \R$, 
taking the limit in the $L^2(\R^{n-1})$ norm as $j \to \infty$, 
we obtain that identity \eqref{hat.g} holds in $L^2(\R^{n-1})$, 
i.e., for a.e.\ $x_n \in \R$, 
the LHS and the RHS of \eqref{hat.g} are the same element of $L^2(\R^{n-1})$. 

Moreover, since $v_{k} \to v$ in $H^{s,r}(\R^n)$ 
and the operators in \eqref{T.x.n} are continuous, 
$T_{x_n} v$ is the limit in the $H^{s+r-\frac12}(\R^{n-1})$ norm of $T_{x_n} v_{k}$, 
for all $x_n \in \R$.  
Hence $T_{x_n} v_{k} \to T_{x_n} v$ also in the $L^2(\R^{n-1})$ norm, 
and, by Plancherel, $\mF_{\R^{n-1}}(T_{x_n} v_{k}) \to \mF_{\R^{n-1}}(T_{x_n} v)$ in the same norm. 
On the other hand, the Schwartz functions $v_{k}$ satisfy \eqref{hat.g}, 
whence 
$\mF_{\R^{n-1}}(T_{x_n} v_{k}) 
= \mF_{\R^{n-1}} \{ v_{k}(\cdot, x_n) \} 
= \ph_{k}(\cdot,x_n)$. 
We have already proved that $\ph_k(\cdot, x_n) \to \ph(\cdot, x_n)$ in the $L^2(\R^{n-1})$ norm 
for all $x_n \in \R$. 
Thus, for all $x_n \in \R$, 
$\mF_{\R^{n-1}}(T_{x_n} v)$ and $\ph(\cdot,x_n)$ are the limit in the $L^2(\R^{n-1})$ norm 
of the same sequence, so that they coincides in $L^2(\R^{n-1})$. 
Now, $\mF_{\R^{n-1}}(T_{x_n} v) = \ph(\cdot,x_n)$ in $L^2(\R^{n-1})$ for all $x_n \in \R$,
and, by \eqref{hat.g},  
$\mF_{\R^{n-1}} \{ v(\cdot, x_n) \} = \ph(\cdot,x_n)$ in $L^2(\R^{n-1})$ for a.e.\ $x_n \in \R$. 
Hence, for a.e.\ $x_n \in \R$, 
$\mF_{\R^{n-1}}(T_{x_n} v)$ and $\mF_{\R^{n-1}} \{ v(\cdot, x_n) \}$ coincides in $L^2(\R^{n-1})$. 
As a consequence, for a.e.\ $x_n \in \R$,  
\begin{equation} \label{surprise}
T_{x_n} v = v(\cdot , x_n) 
\end{equation}
in $L^2(\R^{n-1})$, and therefore a.e.\ in $\R^{n-1}$. 
  
Now consider $v \in H^{s,r}(\R^n)$ with $v|_{\R^n_+} = 0$ in the sense of distributions.  
Taking test functions of the form $\psi(x') \eta(x_n)$, we deduce that, for a.e.\ $x_n \in (0,\infty)$,  
$v(\cdot, x_n)$ is zero a.e.\ in $\R^{n-1}$. Then, by \eqref{surprise}, 
$T_{x_n} v = 0$ for a.e.\ $x_n \in (0,\infty)$, and therefore, 
by the property of continuity in $x_n$ proved above, 
$Tv=0$ (recall that $T$ is $T_{x_n}$ with $x_n=0$). 

The last observation allows one to define the trace $Tu$ of a function $u \in H^{s,r}(\R^n_+)$ 
as the trace $Tv$ of any $v \in H^{s,r}(\R^n)$ such that $v|_{\R^n_+} = u$. 
This definition is well posed because, if two functions $v_1, v_2 \in H^{s,r}(\R^n)$ 
satisfy $v_i|_{\R^n_+} = u$, then their difference belongs to $H^{s,r}(\R^n)$ 
and vanishes in $\R^n_+$, therefore its trace also vanishes, and hence $T v_1 = T v_2$. 
Moreover, any $v \in H^{s,r}(\R^n)$ such that $v|_{\R^n_+} = u$ satisfies \eqref{trace.est.R.n}, 
and taking the infimum over all those $v$ gives \eqref{trace.est}.

\medskip

$(vi)$ \emph{Proof of \eqref{reflex.est}}. 
Let $v \in \mS(\R^n)$, and let $f(x) = v(x', |x_n|)$ in $\R^n$. 
One has 
\[
\| f \|_{L^2(\R^n)}^2 
= \int_{\R^n} |f|^2 \, dx 
= 2 \int_{\R^n_+} |v|^2 \, dx 
= 2 \| v \|_{L^2(\R^n_+)}^2 
\leq 2 \| v \|_{L^2(\R^n)}^2.
\]
For any multi-index $\alpha \in \N_0^{n-1}$, 
one has $\pa_{x'}^\alpha f(x) = (\pa_{x'}^\alpha v) (x',|x_n|)$ in $\R^n$, 
and the same basic argument gives  
\[
\| \pa_{x'}^\alpha f \|_{L^2(\R^n)}^2 
\leq 2 \| \pa_{x'}^\alpha v \|_{L^2(\R^n)}^2.
\]
By induction on the dimension $n$, it is not difficult to prove that  
\[
(1 + |\xi'|^2)^m = \sum c_{m,\alpha} (\xi'^\alpha)^2, \quad \ 
c_{m,\alpha} = \frac{m!}{\alpha! (m-|\alpha|)!},
\]
for all $\xi' \in \R^{n-1}$, all $m \in \N$,    
where the sum is over all multi-indices $\alpha \in \N_0^{n-1}$ of length $|\alpha| \leq m$, 
and the usual multi-index notation is used.
Therefore
\[
\| f \|_{H^{0,m}(\R^n)}^2 
= \sum_{|\a| \leq m} c_{m,\a} \int_{\R^n} |\hat f(\xi) \xi'^\a |^2 \, d\xi 
= \sum_{|\a| \leq m} \frac{ c_{m,\a} }{ (2\pi)^{|\alpha|} } \| \pa_{x'}^\alpha f \|_{L^2(\R^n)}^2,
\]
and 
\[
\| f \|_{H^{0,m}(\R^n)}^2 
\leq 2 \| v \|_{H^{0,m}(\R^n)}^2.
\]
For all $m \in \N_0$, by density, the linear map $\mS(\R^n) \to H^{0,m}(\R^n)$, $v \to f$, 
has a unique bounded extension to $H^{0,m}(\R^n)$, with the same bound. 
By complex interpolation of linear bounded operators, 
this also holds for any $r = m + \th$, $\th \in (0,1)$, with the same bound. 
The proof of Riesz-Thorin interpolation theorem can be easily adapted 
to the spaces $H^{0,r}(\R^n)$, using the Schwartz functions $\mS(\R^n)$ as a common dense subset, 
and introducing the complex parameter $z$ in the Fourier inversion formula as 
$v(z)(x) = \int_{\R^n} \hat v(\xi) (1+|\xi'|^2)^{\frac12 (z-\th)} \, e^{i 2\pi \xi \cdot x} \, d\xi$.
Alternatively, one obtains the same conclusion using directly the operator $\Lm'_r$ in \eqref{def.X.s.m}
instead of $\pa_{x'}^\alpha$. 

The partial derivative $\pa_{x_n} f$, as a distribution on $\R^n$, is represented by the function
\[
\pa_{x_n} f = \pa_{x_n} v \quad \text{for } x_n > 0, 
\qquad 
\pa_{x_n} f(x) = - (\pa_{x_n} v)(x',-x_n) \quad \text{for } x_n < 0,
\] 
because, integrating by parts with test functions on $\R^n$, 
the two boundary terms at $x_n = 0$ cancel out. 
Then the argument above applies to $\pa_{x_n} f$ and $\pa_{x_n} v$, 
and therefore, for $s = 0,1$, $r \geq 0$, there exists a unique linear bounded operator $R$ 
of $H^{s,r}(\R^n)$ into itself extending the map $v \to f$, satisfying 
\begin{equation} \label{reflex.est.R.n} 
\| f \|_{H^{s,r}(\R^n)}^2 \leq 2 \| v \|_{H^{s,r}(\R^n)}^2. 
\end{equation} 

Now let $v \in H^{s,r}(\R^n)$ with $v|_{\R^n_+} = 0$. We prove that $Rv = 0$. 
By definition, $Rv$ is the limit of $R v_k$ in $H^{s,r}(\R^n)$,    
where $(v_k)$ is any sequence in $\mS(\R^n)$ converging to $v$ in $H^{s,r}(\R^n)$. 
One has
\[
\| Rv_k \|_{L^2(\R^n)}^2 = 2 \| v_k \|_{L^2(\R^n_+)}^2 
= 2 \| v_k - v \|_{L^2(\R^n_+)}^2 
\leq 2 \| v_k - v \|_{L^2(\R^n)}^2 \to 0,
\]
and therefore $Rv = \lim Rv_k = 0$ in $L^2(\R^n)$.  
This observation allows one to define the extension by reflection $R u$ 
of a function $u \in H^{s,r}(\R^n_+)$ 
as the image $R v$ of any $v \in H^{s,r}(\R^n)$ such that $v|_{\R^n_+} = u$. 
This definition is well posed because, if two functions $v_1, v_2 \in H^{s,r}(\R^n)$ 
satisfy $v_i|_{\R^n_+} = u$, then their difference belongs to $H^{s,r}(\R^n)$ 
and vanishes in $\R^n_+$, therefore its image by $R$ also vanishes, 
and hence $R v_1 = R v_2$. 
Moreover, any $v \in H^{s,r}(\R^n)$ such that $v|_{\R^n_+} = u$ satisfies \eqref{reflex.est.R.n}, 
and taking the infimum over all those $v$ gives \eqref{reflex.est}.

\medskip

$(iv)$ \emph{Proof of \eqref{equiv.norm.H.0.r}, \eqref{equiv.norm.H.1.r}, \eqref{equiv.norm.H.r.zero}}.
\emph{Proof of \eqref{equiv.norm.H.0.r}}.
We use the operators $\Lm'_r$ in \eqref{def.X.s.m}, which we briefly revisit here. 
First, we define $\Lm'_r$ on $\mS(\R^{n-1})$ as the Fourier multiplier 
$\Lm'_r v = \mF_{\R^{n-1}}^{-1} (\lm \mF_{\R^{n-1}} v)$ 
of symbol $\lm(\xi') = (1+|\xi'|^2)^{\frac{r}{2}}$. 
Next, we extend this definition to functions $v \in \mS(\R^n)$ of $n$ real variables by setting 
$(\Lm'_r v)(x',x_n) := \Lm'_r \{ v(\cdot, x_n) \} (x')$ 
(as we do when we define partial derivatives of a smooth function of more than one real variable), 
and, using \eqref{hat.g}, we observe that $\Lm_r'$ on $\mS(\R^n)$ is the Fourier multiplier 
$\Lm_r' v = \mF_{\R^n}^{-1}(\lm \mF_{\R^n} v)$ of the same symbol $\lm(\xi')$ as above. 
The fact that both the Fourier transform in $\R^n$ and that in $\R^{n-1}$ are used 
is related to Remark 1 in Section 4.2.1, page 218, of \cite{Triebel.1983}.  
We observe that $\int_{\R^m} (\Lm_r' v) \psi  
= \int_{\R^m} v (\Lm_r' \psi) $ for all $v, \psi \in \mS(\R^m)$, 
for both $m=n-1$ and $m=n$, 
and therefore we define $\Lm'_r v$ for $v \in \mS'(\R^m)$ as the distribution 
$\la \Lm_r' v, \psi \ra := \la v, \Lm_r' \psi \ra$ for all $\psi \in \mS(\R^m)$.  
The identities $\Lm_r' v = \mF_{\R^m}^{-1}(\lm \mF_{\R^m} v)$, $m=n-1, n$,  
continue to hold in the sense of distributions.
Moreover, for $v \in H^{0,r}(\R^n)$, $r \geq 0$, 
the distribution $\Lm_r' v$ is a function in $L^2(\R^n)$, 
and $(\Lm_r' v)(x) = \Lm_r' \{ v(\cdot, x_n) \} (x')$ 
in the sense of distributions. Hence 
\begin{align}
\| v \|_{H^{0,r}(\R^n)}^2 
= \| \Lambda_r' v \|_{L^2(\R^n)}^2 
& = \int_{\R} \| (\Lambda_r' v)(\cdot , x_n) \|_{L^2(\R^{n-1})}^2 \, dx_n 
\notag \\ 
& = \int_{\R} \| \Lambda_r' \{ v(\cdot , x_n) \} \|_{L^2(\R^{n-1})}^2 \, dx_n 
= \int_\R \| v(\cdot , x_n) \|_{H^r(\R^{n-1})}^2 \, dx_n, 
\label{pinolo.0.r.R.n}
\end{align}
which is identity \eqref{equiv.norm.H.0.r} on $\R^n$. 
On $\R^n_+$, given $u \in H^{0,r}(\R^n_+)$, 
its trivial extension $u_0$ 
defined as $u_0 = u$ in $\R^n_+$ and $u_0 = 0$ in $\R^n \setminus \R^n_+$ 
belongs to $H^{0,r}(\R^n)$ and it satisfies 
\[
\| u_0 \|_{H^{0,r}(\R^n)}^2 
= \int_\R \| u_0(\cdot , x_n) \|_{H^r(\R^{n-1})}^2 \, dx_n 
= \int_0^\infty \| u(\cdot , x_n) \|_{H^r(\R^{n-1})}^2 \, dx_n. 
\]
Hence $\| u_0 \|_{H^{0,r}(\R^n)} \leq \| v \|_{H^{0,r}(\R^n)}$ for any other $v \in H^{0,r}(\R^n)$ 
such that $v|_{\R^n_+} = u$, and therefore $u_0$ realizes the infimum of definition \eqref{def.H.s.r.Om}.
\emph{Proof of \eqref{equiv.norm.H.1.r}}. 
On $\R^n$, one has 
\begin{equation} \label{equiv.norm.H.1.r.R.n} 
\| v \|_{H^{1,r}(\R^n)}^2 
= \| v \|_{H^{0,r+1}(\R^n)}^2 + a \| \pa_{x_n} v \|_{H^{0,r}(\R^n)}^2 
\quad \forall v \in H^{1,r}(\R^n),
\end{equation}
where $a = (2\pi)^{-2}$  
(use Fourier transform and write $|\xi|^2$ as $|\xi'|^2 + \xi_n^2$).  
On $\R^n_+$, let $u \in H^{1,r}(\R^n_+)$, and let $v = Ru$, 
where $R$ is in \eqref{reflex.est}. 
Using \eqref{equiv.norm.H.1.r.R.n}, 
\eqref{pinolo.0.r.R.n}, 
\eqref{equiv.norm.H.0.r},  
\begin{align*}
\| v \|_{H^{1,r}(\R^n)}^2 
& = \int_\R \| v(\cdot, x_n) \|_{H^{r+1}(\R^{n-1})}^2 \, d x_n 
+ a \int_\R \| \pa_{x_n} v(\cdot, x_n) \|_{H^r(\R^{n-1})}^2 \, d x_n 
\notag \\ 
& = 2 \int_0^\infty \| u(\cdot, x_n) \|_{H^{r+1}(\R^{n-1})}^2 \, d x_n 
+ 2 a \int_0^\infty \| \pa_{x_n} u(\cdot, x_n) \|_{H^r(\R^{n-1})}^2 \, d x_n
\\ & = 2 \| u \|_{H^{0,r+1}(\R^n_+)}^2 
+ 2 a \| \pa_{x_n} u \|_{H^{0,r}(\R^n_+)}^2. 
\end{align*}
By definition \eqref{def.H.s.r.Om} and inequality \eqref{reflex.est}, one has  
\[
\| u \|_{H^{1,r}(\R^n_+)}^2 
\leq \| Ru \|_{H^{1,r}(\R^n)}^2
\leq 2 \| u \|_{H^{1,r}(\R^n_+)}^2,
\]
and \eqref{equiv.norm.H.1.r} is proved. 
\emph{Proof of \eqref{equiv.norm.H.r.zero}}. It follows immediately from 
\eqref{equiv.norm.H.0.r}, \eqref{equiv.norm.H.1.r}.

\medskip

$(vi)$ \emph{Proof of \eqref{prod.est}, \eqref{prod.est.below.r.0}}. 
\emph{Proof of \eqref{prod.est}}. 
On $\R^n$, let $u,v \in H^{1,r}(\R^n)$, $r \geq r_0$.
The Fourier transform of the product $uv$ is the convolution 
$\int_{\R^n} \hat u(\xi - \eta) \hat v(\eta) \, d\eta$, 
and 
\begin{equation} \label{int.uv.norm}
\| uv \|_{H^{1,r}(\R^n)}^2 
= \int_{\R^n} \left| \int_{\R^n} \hat u(\xi - \eta) \hat v(\eta) 
(1 + |\xi|^2)^{\frac12} (1+ |\xi'|^2)^{\frac{r}{2}} \, d\eta \right|^2 d\xi.
\end{equation}
Now 
\begin{align}
(1 + |\xi|^2)^{\frac12}
& \leq (1 + |\xi - \eta|^2)^{\frac12} + (1 + |\eta|^2)^{\frac12},
\label{key.basic.est.1} 
\\
(1+ |\xi'|^2)^{\frac{r}{2}} 
& \leq 4 (1+ |\xi'-\eta'|^2)^{\frac{r}{2}} + C_r (1 + |\eta'|^2)^{\frac{r}{2}} 
\quad \forall \xi, \eta \in \R^n.
\label{key.basic.est.r}
\end{align}
To prove \eqref{key.basic.est.1}, consider its square, 
use triangular inequality $|\xi| \leq |\xi-\eta| + |\eta|$ and its square to bound $1 + |\xi|^2$,
and use the fact that $|x| \leq (1 + |x|^2)^{\frac12}$ for $x=\eta$ and $x = \xi-\eta$.  
Inequality \eqref{key.basic.est.r} is also elementary, 
and it is proved in \cite{Baldi.Haus.AIHP.2023}, inequality (B.10) and Lemma B.5, 
where the slightly stronger version 
with $|\eta'|^r$ instead of $(1 + |\eta'|^2)^{\frac{r}{2}}$ is proved, 
and where the constant $C_r$ is written explicitly; 
$C_r$ is a continuous increasing function of $r \in [0, \infty)$, 
with $C_r = 1$ for $r \in [0,1]$ and $C_r \to \infty$ as $r \to \infty$. 
The product of the LHS  of \eqref{key.basic.est.1} and \eqref{key.basic.est.r} 
is bounded by the product of the corresponding RHS, 
which is the sum of 4 products. 
Then the internal integral in \eqref{int.uv.norm} is bounded by the sum of 4 integrals, 
and, since $(a_1 + \ldots + a_4)^2 \leq 4 (a_1^2 + \ldots + a_4^2)$ for any $a_i \in \R$, 
one has 
\[
\eqref{int.uv.norm} \leq 
4 \sum_{i,j = 1,2} \int_{\R^n} \Big( \int_{\R^n} |\hat u(\xi - \eta)| |\hat v(\eta)| \sigma_i \tau_j 
\, d\eta \Big)^2 d\xi, 
\]
where $\sigma_1, \sigma_2$ are the two terms in the RHS  of \eqref{key.basic.est.1}
and $\tau_1, \tau_2$ are those in the RHS  of \eqref{key.basic.est.r}.
In the case $\sigma_1 \tau_1$, multiplying and dividing by $\sigma_2 (1 + |\eta'|^2)^{\frac{r_0}{2}}$, 
and applying H\"older's inequality, we get  
\begin{align*}
\int_{\R^n} \Big( \int_{\R^n} |\hat u(\xi - \eta)| |\hat v(\eta)| \sigma_1 \tau_1 \, d\eta \Big)^2 d\xi
\leq 4^2 C_{r_0}^2 \| u \|_{H^{1,r}(\R^n)}^2 \| v \|_{H^{1,r_0}(\R^n)}^2, 
\end{align*}
where 
\begin{equation} \label{C.r.0.integral}
C_{r_0}^2 := \int_{\R^n} \frac{1}{ (1 + |\eta|^2)(1 + |\eta'|^2)^{r_0} } \, d\eta 
= \int_{\R^{n-1}} \frac{\pi}{ (1 + |\eta'|^2)^{r_0 + \frac12} } \, d\eta 
\end{equation}
is finite for $2 (r_0 + \frac12) > n-1$. 
The second identity in \eqref{C.r.0.integral} holds by \eqref{int.on.R.like.arctan}. 
In the case $\sigma_1 \tau_2$, we multiply and divide by 
$\sigma_2 (1 + |\xi' - \eta'|^2)^{\frac{r_0}{2}}$ and, proceeding similarly,  
we find that the corresponding integral is bounded by 
the square of $C_{r_0} C_r \| u \|_{H^{1,r_0}(\R^n)} \| v \|_{H^{1,r}(\R^n)}$, 
where $C_r$ is the constant in \eqref{key.basic.est.r}, 
and $C_{r_0}$ is in \eqref{C.r.0.integral}. 
In fact, using \eqref{int.on.R.like.arctan} and H\"older's inequality, 
for every $\xi \in \R^n$ one has 
\begin{multline}
\int_{\R^n} \frac{1}{ (1 + |\eta|^2)(1 + |\xi' - \eta'|^2)^{r_0} } \, d\eta 
= \int_{\R^{n-1}} \frac{\pi}{ (1 + |\eta'|^2)^{\frac12} (1 + |\xi' - \eta'|^2)^{r_0} } \, d\eta' 
\\
\leq \pi \Big( \int_{\R^{n-1}} \frac{1}{ (1 + |\eta'|^2)^{\frac{p}{2}} } \, d \eta' \Big)^{\frac{1}{p}} 
\Big( \int_{\R^{n-1}} \frac{1}{ (1 + |\xi' - \eta'|^2)^{r_0 q} } \, d\eta' \Big)^{\frac{1}{q}}
= C_{r_0}^2 
\label{Holder.C.r.0.integral}
\end{multline} 
with $p = 2 r_0 + 1$, $\frac{1}{p} + \frac{1}{q} = 1$.
The cases $\sigma_2 \tau_1$, $\sigma_2 \tau_2$ are analogous. 
This proves \eqref{prod.est} with $\R^n$ instead of $\R^n_+$. 
To obtain \eqref{prod.est}, let $u, v \in H^{1,r}(\R^n_+)$, 
and consider the functions $R u, R v$, where $R$ is the 
extension by reflection in \eqref{reflex.est}. 
The product $(Ru)(Rv)$ satisfies the estimate \eqref{prod.est} on $\R^n$ 
that we have just proved, and, by \eqref{reflex.est}, 
we deduce that the product $uv$ satisfies \eqref{prod.est}. 
\emph{Proof of \eqref{prod.est.below.r.0}}.
To prove \eqref{prod.est.below.r.0}, we proceed similarly, but without exceeding $r$ 
in the powers of $\eta'$: for example, in the case $\sigma_1 \tau_1$, 
we multiply and divide by 
$\sigma_2 (1 + |\eta'|^2)^{\frac{r}{2}} (1 + |\xi'-\eta'|^2)^{\frac{r_0 - r}{2}}$. 
As above, we use \eqref{int.on.R.like.arctan} and H\"older's inequality, 
and we choose H\"older's exponents $p,q$ giving the power $r_0+\frac12$ 
like in \eqref{C.r.0.integral}, \eqref{Holder.C.r.0.integral}. 

\medskip

$(vii)$ \emph{Proof of \eqref{est.prod.infty.0}, \eqref{est.prod.infty.1}}.
\emph{Proof of \eqref{est.prod.infty.0}}. On $\R^n$, one has 
\begin{equation} \label{prod.infty.standard.R.n}
\| g v \|_{H^r(\R^n)} \leq 2 \| g \|_{L^\infty(\R^n)} \| v \|_{H^r(\R^n)} 
+ C_r \| g \|_{W^{b,\infty}(\R^n)} \| v \|_{L^2(\R^n)}
\end{equation}
for all $v \in H^r(\R^n)$, all $g \in W^{b,\infty}(\R^n)$, 
where $r \geq 0$ is real, $b$ is the smallest integer such that $b \geq r$, 
and $C_r$ is an increasing function of $r$. 
Estimate \eqref{prod.infty.standard.R.n} is proved, e.g., in Lemma B.2 of \cite{Baldi.Haus.AIHP.2023}. 
Then \eqref{est.prod.infty.0} follows by using identity \eqref{equiv.norm.H.0.r} 
and applying estimate \eqref{prod.infty.standard.R.n} (with $\R^{n-1}$ instead of $\R^n$) 
to the product $f(\cdot, x_n) u(\cdot, x_n)$ in the integral, 
and H\"older's inequality (or triangular inequality for the $L^2(0,\infty)$ norm). 
\emph{Proof of \eqref{est.prod.infty.1}}. 
Use the first inequality in \eqref{equiv.norm.H.1.r}, 
estimate separately each product $fu$, $(\pa_{x_n} f) u$ and $f(\pa_{x_n} u)$ 
by \eqref{est.prod.infty.0}, then use the second inequality in \eqref{equiv.norm.H.1.r}. 

\medskip

$(viii)$ \emph{Proof of \eqref{composition.est.s.r}}. 
On $\R^n$, given a diffeomorphism $f$ of $\R^n$ of the form $f(x) = a + A x + g(x)$, 
one has 
\begin{equation} \label{comp.est.Rn}
\| u \circ f \|_{H^r(\R^n)} 
\leq C_{r,f} ( \| u \|_{H^r(\R^n)} + \| g \|_{H^{r+1+r_0}(\R^n)} \| u \|_{L^2(\R^n)})
\end{equation} 
for all $u \in H^r(\R^n)$, 
all real $r \geq 0$, 
where $C_{r,f}$ depends on $r,|A|, \| g \|_{H^{1+r_0}(\R^n)}$
and $r_0 > n/2$.
Estimate \eqref{comp.est.Rn} is classical, and it can be proved in this way. 
For $r=0$, use the change of variable $f(x) = y$ in the integral 
giving the $L^2(\R^n)$ norm of $u \circ f$, 
and the $L^\infty(\R^n)$ norm of the Jacobian determinant $\det D(f^{-1})$
of the inverse transformation, 
to get $\| u \circ f \|_{L^2(\R^n)} \leq C_0 \| u \|_{L^2(\R^n)}$ 
for some $C_0$ depending on $f$.  
For $r=1$, we note that 
\begin{equation} \label{comp.temp.01}
\grad (u \circ f)(x) = [Df(x)]^T (\grad u)\circ f(x)
= A^T (\grad u)\circ f(x) + [Dg(x)]^T (\grad u)\circ f(x)
\end{equation}
for all $x \in \R^n$, then we proceed like in the $r=0$ case 
to estimate the $L^2(\R^n)$ norm of \eqref{comp.temp.01}, 
and we get 
$\| u \circ f \|_{H^1(\R^n)} \leq C_1 \| u \|_{H^1(\R^n)}$ 
for some $C_1 \geq C_0$ depending on $f$.  
Then, by the classical interpolation theorems for linear operators, 
$\| u \circ f \|_{H^r(\R^n)} \leq C_1 \| u \|_{H^r(\R^n)}$
for all real $r \in [0,1]$.  
Assume that \eqref{comp.est.Rn} holds for some real $r \geq 0$. 
The norm $\| u \circ f \|_{H^{r+1}(\R^n)}$ is controlled by $\| u \circ f \|_{H^r(\R^n)}
+ \| \grad(u \circ f) \|_{H^r(\R^n)}$, and we use identity \eqref{comp.temp.01} 
and triangular inequality. 
To estimate the product $(Dg)^T (\grad u)\circ f$,
we use the product estimate 
\begin{equation} \label{comp.temp.02}
\| v w \|_{H^r(\R^n)} 
\leq C_{r_0} \| v \|_{H^{r_0}(\R^n)} \| w \|_{H^r(\R^n)} 
+ C_r \| v \|_{H^{r+r_0}(\R^n)} \| w \|_{L^2(\R^n)}
\end{equation}
with $v = (Dg)^T$ and $w = (\grad u) \circ f$. 
Estimate \eqref{comp.temp.02} is proved similarly as \eqref{prod.est}, \eqref{prod.est.below.r.0},
using \eqref{key.basic.est.r} (in fact, the proof of \eqref{comp.temp.02} is simpler
because no coordinate has a special role). 
By the interpolation inequality of $H^r(\R^n)$ spaces, 
the product $\| g \|_{H^{r+1+r_0}} \| u \|_{H^1}$ 
is bounded by $\| g \|_{H^{r+2+r_0}} \| u \|_{H^0}$ 
$+ \| g \|_{H^{r_0+1}} \| u \|_{H^{r+1}}$,
and \eqref{comp.est.Rn} follows by induction. 
On $\R^n_+$, estimate \eqref{composition.est.s.r} is immediately deduced from 
\eqref{equiv.norm.H.0.r},
\eqref{equiv.norm.H.1.r}, 
applying \eqref{comp.est.Rn} with $\R^{n-1}$ instead of $\R^n$; 
for $s=1$, we also use the fact that
$\pa_{x_n} \{ u(f(x'),x_n) \} = (\pa_{x_n} u)(f(x'),x_n)$. 
\end{proof}

\begin{proof}[\textbf{Proof of Lemma \ref{lemma:est.Lap.R.n.+}}] 
This is essentially Theorem 4.2.2 of \cite{Triebel.1983}, page 221, in the case $m=1$, $A = \Delta$, 
$p=q=2$, $\s=0$, $s \in \{0,1\}$, 
adapted to the non-isotropic spaces $H^{s,r}(\R^n_+)$. 
Note that the spaces $H^{s,r}(\R^n_+)$ are already used by Triebel 
in Step 1 of the proof of Theorem 4.2.2. 
The proof of our lemma consists in slightly modifying the proof in \cite{Triebel.1983}, 
using $r \in [0, \infty)$ as the higher regularity parameter, 
and paying attention to the constants in the various inequalities along the proof. 
Also, the two functions $u_0, u_1$ in Step 2 in \cite{Triebel.1983}, page 222, 
corresponding to low/high frequency components of the function $u$, 
must be treated separately regarding their regularity in $x'$;  
in particular, the factor $C^r$ in \eqref{meran.04} comes from the estimate of the term $u_0$. 
\end{proof}

\begin{proof}[\textbf{Proof of Lemma \ref{lemma:est.tools.S.n-1}}]
\emph{Proof of \eqref{interpol.Hr.S2}}. 
Recall the definition \eqref{def.H.s.manifold} of the norm, 
apply the standard interpolation inequality 
$\| f \|_{H^s(\R^{n-1})} 
\leq \| f \|_{H^{s_0}(\R^{n-1})}^{1-\th}
\| f \|_{H^{s_1}(\R^{n-1})}^\th$
of the spaces $H^s(\R^{n-1})$ 
to each function $f = (u \psi_j) \circ \mathtt g_j(\cdot, 0)$, 
then use H\"older's inequality 
\[
\sum_{j=1}^N a_j b_j 
\leq \Big( \sum_{j=1}^N a_j^p \Big)^{\frac{1}{p}} 
\Big( \sum_{j=1}^N b_j^q \Big)^{\frac{1}{q}}
\]
with $p = 1/(1-\th)$, $q = 1/\th$, 
$a_j = \| (u \psi_j) \circ \mathtt g_j \|_{H^{s_0}(\R^{n-1})}^{1-\th}$, 
$b_j = \| (u \psi_j) \circ \mathtt g_j \|_{H^{s_1}(\R^{n-1})}^\th$.

\emph{Proof of \eqref{prod.est.unified.Hr.S2}}.
By \eqref{res.unity.in.the.annulus}, 
\begin{align*}
(u v \psi_j) \circ \mathtt g_j (\cdot, 0) 
= \sum_{\ell \in \mA_j} \tilde u_{\ell j} \, \tilde v_j,
\quad 
\tilde u_{\ell j} := (u \psi_\ell) \circ \mathtt g_j (\cdot, 0), 
\quad 
\tilde v_j := (v \psi_j) \circ \mathtt g_j (\cdot, 0).
\end{align*}
On $\R^{n-1}$, one has the well-known product estimate  
\begin{equation} \label{echin.25}
\| \tilde u_{\ell j} \tilde v_j \|_{H^{s}(\R^{n-1})} 
\leq 
C_{s_0} \| \tilde u_{\ell j} \|_{H^{s_0}(\R^{n-1})} 
\| \tilde v_j \|_{H^{s}(\R^{n-1})} 
+ (\chi_{s \geq s_0}) C_s \| \tilde u_{\ell j} \|_{H^{s}(\R^{n-1})} 
\| \tilde v_j \|_{H^{s_0}(\R^{n-1})},
\end{equation}
which can be proved similarly as \eqref{prod.est.unified} 
(in fact, the proof of \eqref{echin.25} is slightly easier).
Since $\mathtt g_j = \mathtt g_\ell \circ \mathtt f_\ell \circ \mathtt g_j$, one has  
\[
\tilde u_{\ell j} 
= (u \psi_\ell) \circ \mathtt g_\ell \circ \mathtt f_\ell \circ \mathtt g_j( \cdot ,0)
= (u \psi_\ell) \circ \mathtt g_\ell (\tilde{\mathtt T}_{\ell j}( \cdot ) , 0)
= \tilde u_\ell \circ \tilde{\mathtt T}_{\ell j},
\]
where $\tilde u_\ell := (u \psi_\ell) \circ \mathtt g_\ell (\cdot, 0)$
and $\tilde{\mathtt T}_{\ell j}$ is defined in \eqref{rotat.09}. 
By Lemma \ref{lemma:rotat.tilde} and the composition estimate 
\eqref{composition.est.s.r}, one has 
\begin{equation}  \label{echin.26}
\| \tilde u_{\ell j} \|_{H^{s_0}(\R^{n-1})} 
\leq C_{s_0} \| \tilde u_\ell \|_{H^{s_0}(\R^{n-1})},
\quad 
\| \tilde u_{\ell j} \|_{H^{s}(\R^{n-1})} 
\leq C_s \| \tilde u_\ell \|_{H^{s}(\R^{n-1})}. 
\end{equation}
Using \eqref{echin.26} into \eqref{echin.25} and taking the sum over $j, \ell$
gives \eqref{prod.est.unified.Hr.S2}.

\emph{Proof of \eqref{power.est.Hr.S2}}.
For $m=2$, \eqref{power.est.Hr.S2} is \eqref{prod.est.unified.Hr.S2} with $v = u$;
\eqref{power.est.Hr.S2} with $m+1$ instead of $m$ is deduced 
from \eqref{power.est.Hr.S2} by applying \eqref{prod.est.unified.Hr.S2} with $v = u^m$.

\emph{Proof of \eqref{est.grad.Hr.S2}}. Let $u \in H^{s+1}(\S^{n-1})$. 
Denote $u_0(x) := u(x/|x|)$ its $0$-homogeneous extension to $\R^n \setminus \{ 0 \}$, 
so that $\grad_{\S^{n-1}} u = \grad u_0$ on $\S^{n-1}$. 
By \eqref{def.H.s.manifold}, we have to estimate the sum over $j=1, \ldots, N$ of 
$\| ((\grad u_0) \psi_j) \circ \mathtt g_j(\cdot, 0) \|_{H^s(\R^{n-1})}$. 
The function $((\grad u_0) \psi_j) \circ \mathtt g_j(\cdot, 0)$ 
is the result of the composition 
$T ( \mathtt G_j ( \Psi_j ( \grad u_0 )))$, where 
$T \colon f \mapsto f(\cdot , 0)$ is the trace operator in \eqref{def.trace.H.s.r.R.n},
$\mathtt G_j$ is the composition operator $f \mapsto f \circ \mathtt g_j$, 
and $\Psi_j$ is the multiplication operator $f \mapsto f \psi_j$. 
The composition of these operators with the gradient operator satisfies the following identities.
First, one has $\Psi_j (\grad f) = \grad (\Psi_j f) - (\grad \psi_j) f$. 
Second, $\mathtt G_j(\grad f) = (D \mathtt g_j)^{-T} \grad( \mathtt G_j f)$. 
Third, $T (\grad f) = (\grad_{y'} (T f) , T( \pa_{y_n} f ))$. 
Applying these identities, one has that $((\grad u_0) \psi_j) \circ \mathtt g_j(\cdot, 0)$ 
$= S_{1,j} + S_{2,j}$ $+ S_{3,j}$, where 
$S_{1,j}$ is the product of the matrix $T((D \mathtt g_j)^{-T})$ with the vector 
$( \grad_{y'} \{ T (\mathtt G_j (\Psi_j u_0)) \} , 0 )$, 
$S_{2,j}$ is the product of the same matrix $T( (D \mathtt g_j)^{-T})$ with the vector 
$( 0, T ( \pa_{y_n} \{ \mathtt G_j (\Psi_j u_0) \} ) )$, 
and $S_{3,j}$ is $- T( \mathtt G_j \{ (\grad \psi_j) u_0 \} )$. 
By \eqref{prod.infty.standard.R.n}, 
\begin{align*}
\| S_{1,j} \|_{H^s(\R^{n-1})} 
& \leq C_0 \| \grad_{y'} \{ T (\mathtt G_j (\Psi_j u_0)) \} \|_{H^s(\R^{n-1})} 
+ C_s \| \grad_{y'} \{ T (\mathtt G_j (\Psi_j u_0)) \} \|_{L^2(\R^{n-1})} 
\\
& \leq C_0 \| T (\mathtt G_j (\Psi_j u_0)) \|_{H^{s+1}(\R^{n-1})} 
+ C_s \| T (\mathtt G_j (\Psi_j u_0)) \|_{H^1(\R^{n-1})}, 
\end{align*}
and therefore, recalling the definition \eqref{def.norm.u} of the norm on $\S^{n-1}$, 
the sum over $j=1, \ldots, N$ of the terms $S_{1,j}$ is bounded by 
$C_0 \| u \|_{H^{s+1}(\S^{n-1})} + C_s \| u \|_{H^1(\S^{n-1})}$.  
Concerning $S_{2,j}$, we observe that $\pa_{y_n} (\mathtt G_j u_0) = 0$ 
because, by \eqref{explicit.mathtt.f.inv},  
the vector $\pa_{y_n} \mathtt g_j(y)$ is parallel to $\mathtt g_j(y)$, 
and $\la x, \grad u_0(x) \ra = 0$, as $u_0$ is a $0$-homogeneous function. 
Hence $T( \pa_{y_n} \{ \mathtt G_j (\Psi_j u_0) \}) 
= T ( \pa_{y_n} ( \mathtt G_j \psi_j)) T (\mathtt G_j u_0)$. 
We proceed similarly as in the proof of \eqref{prod.est.unified.Hr.S2}: 
we insert $1 = \sum_{\ell \in \mA_j} \psi_\ell$, 
and write $\mathtt g_j = \mathtt g_\ell \circ \mathtt f_\ell \circ \mathtt g_j$, 
so that $\mathtt g_j (\cdot, 0) = \mathtt g_\ell (\tilde{\mathtt T}_{\ell j} (\cdot), 0)$, 
with $\tilde{\mathtt T}_{\ell j}$ defined in \eqref{rotat.09}. 
Then we use the product estimate \eqref{prod.infty.standard.R.n}, 
Lemma \ref{lemma:rotat.tilde} and the composition estimate 
\eqref{composition.est.s.r}, and we obtain that 
the sum over $j=1, \ldots, N$ of the terms $\| S_{2,j} \|_{H^s(\R^{n-1})}$ 
is bounded by the sum over $j, \ell$ with $\ell \in \mA_j$ 
of $C_s \| T ( \mathtt G_\ell ( \Psi_\ell u_0) ) \|_{H^s(\R^{n-1})}$.
The terms $S_{3,j}$ are estimated similarly as $S_{2,j}$. 
Thus we obtain 
\[
\| \grad_{\S^{n-1}} u \|_{H^s(\S^{n-1})} 
\leq C_0 \| u \|_{H^{s+1}(\S^{n-1})} + C_s \| u \|_{H^1(\S^{n-1})}
+ C_s \| u \|_{H^s(\S^{n-1})}.
\]
To absorb the last term, we use \eqref{interpol.Hr.S2} with $s_0 = 0$, $s_1 = s+1$, 
and proceed like in \eqref{interpol.trick}.

\emph{Proof of \eqref{embedding.Sd}}. 
On $\R^{n-1}$ one has the embedding inequality 
$\| f \|_{L^\infty(\R^{n-1})} \leq C_{s_0} \| f \|_{H^{s_0}(\R^{n-1})}$, 
and every $f \in H^{s_0}(\R^{n-1})$ is a continuous function 
(see the proof of \eqref{embedding.ineq}).
On $\S^{n-1}$, by \eqref{res.unity.in.the.annulus}, one has $u = \sum_{j=1}^N u \psi_j$, 
and 
\begin{align*}
\| u \|_{L^\infty(\S^{n-1})} 
\leq \sum_{j=1}^N \| u \psi_j \|_{L^\infty(\S^{n-1})}
& = \sum_{j=1}^N \| (u \psi_j) \circ \mathtt g_j(\cdot,0) \|_{L^\infty(\R^{n-1})}
\\ 
& \leq C_{s_0} \sum_{j=1}^N \| (u \psi_j) \circ \mathtt g_j(\cdot,0) \|_{H^{s_0}(\R^{n-1})}
= C_{s_0} \| u \|_{H^{s_0}(\S^{n-1})},
\end{align*}
because $K_j \cap \S^{n-1} = \{ \mathtt g_j(y',0) : (y',0) \in A_0 \}$ 
and $\psi_j = 0$ outside $K_j$, 
$\psi_j \circ \mathtt g_j = 0$ outside $A_0$.  
Moreover $(u \psi_j) \circ \mathtt{g}_j(\cdot, 0)$ is continuous for all $j$, 
whence we deduce that $u(x) \to u(x_0)$ as $x \to x_0$. 
\end{proof}

\bigskip

\emph{Statements and Declarations.} 
The authors state that there is no conflict of interest. 
No data was used for the research described in the article.

\bigskip

\bigskip

\begin{flushright}

\textbf{Pietro Baldi}

Dipartimento di Matematica e Applicazioni ``R. Caccioppoli''

University of Naples Federico II

Via Cintia, Monte Sant'Angelo, 80126 Naples, Italy

pietro.baldi@unina.it

ORCID 0000-0002-9644-3935

\medskip

\textbf{Vesa Julin} 

Department of Mathematics and Statistics

University of Jyv\"askyl\"a

P.O. Box 35, 40014 Jyv\"askyl\"a, Finland

vesa.julin@jyu.fi

ORCID 0000-0002-1310-4904

\medskip

\textbf{Domenico Angelo La Manna}

Dipartimento di Matematica e Applicazioni ``R. Caccioppoli''

University of Naples Federico II

Via Cintia, Monte Sant'Angelo, 80126 Naples, Italy

domenicoangelo.lamanna@unina.it

ORCID 0000-0003-1900-2025

\end{flushright}


\begin{thebibliography}{99}

\bibitem{Alazard.Burq.Zuily1}
T.\ Alazard, N.\ Burq, C.\ Zuily, 
\emph{Cauchy theory for the water waves system in an analytic framework}, 
Tokyo J.\ Math. 45 (2022), no.1, 103-199.

\bibitem{Alazard.Delort}
T.\ Alazard, J.-M.\ Delort, 
\emph{Sobolev estimates for two dimensional gravity water waves}, 
Asterisque, 374 (2015), viii + 241.

\bibitem{Alazard.Metivier}
T.\ Alazard, G.\ Metivier, 
\emph{Paralinearization of the Dirichlet to Neumann operator, 
and regularity of the three dimensional water waves}, 
Comm.\ Partial Differential Equations 34 (2009), no.\,10-12, 1632-1704.

\bibitem{Baldi.Haus.AIHP.2023}
P.\ Baldi, E.\ Haus, 
\emph{Size of data in implicit function problems 
and singular perturbations for nonlinear Schr\"odinger systems},
Ann.\ Inst.\ H.\ Poincar\'e Anal.\ Non Lin\'eaire 40 (2023), 1051-1091.

\bibitem{BJL} 
P.\ Baldi, V.\ Julin, D.A.\ La Manna, 
\emph{Liquid Drop with Capillarity and Rotating Traveling Waves}, 
Arch.\ Rational Mech.\ Anal.\ 250 (2026), 4.

\bibitem{Berti.Maspero.Ventura} 
M.\ Berti, A.\ Maspero, P.\ Ventura,
\emph{On the analyticity of the Dirichlet-Neumann operator and Stokes waves}, 
Atti Accad.\ Naz.\ Lincei Cl.\ Sci.\ Fis.\ Mat.\ Natur.\ 33 (2022), no.\,3, 611-650.

\bibitem{Beyer.Gunther} 
K.\ Beyer, M.\ G\"unther,
\emph{On the Cauchy Problem for a Capillary Drop. Part I: Irrotational Motions},
Math.\ Meth.\ Appl.\ Sci.\ 21 (1998), 1149-1183.

\bibitem{Cagnetti.Mora.Morini}
F.\ Cagnetti, M.G.\ Mora, M.\ Morini, 
\emph{A second order minimality condition for the Mumford-Shah functional},  
Calc.\ Var.\ 33 (2008), 37-74.

\bibitem{Calderon}
A.P.\ Calderon, 
\emph{Cauchy integrals on Lipschitz curves and related operators},
Proc.\ Nat.\ Acad.\ Sci.\ USA 75 (1977), 1324-1327.

\bibitem{Coifman.Meyer}
R.\ Coifman, Y.\ Meyer, 
\emph{Nonlinear harmonic analysis and analytic dependence}, 
Proc.\ Symp.\ Pure Math., 43 (1985), 71-78.

\bibitem{Coutand.Shkoller} 
D.\ Coutand, S.\ Shkoller, 
\emph{Well-posedness of the free-surface incompressible Euler equations with or without surface tension}, 
J.\ Amer.\ Math.\ Soc.\ 20 (2007), 829–930.

\bibitem{Craig.Nicholls}
W.\ Craig, D.\ Nicholls, 
\emph{Travelling two and three dimensional capillary gravity water waves}, 
SIAM J.\ Math.\ Anal.\ 32 (2000), no.\,2, 323-359.

\bibitem{Craig.Schanz.Sulem}
W.\ Craig, U.\ Schanz, C.\ Sulem, 
\emph{The modulational regime of three-dimensional water waves and the Davey-Stewartson system}, 
Ann.\ Inst.\ H.\ Poincar\'e Anal.\ Non Lin\'eaire 14 (1997), no.\,5, 615-667.

\bibitem{Feola.Montalto.Terracina}
R.\ Feola, R.\ Montalto, S.\ Terracina, 
\emph{Time quasi-periodic three-dimensional traveling gravity water waves},
preprint, arxiv:2509.10318.

\bibitem{Garcia.Hassainia.Roulley}
C.\ Garcia, Z.\ Hassainia, E.\ Roulley,
\emph{Dynamics of vortex cap solutions on the rotating unit sphere},
J.\ Differential Equations 417 (2005), 1-63.

\bibitem{Groves.Nilsson.Pasquali.Wahlen} 
M.\ Groves, D.\ Nilsson, S.\ Pasquali, E.\ Wahl\'en, 
\emph{Analytical Study of a generalised Dirichlet-Neumann operator 
and application to three-dimensional water waves on Beltrami flows}, 
J.\ Differential Equations 413 (2024), 129–189.

\bibitem{Henrot.Pierre} 
A.\ Henrot, M.\ Pierre, 
\emph{Shape Variation and Optimization: A Geometrical Analysis}.
EMS Tracts in Mathematics, 28, Europ.\ Math.\ Soc., 2018.

\bibitem{Huang.Karakhanyan.jets}
Y.\ Huang, A.\ Karakhanyan, 
\emph{The well-posedness of cylindrical jets with surface tension}, 
preprint, arxiv:2311.13748.

\bibitem{Huang.Karakhanyan.cone}
Y.\ Huang, A.\ Karakhanyan, 
\emph{The Dirichlet-Neumann operator for Taylor’s Cone}, 
preprint, arxiv:2402.07005.

\bibitem{LaScala}
G.\ La Scala, 
\emph{Two-dimensional capillary liquid drop: Craig-Sulem formulation on $\T^1$ 
and bifurcations from multiple eigenvalues of rotating waves}, 
preprint, arxiv:2505.11650.

\bibitem{Lannes.book}
D.\ Lannes, 
\emph{The Water Waves Problem: Mathematical Analysis and Asymptotics}, 
Math.\ Surveys and Monographs 188, Amer.\ Math.\ Soc., 2013.

\bibitem{Lions.Magenes.volume.1}
J.L.\ Lions, E.\ Magenes,
\emph{Non-homogeneous boundary value problems and applications}, 
Die Grundlehren d. math. Wissenschaften. Band 183, 1972.

\bibitem{Moon.Wu}
G.\ Moon, Y.\ Wu.
\emph{Global Bifurcation of Steady Surface Capillary Waves on a 2D Droplet}, 
Nonlinearity 39 (2026), 015005. 

\bibitem{Pasquali}
S.\ Pasquali, 
\emph{Analytical study of a generalized Dirichlet-Neumann operator 
for three-dimensional water waves with vorticity},
preprint, arxiv:2502.09370.

\bibitem{Shao.On.the.Cauchy}
C.\ Shao,
\emph {On the Cauchy Problem of Spherical Capillary Water Waves}, 
Forum Mathematicum 38 (2026), no.\,3, 727-769.

\bibitem{Shao.Toolbox}
C.\ Shao,
\emph{Toolbox of para-differential calculus on compact Lie groups},
J.\ Funct.\ Anal. 290 (2026), no.\,1, 111176.

\bibitem{Taylor.volume.1}
M.E.\ Taylor,
\emph{Partial Differential Equations, Volume I. Basic Theory}, 
Appl.\ Math.\ Sci.\ 115, Springer, 1996. 

\bibitem{Taylor.volume.3}
M.E.\ Taylor,
\emph{Partial Differential Equations, Volume III. Nonlinear Equations}, 
Appl.\ Math.\ Sci.\ 117, Springer, 1996. 

\bibitem{Triebel.1983} 
H.\ Triebel, 
\emph{Theory of Function Spaces}, 
Birkh\"auser, Springer Basel AG, 2010, reprint of the 1983 edition.

\bibitem{Yang}
H.\ Yang, 
\emph{Cauchy Problem for Cylinder-like Capillary Jets},
preprint, arxiv:2410.12506.

\end{thebibliography}
\end{document}